\documentclass[letterpaper]{article} 
\usepackage{aaai24}  
\usepackage{times}  
\usepackage{helvet}  
\usepackage{courier}  
\usepackage[hyphens]{url}  
\usepackage{graphicx} 
\usepackage{natbib}  
\usepackage{caption} 
\usepackage{algorithm}
\usepackage{algorithmic}

\usepackage{newfloat}
\usepackage{listings}
\DeclareCaptionStyle{ruled}{labelfont=normalfont,labelsep=colon,strut=off} 
\floatstyle{ruled}
\newfloat{listing}{tb}{lst}{}
\floatname{listing}{Listing}
\title{Jointly Improving the Sample and Communication Complexities in Decentralized Stochastic Minimax Optimization}
\author{
   Xuan Zhang\textsuperscript{\rm 1},
    Gabriel Mancino-Ball\textsuperscript{\rm 2},
    Necdet Serhat Aybat\textsuperscript{\rm 1},
    Yangyang Xu\textsuperscript{\rm 2}
}
\affiliations{
    \textsuperscript{\rm 1}Department of
Industrial and Manufacturing Engineering, The Pennsylvania State University,
University Park, PA \\
    \textsuperscript{\rm 2}Department of Mathematical Sciences, Rensselaer Polytechnic Institute, Troy, NY 12180\\

        xxz358@psu.edu, gabriel.mancino.ball@gmail.com, nsa10@psu.edu, 
        xuy21@rpi.edu
}

\RequirePackage{algorithm}
\RequirePackage{algorithmic}

\usepackage{amsmath}
\usepackage{amssymb}
\usepackage{mathtools}
\usepackage{amsthm}

\usepackage[capitalize,noabbrev]{cleveref}
\crefname{section}{Section}{Sections} 
\theoremstyle{plain}
\newtheorem{theorem}{Theorem}

\newtheorem{lemma}{Lemma}

\theoremstyle{definition}
\newtheorem{definition}{Definition}
\newtheorem{assumption}{Assumption}
\theoremstyle{remark}
\newtheorem{remark}{Remark}
\theoremstyle{definition}
\newtheorem{param}{Parameter Condition}

\def\grad{\nabla}

\def\ba{\mathbf{a}}
\def\bb{\mathbf{b}}

\def\bd{\mathbf{d}}
\def\be{\mathbf{e}}

\def\bp{\mathbf{p}}
\def\bq{\mathbf{q}}
\def\br{\mathbf{r}}

\def\bv{\mathbf{v}}

\def\bx{\mathbf{x}}  
\def\by{\mathbf{y}}
\def\bz{\mathbf{z}}

\def\bE{\mathbf{E}}

\def\bI{\mathbf{I}}

\def\bP{\mathbf{P}}
\def\bQ{\mathbf{Q}}
\def\bR{\mathbf{R}}

\def\bY{\mathbf{Y}}
\def\bX{\mathbf{X}}

\def\cB{\mathcal{B}}
\def\cC{\mathcal{C}}

\def\cE{\mathcal{E}}

\def\cG{\mathcal{G}}

\def\cO{\mathcal{O}}

\def\cV{\mathcal{V}}

\def\mR{\mathbb{R}}

\def\smskip{\smallskip}

\def\texitem#1{\par\smskip\noindent\hangindent 25pt
               \hbox to 25pt {\hss #1 ~}\ignorespaces}

\def\norm#1{\|#1\|}

\newcommand{\BEAS}{\begin{eqnarray*}}
\newcommand{\EEAS}{\end{eqnarray*}}
\newcommand{\BEA}{\begin{eqnarray}}
\newcommand{\EEA}{\end{eqnarray}}
\newcommand{\BEQ}{\begin{eqnarray}}
\newcommand{\EEQ}{\end{eqnarray}}
\newcommand{\BIT}{\begin{itemize}}
\newcommand{\EIT}{\end{itemize}}
\newcommand{\BNUM}{\begin{enumerate}}
\newcommand{\ENUM}{\end{enumerate}}

\newcommand{\BA}{\begin{array}}
\newcommand{\EA}{\end{array}}

\newcommand{\reals}{\mathbb{R}}

\newif\ifpagenumbering
\pagenumberingtrue

\pagenumberingfalse

\newsavebox{\theorembox}
\newsavebox{\lemmabox}
\newsavebox{\defnbox}
\newsavebox{\assbox}
\savebox{\theorembox}{\noindent\bf Theorem}
\savebox{\lemmabox}{\noindent\bf Lemma}
\savebox{\defnbox}{\noindent Definition}

\usepackage{soul}
\DeclareUnicodeCharacter{2212}{~}

\newcommand{\blx}{\mathbf{x}}
\newcommand{\bly}{\mathbf{y}}

\def\bbx{\Bar{\mathbf{x}}}
\def\bby{\Bar{\mathbf{y}}}
\def\bbd{\Bar{\mathbf{d}}}
\def\bbv{\Bar{\mathbf{v}}}

\def\xz#1{\textcolor{blue}{#1}}

\def\yy#1{\textcolor{black}{#1}}
\def\gmbf#1{\textcolor{black}{#1}}
\def\sa#1{\textcolor{red}{#1}}
\def\gda{\texttt{DGDA-VR}}
\def\xz#1{\textcolor{black}{#1}}

\def\yy#1{\textcolor{black}{#1}}
\def\gmb#1{\textcolor{black}{#1}}
\def\sa#1{\textcolor{black}{#1}}
\def\xzrr#1{\textcolor{blue}{#1}}
\def\xzy#1{\textcolor{purple}{#1}}
\def\xzf#1{\textcolor{black}{#1}}
\def\xz#1{\textcolor{black}{#1}}

\def\yy#1{\textcolor{black}{#1}}
\def\gmb#1{\textcolor{black}{#1}}
\def\sa#1{\textcolor{black}{#1}}
\def\gda{\texttt{DGDA-VR}}
\def\xz#1{\textcolor{black}{#1}}

\def\yy#1{\textcolor{black}{#1}}
\def\gmb#1{\textcolor{black}{#1}}
\def\sa#1{\textcolor{black}{#1}}
\def\xzrr#1{\textcolor{black}{#1}}
\def\xzy#1{\textcolor{black}{#1}}
\def\xzref#1{\textcolor{black}{#1}}

\def\saa#1{\textcolor{black}{#1}}
\def\yx#1{\textcolor{black}{#1}}
\DeclareMathOperator*{\argmax}{argmax}

\def\tgrad{\tilde{\nabla}}
\def\grad{\nabla}

\def\ba{\mathbf{a}}
\def\bb{\mathbf{b}}

\def\bd{\mathbf{d}}
\def\be{\mathbf{e}}

\def\bp{\mathbf{p}}
\def\bq{\mathbf{q}}
\def\br{\mathbf{r}}

\def\bv{\mathbf{v}}

\def\bx{\mathbf{x}}  
\def\by{\mathbf{y}}
\def\bz{\mathbf{z}}

\def\bE{\mathbf{E}}

\def\bI{\mathbf{I}}

\def\bP{\mathbf{P}}
\def\bQ{\mathbf{Q}}
\def\bR{\mathbf{R}}

\def\bY{\mathbf{Y}}
\def\bX{\mathbf{X}}

\def\cB{\mathcal{B}}
\def\cC{\mathcal{C}}

\def\cE{\mathcal{E}}

\def\cG{\mathcal{G}}

\def\cO{\mathcal{O}}

\def\cV{\mathcal{V}}

\def\mR{\mathbb{R}}

\def\smskip{\smallskip}

\def\texitem#1{\par\smskip\noindent\hangindent 25pt
               \hbox to 25pt {\hss #1 ~}\ignorespaces}

\newif\ifpagenumbering
\pagenumberingtrue
\pagenumberingfalse

\savebox{\theorembox}{\noindent\bf Theorem}
\savebox{\lemmabox}{\noindent\bf Lemma}
\savebox{\defnbox}{\noindent Definition}

\usepackage[utf8]{inputenc}
\usepackage{textgreek}
\usepackage{booktabs}
\usepackage{enumitem}
\usepackage{comment}
\usepackage{pifont}

\newcommand{\xmark}{\ding{55}}
\newcommand{\bigO}[1]{\mathcal{O}\left(#1\right)}
\usepackage[textsize=tiny]{todonotes}

\begin{document}

\maketitle

\begin{abstract}
 We propose a novel single-loop decentralized algorithm, \texttt{DGDA-VR}, for solving the stochastic nonconvex strongly-concave minimax problems over a connected network of 
 agents, \saa{which are equipped with  stochastic first-order oracles to estimate their local gradients. 
 \xzf{\texttt{DGDA-VR}, \saa{incorporating variance reduction,} achieves} 
 } $\mathcal{O}(\epsilon^{-3})$ 
 \xzf{oracle complexity} and 
 $\mathcal{O}(\epsilon^{-2})$ 
 communication complexity \saa{\textit{without} resorting to multi-communication rounds -- both 
 are \textit{optimal}, i.e., matching} the lower bounds for this class of problems. 
 \saa{Since 
 \xzf{\texttt{DGDA-VR} }
 does not require multiple communication rounds,
 it is applicable to a broader range of decentralized computational environments.} 
 To the best of our knowledge, \saa{this is the first distributed method using a single communication round in each iteration} 
 to jointly optimize the 
 \saa{oracle} and communication complexities for the problem considered here.
\end{abstract}

\section{Introduction}
This paper considers a connected network $\mathcal{G}=(\mathcal{V},\mathcal{E})$ of $M$ agents which cooperatively solve 
\vspace{-2.5mm}
\begin{equation}
    \label{eq:main-problem}
\small{\min_{\bx\in\mR^{\sa{d}}}\max_{\by\in\mR^m}f(\bx,\by)\triangleq  \frac{1}{M}\sum_{i=1}^M f_i(\bx,\by),}
\vspace{-2mm}
\end{equation}
where 
$f_i:\xz{\reals^d\times\reals^m}\rightarrow \mR$ is \sa{smooth and possibly \emph{nonconvex}} in $\bx\in\mR^{d}$ and \sa{are} \gmb{strongly-}concave in $\by\in\mR^m$ for $i=1,2,...,M$. \sa{Furthermore, each agent-$i$ can only access 
unbiased stochastic gradients $\tilde\grad f_i$ rather than exact gradients $\grad f_i$, and we assume that $\{\tilde\grad f_i\}_{i\in\cV}$ have finite variances, uniformly bounded by some $\sigma>0$.} \gmb{The set $\mathcal{V}=\{1,2,\dots,M\}$ indexes the $M$ agents and $(i,j)\in\mathcal{E}\subseteq\mathcal{V}\times\mathcal{V}$ only if agent $i$ can send information to agent $j$.} Minimax optimization has garnered recent interest due to applications in many machine learning settings such as adversarial training~\cite{Goodfellow2014,Mingrui2020}, distributionally robust optimization~\cite{Namkoong2016,Wenhan2021}, reinforcement learning~\cite{Xin2021}, and fair machine learning~\cite{Nouiehed2019}. 
The problem \sa{in}~\eqref{eq:main-problem} arises naturally when the data is 
\sa{physically} distributed among many agents or is too large to store on 
\sa{a single} computing device~\cite{Xin2021HSGD}. It is well known that \emph{centralized} methods suffer from communication bottlenecks on the parameter server~\cite{Lian2017,Wenhan2021} and potential data privacy violations~\cite{Verbraeken2020}; hence, decentralized methods have emerged as a practical alternative to overcome these issues.

In a decentralized setting, \saa{only agent-$i$ has access to} $f_i$ and its stochastic gradient oracle; 
thus, in order for the $M$ agents to collaboratively solve \eqref{eq:main-problem}, each agent-$i$ will make a local copy, denoted as $(\bx_i, \by_i)$, of the primal-dual variable $(\bx,\by)$ and communicate the local variables and gradient information with its \saa{immediate} (1-hop) neighbors. In this way, \eqref{eq:main-problem} can be reformulated equivalently into the following problem in a decentralized format:
\vspace*{-1mm}
\begin{equation}\label{eq:main-problem-2}
\small{
        \begin{split}
        &\min_{\{\blx_i\}_{\sa{i\in\cV}}}\max_{\{\bly_i\}_{\sa{i\in\cV}}}\frac{1}{\sa{|\cV|}}\sum_{\sa{i\in\cV}} f_i(\bx_i,\by_i)\\
        \text{ s.t. }&\bx_i=\bx_j,\quad \by_i=\by_j,\enskip\forall\,(i,j)\in\mathcal{E}.
        \end{split}
        }
\end{equation}
\vspace*{-2.5mm}\\
Consensus among the agents is then enforced through the use of a \emph{mixing matrix} \sa{encoding the topology of $\cG$}. 

\sa{Although there are 
decentralized algorithms for stochastic nonconvex-strongly-concave minimax problems, the existing work~\cite{Mingrui2020,Chen2022} requires \textit{multi-communication rounds} at each iteration; hence, they can be analyzed as a centralized algorithm with inexact gradients. In \gmb{a} multi-agent setting, \gmb methods requiring multi-communication rounds per iteration are \emph{not} desired as they require more strict coordination among the agents while single round communication methods are much easier to implement.} We will design a decentralized algorithm for \eqref{eq:main-problem} or equivalently \eqref{eq:main-problem-2} that only requires a single communication round per iteration.
Although another recent work~\cite{Wenhan2021} 
\saa{also proposed a decentralized algorithm for the same setting} with a single round of communication per iteration, we noticed that its proof has a fundamental issue and the claimed complexity results do not hold --- we explain this problem in detail when we compare our results with the existing work \saa{below}.
In addition, the communication complexity 
of the algorithm in \cite{Wenhan2021} 
\saa{is intrinsically 
of the same order 
with its \xzf{oracle complexity}; 
hence, it} cannot be optimal. In contrast, the 
method we 
\saa{propose} can achieve an optimal complexity result for both \xzf{oracle complexity} and communication complexity in terms of 
\saa{its} dependence on a given tolerance \saa{$\epsilon>0$ for $\epsilon$-stationarity, defined below.}

\paragraph{Contributions.} Our contributions are two-fold. First, 
we propose a decentralized 
stochastic gradient-type method, called \gda{}, for solving \eqref{eq:main-problem} or equivalently \eqref{eq:main-problem-2}. At every iteration of the method, each agent-$i$ performs a local stochastic gradient descent step for $\bx_i$ and a local stochastic gradient ascent step for $\by_i$, along a tracked (global) stochastic gradient direction. \gda{} needs only a single communication round per iteration 
\saa{among neighbors for} (weighted) averaging local variables and \saa{tracking the global} stochastic gradient information.

Second, we show that when each agent uses a SPIDER-type stochastic gradient estimator \cite{fang2018spider}, which is a variant of SARAH \cite{nguyen2017sarah}, \gda{} can, in a decentralized manner, generate $\{\bz_i(\epsilon)\}_{i\in\cV}$ with $\bz_i(\epsilon)\triangleq\big(\bx_i(\epsilon),\by_i(\epsilon)\big)$ such that the local decisions $\{\bz_i(\epsilon)\}_{i\in\cV}$ and their average $(\bbx_\epsilon,\bby_\epsilon)=\bar{\bz}_\epsilon=\frac{1}{|\cV|}\sum_{i\in\cV}\bz_i(\epsilon)$ have the following properties:
\begin{enumerate}[ topsep=-0.5ex, itemsep=-0.5ex]
\item $\bbx_\epsilon$ is an \emph{$\epsilon$-stationary point} of the primal function $\Phi(\cdot)\triangleq\max_{\by}f(\cdot,\by)$, i.e., $\mathbf{E}[\norm{\grad \Phi(\bbx_\epsilon)}]\leq \epsilon$;
\item $\bby_\epsilon$ is an \emph{$\cO(\epsilon)$-optimal-response} to $\bbx_\epsilon$, i.e., $\mathbf{E}[\norm{\bby_\epsilon-\by^*(\bbx_\epsilon)}] = \cO(\epsilon)$, where $\by^*(\bbx_\epsilon)=\argmax_{\by}f(\bbx_\epsilon,\by)$; 
\item $\{\bz_i(\epsilon)\}_{i\in\cV}$ has \emph{$\cO(\epsilon)$-consensus-violation}, i.e., $\mathbf{E}[\sum_{i\in\cV}\norm{\bz_i(\epsilon)-\bar{\bz}_\epsilon}^2] = \cO(\epsilon^2)$;
\item 
\saa{computing $\{\bz_i(\epsilon)\}_{i\in\cV}$} requires $\cO((1-\rho)^{-2}\epsilon^{-2})$ communication among neighboring nodes, which employ $\cO(\sigma(1-\rho)^{-2}\epsilon^{-3})$ stochastic oracle calls, i.e., the sampling complexity --- here, \gmb{$\rho\in[0,1)$} 
measures the connectivity of the underlying communication network, and a smaller $\rho$ means \gmb{a} more connected network. 
The orders $\cO(\epsilon^{-2})$ for communication rounds and $\cO(\epsilon^{-3})$ for stochastic \yx{gradient} oracles both match with existing lower bounds \cite{sun2019distributed, arjevani2022lower}.
\end{enumerate}


\paragraph{Notation and definitions.}
Throughout the paper, \xz{we use bold lower-case letters $\bx, \by, \ldots$ to denote vectors and upper-case letters $X, Y, \ldots$  to denote matrices.} $\|\cdot\|$ denotes the \emph{Euclidean norm} for a vector. $\|\cdot\|_F$ and $\|\cdot\|_2$ denote the \emph{Frobenius norm}, and the \emph{spectral norm} of a matrix, respectively. The symbols $\mathbf{I}$ and $\mathbf{1}$ denote the identity matrix and the column vector with all elements $1$, respectively. The symbol $\bE$ is used for expectation. $W$ represents a mixing matrix and $\Pi\triangleq\frac{1}{M}\mathbf{1}\mathbf{1}^\top\in\reals^{M\times M}$ the averaging matrix.
We let $\mathbb{N}^+\triangleq \mathbb{N}/\{0\}$. Given $M\in\mathbb{N}^+$, $[M]$ denotes the integer set $\{1,2,..,M\}$. 
Given a random variable $\xi$, \sa{for any $i\in [M]$, $\tilde{\grad}f_i(\bx,\by;\xi)$ denotes an unbiased estimator of $\grad f_i(\bx,\by)$, of which properties are stated in~Assumptions~\ref{ASPT:general-sto-grad} and 
 \ref{ASPT:smooth-F-mean-squared}.} \gmb{We interchangeably use $\reals^{d}\times\reals^m=\reals^{d+m}$ when it is convenient to define the inputs to $f_i$ as a single vector.}
 We will compactly use 
 matrix variables \saa{for the formulation in \eqref{eq:main-problem-2}:}
 \vspace{-1mm}
\begin{equation*}
    \small{X \triangleq [\bx_1, \ldots, \bx_M]^\top, \ Y \triangleq [\by_1, \ldots, \by_M]^\top, \ Z \triangleq [X, Y].}
\end{equation*} 
\paragraph{Organization.}
We first briefly discuss the previous work on decentralized minimax problems related to ours. After we give some important definitions and state our assumptions, we describe our proposed method and main results in detail. Finally, we test our method against the SOTA methods employing variance reduction on \yx{a game problem and two \saa{different} robust machine learning problems.}
\section{Related Work}\label{sec:related-works}
\begin{table*}[h]
  \centering
    {\small
    \renewcommand{\arraystretch}{1.6}
  \begin{tabular}[b]{lccccc}
	\toprule
	Method & P & U & 
 \xzf{Oracle} Comp. & Comm. Comp. & Requirement\\
	\midrule
    GT-DA~\cite{Tsaknakis2020}$^{\dagger}$ & FS & D & $\tilde{\mathcal{O}}\left(\frac{n\kappa^{a_s}}{(1-\rho)^{b_s}\varepsilon^{2}}\right)$ & $\tilde{\mathcal{O}}\left(\frac{\kappa^{a_c}}{(1-\rho)^{b_c}\varepsilon^{2}}\right)$ &  mult. $\mathbf{y}$-update\\
    GT-SRVR~\cite{Xin2021} & FS & S & $\bigO{n+\frac{\sqrt{n}\kappa^{c_s}}{(1-\rho)^{d_s}\varepsilon^{2}}}^{\diamond}$ & $\bigO{\frac{\kappa^{c_c}}{(1-\rho)^{d_c}\varepsilon^{2}}}$ &  \xmark \\
    DSGDA~\cite{Gao2022} & FS & S & $\bigO{\frac{\sqrt{n}\xzf{L}\kappa^3}{(1-\rho)^{2}\varepsilon^{2}}}^{\P}$ & $\bigO{\frac{\xzf{L}\kappa^3}{(1-\rho)^{2}\varepsilon^{2}}}$ & \xmark \\
    \midrule
    DPOSG~\cite{Mingrui2020} & S & S & $\bigO{\frac{\sigma^2}{(1-\rho^t)^{2}\varepsilon^{12}}}^{\ddagger}$ & $\bigO{\frac{\sigma^2}{(1-\rho^t)^{2}\varepsilon^{12}}}^{\ddagger}$ &  mult. comm.\\
    DREAM~\cite{Chen2022} & S & S & $\bigO{\frac{\xzf{L}\kappa^3\sigma}{\varepsilon^{3}}}$ & $\bigO{\frac{\xzf{L}\kappa^2}{\sqrt{1-\rho}\varepsilon^{2}}}$ & mult. comm. \\
    \saa{\textbf{This Paper (\gda{})}} & S & S & \xzrr{$\cO\left(\frac{\xzf{L}\kappa^3\sigma}{\min\{1/\kappa,(1-\rho)^2\}\epsilon^3}\right)$} & \xzrr{$\cO\left(\frac{\xzf{L} \kappa^2}{\min\{1/\kappa,(1-\rho)^{2}\}\varepsilon^{2}}\right)$} &  \xmark\\
    \midrule
    \xzf{\textbf{Lower Bounds}$^{\circ}$}  & S & S & \xzrr{$\Omega\left(L\sigma\epsilon^{-3}\right)$}  & \xzrr{$\Omega\left(\frac{L}{\sqrt{1-\rho}\epsilon^{2}}\right)$} &  \xmark\\
	\bottomrule
  \end{tabular}
  }
   \caption{
    The P column 
    \saa{shows} the problem setting: 
    finite-sum (FS) or stochastic (S) --for FS setting, 
    $n$ denotes the number of component functions. The U column 
    \saa{indicates} whether stochastic (S) or deterministic (D) gradients are used. Some works do not \saa{explicitly state} the dependence upon the spectral gap or condition number --we use constants $a,b,c,d>0$ with subscripts $s$ and $c$ 
    \saa{indicating that} these unknowns are related to the \emph{sample} or \emph{communication} complexities, respectively. An \xmark\enskip in the final column indicates there is no special requirement for the theoretical results to hold. \saa{Table notes:} $^{\dagger}$GT-DA considers a slightly different problem than~\eqref{eq:main-problem-2} as consensus is only enforced on $\{\mathbf{x}_i\}$ or $\{\mathbf{y}_i\}$; additionally, GT-DA performs deterministic updates; hence, $\sigma$ does not appear in the 
    \saa{complexity} results. $^{\diamond}$GT-SRVR can remove the dependence upon $\sigma$ by computing a full gradient periodically. $^{\P}$DSGDA uses a variance reduction technique which removes the bounded variance assumption (hence $\sigma$ does not appear); however, it is unclear whether this technique can be extended to the stochastic setting. $^{\ddagger}$DPOSG considers the nonconvex-\emph{nonconcave} problem, hence $\kappa$ is undefined for this setting; additionally, $t>1$ represents the required number of communications per iteration.
    $^\circ$
    \xzf{\cite{arjevani2022lower} considers centralized nonconvex minimization problems defined by functions with Lipshitz gradients and assumes that their stochastic oracles are unbiased and have bounded variance. Similarly, \cite{Sun2020} considers a deterministic distributed nonconvex minimization problem under the same conditions. 
    The oracle complexity of distributed methods cannot be less than that of centralized methods, and the communication complexity of minimax problems cannot be less than that of minimization problems; \saa{therefore, their lower bounds apply here.}
    }
    }
    \label{table:related_works}
\end{table*}%

We provide a brief literature review 
on \emph{decentralized optimization} methods (specifically for nonconvex and stochastic problems), 
and \sa{discuss both centralized and decentralized methods for minimax problems}.

\paragraph{Decentralized optimization.} D-PSGD~\cite{Lian2017} first advocated for the use of decentralized methods and provided convergence analysis for a stochastic gradient-type method. $D^2$~\cite{Tang2018} improved the analysis of D-PSGD to allow for data heterogeneity. More recently, gradient tracking has been utilized to further enhance the convergence rate of new methods; see~\cite{Lu2019,Zhang2020,Koloskova2021,Xin2021GradTrack} for further discussions. Variance reduction methods that mimic updates from the SARAH~\cite{Nguyen2017} and SPIDER~\cite{Wang2019} methods provide optimal gradient complexity results at the expense of large batch computations; examples include D-SPIDER-SFO~\cite{Pan2020}, D-GET~\cite{Sun2020}, GT-SARAH~\cite{Xin2021Sarah}, DESTRESS~\cite{Li2022}. To avoid the large batch requirement of these methods, the STORM~\cite{Cutkosky2019,Xu2020} and Hybrid-SGD~\cite{Tran2022} methods have also been adapted to the \emph{decentralized setting}; see GT-STORM~\cite{Zhang2021} and GT-HSGD~\cite{Xin2021HSGD}. Both types of variance reduction have recently been extended to include a proximal term in ProxGT-SR-O/E~\cite{Xin2021ProxGT} and DEEPSTORM~\cite{MancinoBall2022}. There are many other decentralized methods which handle various problem settings, but an exhaustive discussion is beyond the scope of this work; we refer interested readers to the references in the above works for more details.

\paragraph{Minimax optimization.} Before discussing purely decentralized minimax optimization methods, we first provide a brief overview of minimax optimization methods in the \emph{centralized setting}. 
\xz{
In recent years, a \sa{significant 
amount of work has been proposed}~\cite{chen2021accelerated,jin2020local,lin2020gradient,lin-near-optimal,lu2020hybrid,ostrovskii2021efficient, thekumparampil2019efficient,zhang2021robust,yang2022faster}. 
Moreover, the lower complexity bounds \sa{have also been} 
studied for centralized minimax algorithms in \cite{zhang2019lower,zhang2021complexity,li2021complexity}.
Additionally, more methods 
\sa{employing} \emph{variance reduction} have been considered to improve the performance of the stochastic minimax algorithms, e.g., see \cite{xu2020enhanced,huang2021efficient,luo2020stochastic,zhang2022sapd}. In this paper, \sa{to control the noise accumulation,} we 
\sa{propose \gda}, \sa{a decentralized 
method employing} the SPIDER \emph{variance reduction} technique~\cite{fang2018spider}, a variant of SARAH~\cite{nguyen2017sarah}.
}

For the \emph{decentralized setting}, we summarize some representative work for solving the minimax problem in Table~\ref{table:related_works}. The \sa{method} GT-DA~\cite{Tsaknakis2020} 
\sa{is proposed for a slightly modified version of~\eqref{eq:main-problem-2} in the \emph{deterministic} setting;} this method only enforces consensus on $\mathbf{x}_i$ variables and as such requires the $\mathbf{y}$-subproblem \yx{to} be solved to an increasing accuracy at each iteration. GT-SRVR~\cite{Xin2021} is closely related to \gda{}, our proposed algorithm; \sa{that said, the analysis for GT-SRVR is only 
provided} for the finite-sum problem, \gmb{and 
the dependence upon important parameters such as $\kappa$ and $\rho$ is unclear}. Similarly, DSGDA~\cite{Gao2022} \sa{is proposed for} the finite-sum \sa{setting, and employs} a stochastic gradient estimator from~\cite{Li2021Sarah}\gmb{, for which it is unclear on how to theoretically extend to the general stochastic setting}. For the purely stochastic 
\sa{case,} DPSOG~\cite{Mingrui2020} is a general method that solves the \sa{nonconvex-nonconcave} problem, however, \sa{its 
oracle complexity} is sub-optimal. \sa{Furthermore, DPSOG} 
requires multiple communications \sa{per iteration in order to guarantee the convergence to a stationary point.} 

\paragraph{Comparison with DM-HSGD and DREAM.} We provide a 
detailed comparison \sa{of \gda{}} to two closely related methods: \sa{DM-HSGD~\cite{Wenhan2021} and DREAM~\cite{Chen2022}.} 
The recent DM-HSGD~\cite{Wenhan2021} algorithm adapts the STORM-type update to the decentralized minimax setting; however, there are several critical errors in their proof which impact their results. First, their equation (28) does not hold with the given choice of $\theta$. In fact, $\theta$ must depend on $L$, 
\gmb{for which} 
\sa{it is not clear whether their convergence analysis will go through if one chooses $\theta=\Theta(1/L)$ to make their equation (28) valid, e.g., in this scenario the coefficient of $\mathbf{E}\norm{\bar{u}^t}$ becomes positive and cannot be dropped from the final bound while their convergence analysis requires this term to be dropped.} Second, the algorithm is claimed to solve the minimax problem in~\eqref{eq:main-problem-2} such that $\mathbf{y}\in\mathcal{Y}$ for a convex set $\mathcal{Y}\subseteq\mathbb{R}^m$; however, equation (29) in their Lemma~5 cannot hold unless $\mathcal{Y}\triangleq\mathbb{R}^m$ which means that at best, their analysis is only applicable to~\eqref{eq:main-problem-2} \sa{without simple constraint sets}. The recent DREAM~\cite{Chen2022} is similar to our method in terms of the variance reduction technique used to reduce the \xzf{oracle complexity}. However, their proof \emph{requires} the use of \saa{multi-communication rounds},
\gmbf{i.e., \saa{rather than using a mixing matrix $W$ (satisfying Assumption~\ref{ASPT:mixture-matrix}), each iteration of DREAM uses} $W^K$ for $K=\mathcal{O}(\log(M)/(1-\rho))$ which exhibits the typical behavior of a \emph{centralized} method}.\footnote{\saa{Indeed, for $W$ satisfying Assumption~\ref{ASPT:mixture-matrix}, as $k\to\infty$, $W^k$ converges \textit{linearly} to the averaging matrix $\frac{1}{M}\mathbf{1}\mathbf{1}^\top\in\reals^{M\times M}$.}} 
Our proof technique removes such a requirement 
\saa{while ensuring the convergence of \gda{}} \sa{for any 
connected} network. 

\section{Preliminaries}\label{sec:prelim}
\sa{Throughout the paper, for notational convenience, we define $\bz=(\bx,\by)\in\reals^{d+m}$ to be the concatenation of the $\blx$ and $\bly$ variables. We start with some basic definitions.}

\begin{definition}\label{def:L-smooth}
    A differentiable function \sa{$r$}
is $L$-smooth if $\exists L > 0$ such that 
\sa{$\|\grad r(\bz) - \grad r(\bz')\|\leq L \|\bz-\bz'\|$, $\forall\, \bz,\bz'$.}
\end{definition}

Since only stochastic estimates of \yx{$\{\nabla f_i\}$} are available to the agents, we introduce the concept of a stochastic oracle and \saa{state our 
assumptions on the oracle below.} 

\begin{definition}\label{def:oracle}
For all $i\in[M]$, given a random sample $\xi$, we define the stochastic oracle of $\grad f_i(\bx,\by)$ at $(\blx,\bly)$ to be $\tilde{\nabla} f_i(\bx,\by;\xi)$. Additionally, given a set of random samples $\cB$, 
\vspace*{-2mm}
\begin{equation}\label{eq:stoch-oracle-avg}
    \vspace*{-2mm}
    \small{G_i(\cB)\triangleq \frac{1}{|\cB|}\sum_{\xi\in \cB} \tilde{\grad } f_i(\bx_i,\by_i;\xi)}
\end{equation}
 is the averaged stochastic estimator for \yx{$\grad f_i(\bx_i,\by_i)$} with random samples $\cB$. 
 $G_i^t(\cB)$ denotes \eqref{eq:stoch-oracle-avg} evaluated at $(\blx_i^t,\bly_i^t)$.
\end{definition}




Below we 
state our assumptions on the functions $\{f_i\}_{i\in\cV}$ and their stochastic gradient oracles and also, assumptions on the primal objective $\Phi(\cdot)$ and the mixing matrix $W$.
\begin{assumption}
\label{ASPT:smooth-f}
     \sa{There exists $L>0$ such that \gmb{$f_i:\reals^{d+m}\to\reals$ }
     is $L$-smooth for all $i\in[M]$.} 
\end{assumption}
\begin{assumption}\label{ASPT:SC}
\sa{There exists $\mu>0$ such that
$f_i(\bx,\cdot)$ is $\mu$-strongly concave for all fixed $\bx$ and $i\in[M]$.}
\end{assumption}
\begin{remark}
\sa{Assumptions~\ref{ASPT:smooth-f} and \ref{ASPT:SC} imply that $f$ is $L$-smooth and $f(\bx,\cdot)$ is $\mu$-strongly concave for all $\bx$.}
\end{remark}

\begin{definition}\label{def:Phi-condition-number}
\sa{The condition number of 
\eqref{eq:main-problem} is 
$\kappa\triangleq L/\mu$. 
The primal function 
is defined as $\Phi(\cdot)\triangleq\max_{\by}f(\cdot,\by)$.}
\end{definition}

\begin{assumption}\label{ASPT:lower-bounded-Phi}
\gmb{
\saa{$\Phi$} is lower bounded, i.e.,} \sa{$\inf_\bx \Phi(\bx)>-\infty$.} 
\end{assumption}

\gmb{Assumptions~\ref{ASPT:smooth-f},~\ref{ASPT:SC}, and~\ref{ASPT:lower-bounded-Phi} are standard in the minimax literature, e.g., see~\cite{li2021complexity}.} For all $i\in[M],$ we make the following assumptions for the stochastic oracles $\tilde{\nabla} f_i(\bx,\by;\xi)$ \saa{(see 
Definition~\ref{def:oracle}).} 
\begin{assumption}\label{ASPT:general-sto-grad} \sa{The stochastic gradients are unbiased and have finite variance. Namely, there exists $\sigma > 0$ such that for all $i\in [M]$ and for any $\bz = (\bx,\by)\in \mathbb{R}^{d\sa{+} m}$,} the stochastic gradient $\tilde{\grad} f_i(\bz;\xi)$ satisfies the conditions:\\[-4mm]
\begin{enumerate}[itemsep=1ex]
    \item \small{$\mathbf{E}\big[\tilde{\grad} f_i(\bz;\xi)~|~\bz\big] = \grad  f_i(\bz)$};\vspace*{-1mm}
    \item \small{$\mathbf{E}\big[\|\tilde{\grad} f_i(\bz;\xi) - \grad f_i(\bz)\|^2~|~\bz \big]\leq \sigma^2$}.
\end{enumerate}
\end{assumption}%
\saa{Assumption~\ref{ASPT:general-sto-grad} is} common in the literature, e.g., \cite{can2022,fallah2020optimal,yang2022faster}, \saa{and satisfied} when gradients are estimated from randomly sampled data points with replacement. \gmb{We also make the following 
assumption on $f_i$.}
\begin{assumption}
\label{ASPT:smooth-F-mean-squared}
Given random $\xi$, \saa{for any $\bz,\bz'\in \mR^{d+m}$,} we assume $\bE[\|\sa{\tilde\grad} f_i(\bz;\xi) - \sa{\tilde\grad} f_i(\bz’;\xi) \|^2]\leq L^2\bE[\|\bz-\bz'\|^2]$. 
\end{assumption}

\xz{Indeed, Assumptions~\ref{ASPT:general-sto-grad} and \ref{ASPT:smooth-F-mean-squared} imply \cref{ASPT:smooth-f} holds, see section~2.2 in \cite{tran2022hybrid}.}
Finally, we state our assumptions on the mixing matrix \saa{$W\in\reals^{M\times M}$.} 

\begin{assumption}\label{ASPT:mixture-matrix}
    Consider a connected
network $\cG\triangleq (\cV,\cE)$, 
\sa{where $\cV$ denotes the set of $M$ agents
and $\cE\subseteq\cV\times\cV$ is the set of} 
edges. 
\gmb{An ordered pair $(i,j)\in\cE$ if agent $i$ can directly communicate with agent $j$.} 
Let $W\triangleq (w_{ij})\sa{\in\reals^{M\times M}}$ be a 
matrix \gmb{with non-negative entries }
such that  
\begin{enumerate}
    \item  (Decentralized property) If $(i,j)\notin \cE$, then $w_{ij}=0$;
   \item \gmb{(Doubly stochastic property) $W\mathbf{1}=\mathbf{1}$ and $W^\top\mathbf{1}=\mathbf{1}$};
   \item  (Spectral property) \gmb{$\rho\triangleq\|W-\Pi\|_2\in[0,1)$;}
\end{enumerate}
where \sa{$\Pi\triangleq\frac{1}{M}\mathbf{1}\mathbf{1}^\top\in\reals^{M\times M}$ denotes the average operator.} 
\end{assumption}

\gmb{Notice that $W$ is not assumed to be symmetric; \saa{hence, Assumption~\ref{ASPT:mixture-matrix} 
covers both strongly-connected weight-balanced directed networks and undirected ones~\cite{Xin2021HSGD}.} This is a weaker assumption compared to some related papers~\cite{Mingrui2020,Zhang2021,Chen2022}, which require a symmetric $W$ and hence are only theoretically applicable to undirected networks.}

\gmb{Indeed, the main problem in~\eqref{eq:main-problem} is equivalent to $\min_{\bx} \Phi(\bx)$. Moreover, the norm of the gradient of the primal function $\Phi(\bx)$, i.e., $\|\grad \Phi(\bx)\|$, is widely used as the \sa{convergence metric in the algorithmic analysis for nonconvex minimax problems} in the literature. Given that we solve~\eqref{eq:main-problem-2}, we quantify the consensus errors among the agents related to the average point $\bar{\mathbf{z}}=(\bar{\blx},\bar{\bly})$ and also $\|\grad \Phi(\bar{\mathbf{x}})\|$.}
\begin{algorithm}[t]
\begin{algorithmic}[1]
\STATE \textbf{Input}: $Z^0$, $\{\eta_x,\eta_y\}$, $\{S_1,S_2,q,T\}$
\FOR{$t = 0,1,2,...,\gmb{T-1}$}
\STATE $X^{t+1} = WX^t - \eta_x D^t_x$
\STATE  $Y^{t+1} = WY^t + \eta_y D^t_y$
\IF{\yx{$\mathrm{mod}(t,q)=0$}}
\STATE \gmb{Let $\cC_{i}^{t+1}$ be random samples with $|\cC_{i}^{t+1}|=S_1$}
\STATE \gmb{$\bv^{t+1}_i = G_{i}^{t+1}(\cC_{i}^{t+1}),\;\forall i\in[M]$}
\ELSE 
\STATE \gmb{Let $\cB_{i}^{t+1}$ be random samples with $|\cB_{i}^{t+1}|=S_2$}
\STATE \gmb{$\bv_{i}^{t+1} =  G^{t+1}_{i}(\cB^{t+1}_i)-G^{t}_i(\cB^{t+1}_i) + \bv^{t}_i,\;\forall i\in[M]$}
\ENDIF
\STATE $D^{t+1}_x=W(D^t_x + V^{t+1}_x-V^t_x)$
\STATE $D^{t+1}_y=W(D^t_y + V^{t+1}_y-V^t_y)$
\vspace{1mm}
\ENDFOR
\STATE \textbf{Output}:\xz{
$(X^\tau, Y^\tau)$, where $\tau$ is selected from $\{0,\ldots,T-1\}$ uniformly at random} 
\end{algorithmic}
\caption{\sa{\gda{}}}
\label{Alg:main-alg}
\end{algorithm}

\section{\xz{\texttt{DGDA-VR}} Method}
\label{sec:method}
\gmb{We introduce our proposed \texttt{D}ecentralized \texttt{G}radient \texttt{D}ecent \texttt{A}scent - \texttt{V}ariance \texttt{R}eduction, \gda{}, method in Algorithm~\ref{Alg:main-alg} for solving~\eqref{eq:main-problem-2}.}
Specifically, \sa{through local computations and communicating with neighboring agents,} each \sa{agent-$i$ for $i\in [M]$} 
\sa{iteratively updates its} local variable $\mathbf{z}_i \triangleq(\mathbf{x}_i,\mathbf{y}_i)\in \mathbb{R}^{d\sa{+} m}$ -- 
its value at iteration $t\in\mathbb{N}$ is denoted by
$\mathbf{z}^t_i \triangleq(\mathbf{x}^t_i,\mathbf{y}^t_i)\in \mathbb{R}^{d\sa{+}m}$. 
For 
notational convenience, we define the following terms.
\begin{definition}\label{def:matrices-1}
\sa{$(X^t,Y^t),D^t,V^t\in \mR^{M \times (d+m)}$ such that 
\vspace*{-1mm}
\small{\begin{align*}
\vspace*{-1.5mm}
     & X^t \triangleq [\bx^t_1, \ldots, \bx^t_M]^\top, \ Y^t \triangleq [\by^t_1, \ldots, \by^t_M]^\top, \\
     & V^t \triangleq [\bv^t_1, \ldots, \bv^t_M]^\top, D^t \triangleq [\bd^t_1, \ldots, \bd^t_M]^\top,
\end{align*}}%
}%
where \sa{$(X^t,Y^t)$ 
denotes the iterates of \gda{} displayed in~\cref{Alg:main-alg}, ${\bd_i^t}=({\bd_{x,i}^t}, {\bd_{y,i}^t})$ denotes the gradient-tracking term, and ${\bv_i^t}=({\bv_{x,i}^t}, {\bv_{y,i}^t})$ 
denotes the SPIDER-type stochastic gradient estimates of agent-$i$ at iteration $t\in\mathbb{N}$.} 
Let $Z^t=\begin{pmatrix}
    X^t,Y^t
\end{pmatrix}\in\mR^{M \times (d+m)}$ for $t\in\mathbb{N}$.
\end{definition}
\begin{definition}\label{def:matrices-2}
For $t\geq 0$, given a matrix $X^t\in\mR^{M\times \saa{d}}$, we define \sa{$\bar{X}^t\triangleq \Pi(X^t)$, i.e.,} let 
\vspace*{-1.5mm}
\begin{equation}\label{def:bar-matrix}
 \small{\bbx^t \triangleq \frac{1}{M}\sum_{i=1}^M \sa{\bx_i^t},\quad \Bar{X}^t = \mathbf{1} \bbx^{t \top}
,\quad
X_\perp \sa{\triangleq} X^t - \Bar{X}^t, }
\vspace*{-1mm}
\end{equation}
and $\{\sa{\bar{Y}^t},Y^t_\perp,\sa{\bar{Z}^t},Z^t_\perp,\sa{\bar{D}^t},D^t_\perp,\bar{V}^t,V^t_\perp\}$ is defined similarly.
\end{definition}%
\gmb{Notice that under Assumption~\ref{ASPT:mixture-matrix}, Algorithm~\ref{Alg:main-alg} implies that}
\vspace{-2.5mm}
\begin{equation}\label{eq:avg-update}
    \begin{aligned}
    \bbx^{t+1} &= \bbx^t-\eta_x\bbd^t_x,\quad\ \bbd^{t+1}_x = \bbd^t_x + \bbv^{t+1}_x - \bbv^t_x, \\
        \bby^{t+1} & = \sa{\bby^t}+\eta_y\bbd^t_y, \
        \quad \bbd^{t+1}_y = \bbd^t_y + \bbv^{t+1}_y - \bbv^t_y, 
    \end{aligned}
    \vspace{-1mm}
\end{equation}
hold for all $t\geq 0$. Moreover, when $\bbd^0=\bbv^0$, it holds that $\bbd^t=\bbv^t$
\sa{for $t\in\mathbb{N}$}; thus, in such scenarios, we have
\vspace{-1.5mm}
\begin{equation}\label{eq:avg-update-2}
    \bbx^{t+1} = \sa{\bbx^t}-\eta_x\bbv^t_x,\quad
        \bby^{t+1} = \bby^t+\eta_y\bbv^t_y,\quad\sa{\forall~t\in\mathbb{N}}. 
\end{equation}

\begin{figure*}[t]
  \setlength\tabcolsep{1pt}
  \begin{center}
  \begin{tabular}{cccc}
	\includegraphics[width=0.225\textwidth]{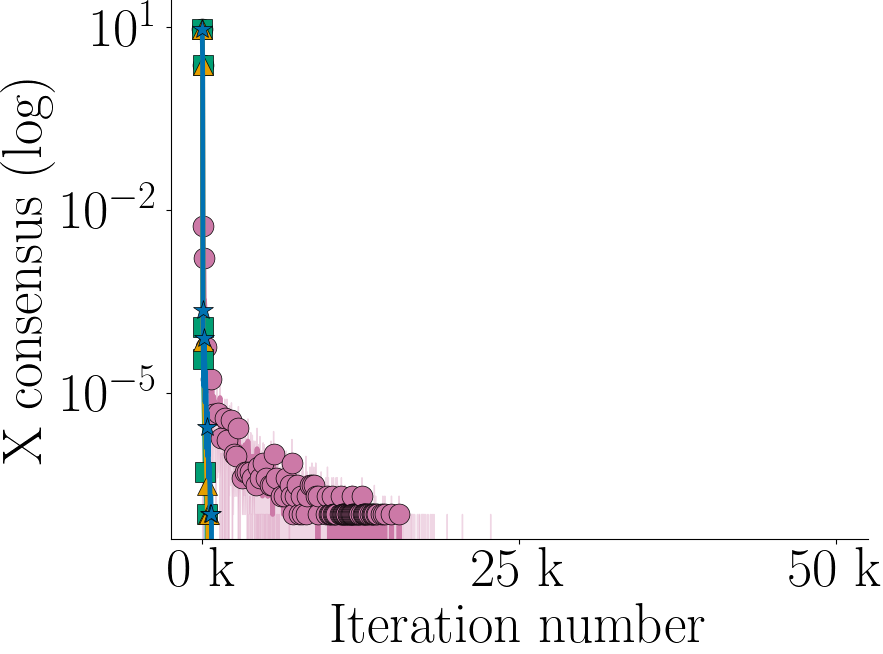} &

    \includegraphics[width=0.225\textwidth]{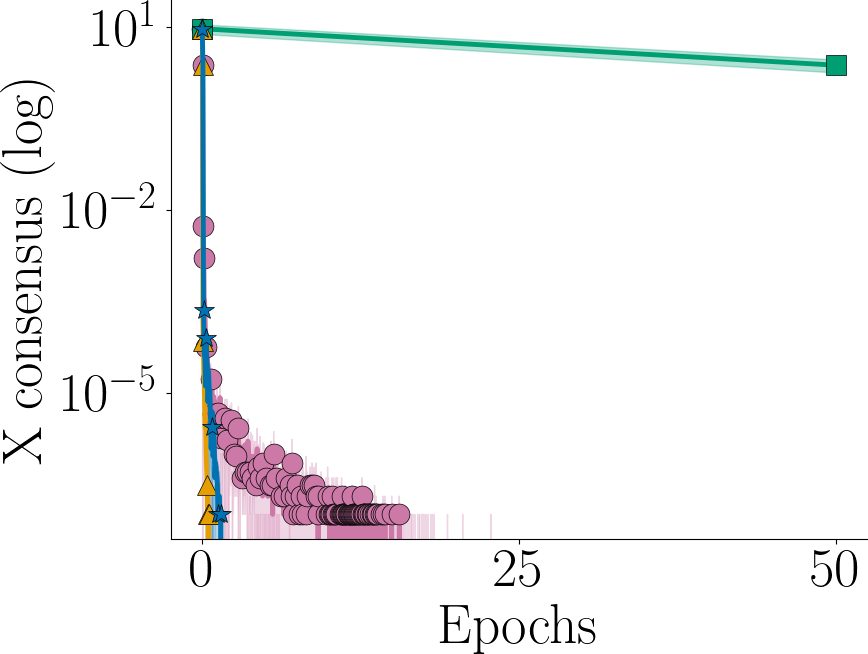} &
 
	\includegraphics[width=0.225\textwidth]{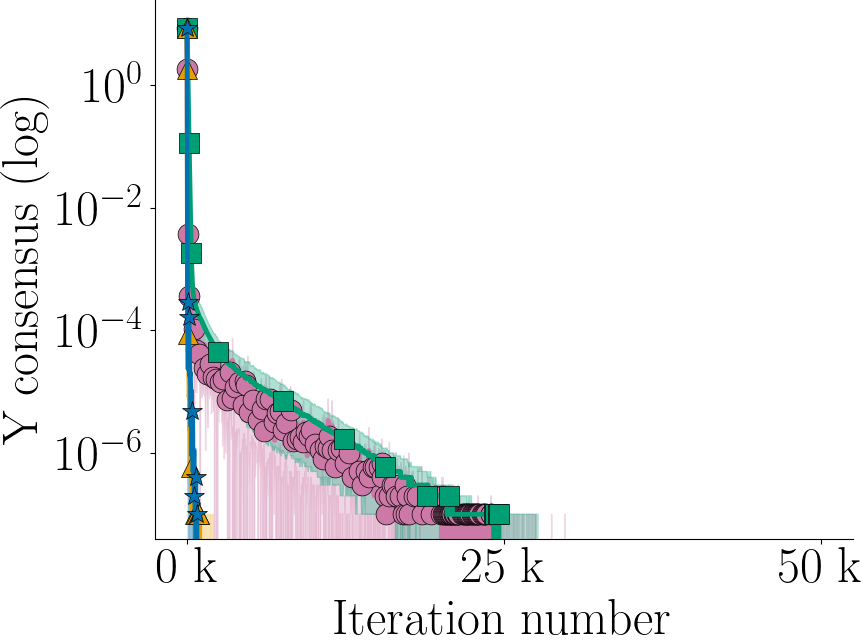} &

    \includegraphics[width=0.225\textwidth]{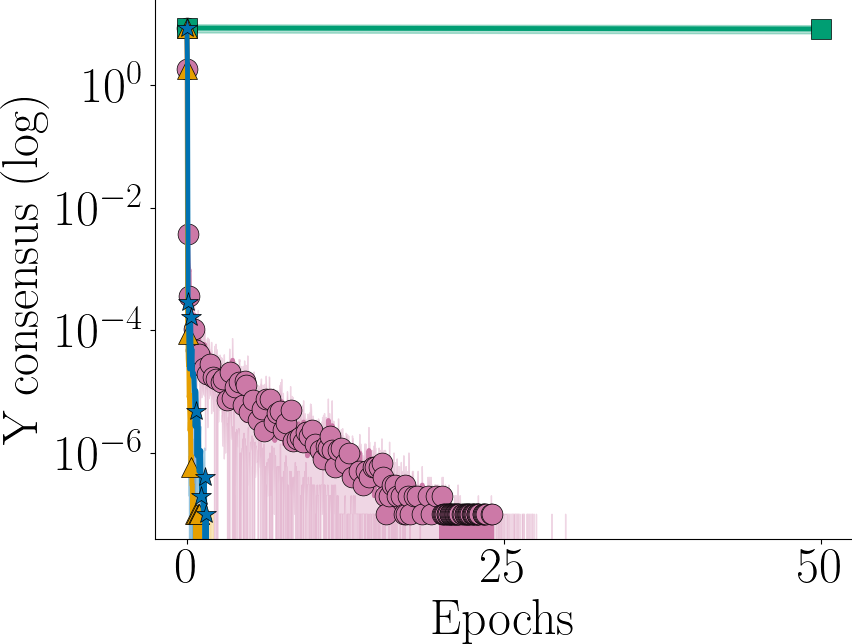} 
    \\
	\includegraphics[width=0.225\textwidth]{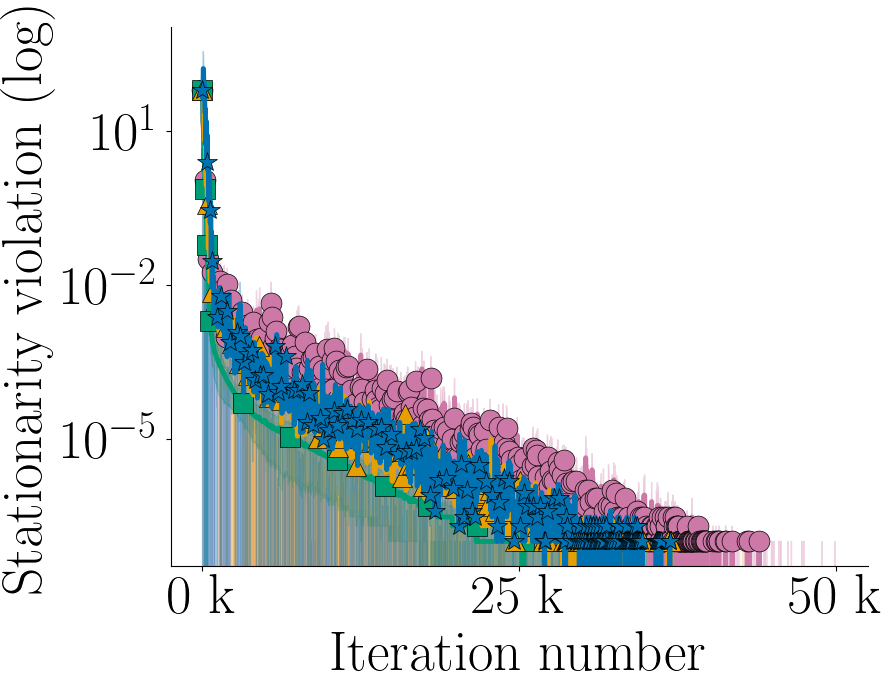} &

    \includegraphics[width=0.225\textwidth]{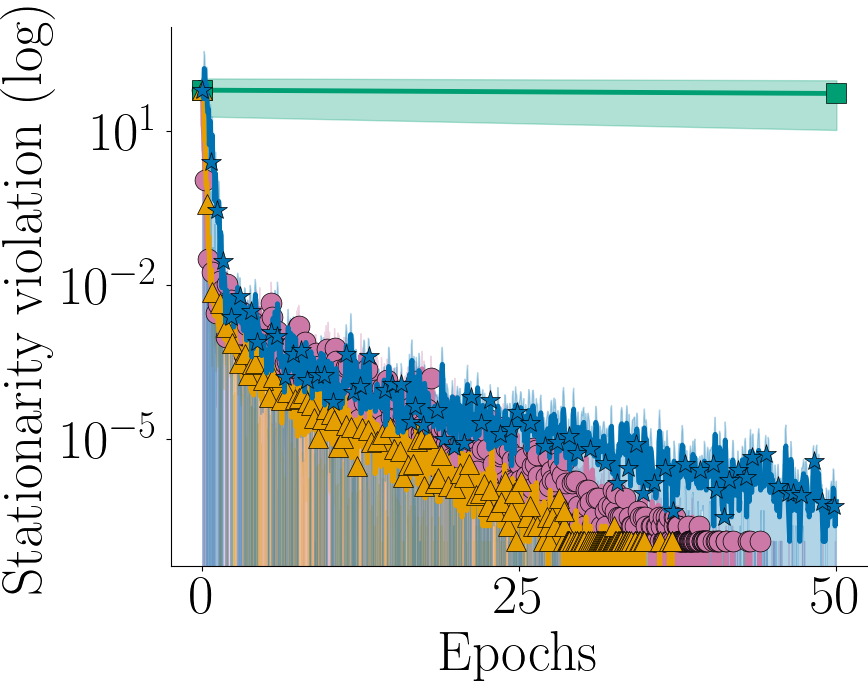} 
    &

	\includegraphics[width=0.225\textwidth]{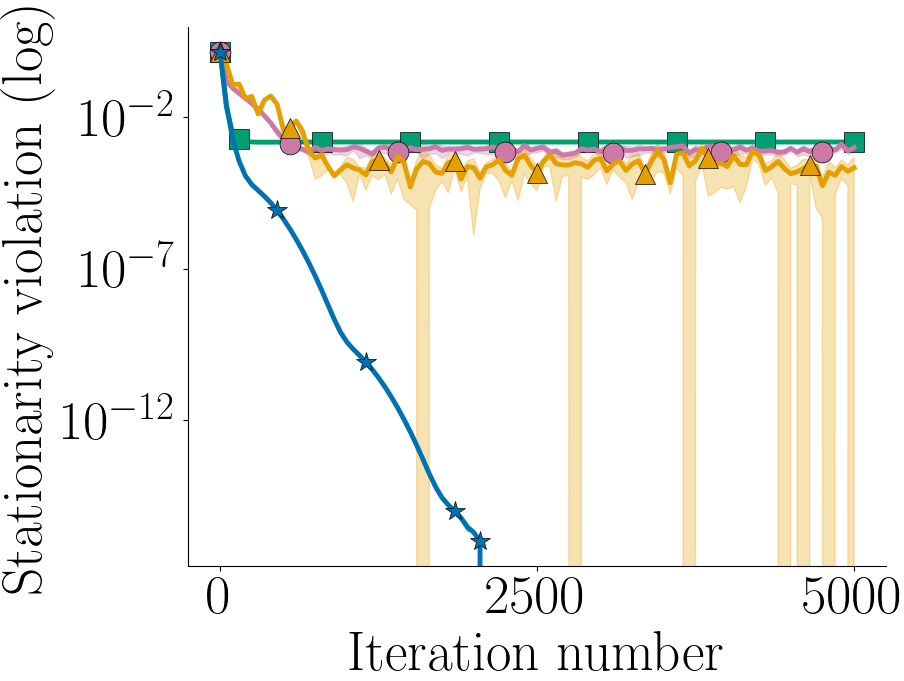} &
 
    \includegraphics[width=0.225\textwidth]{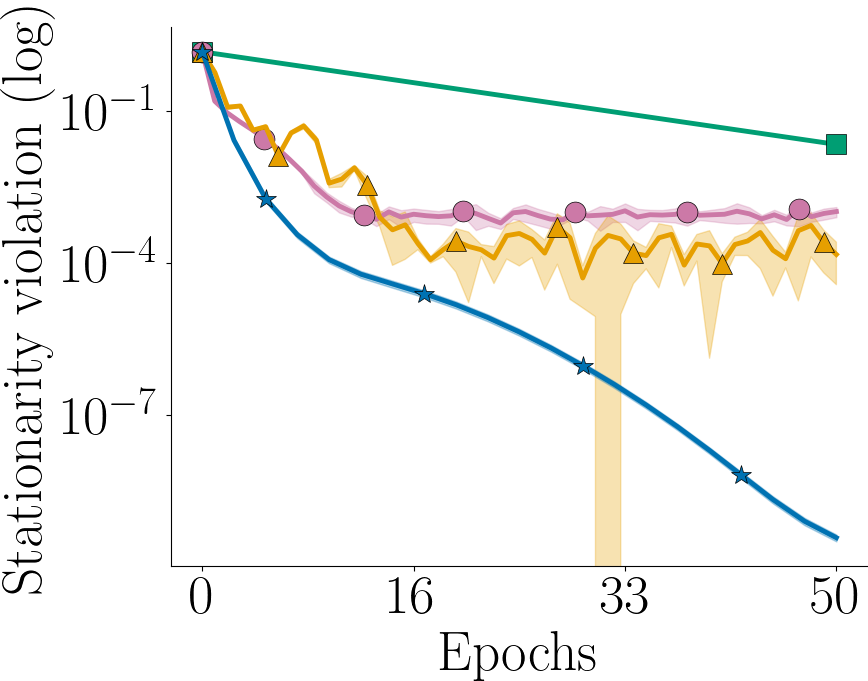} 

    \\

	\includegraphics[width=0.225\textwidth]{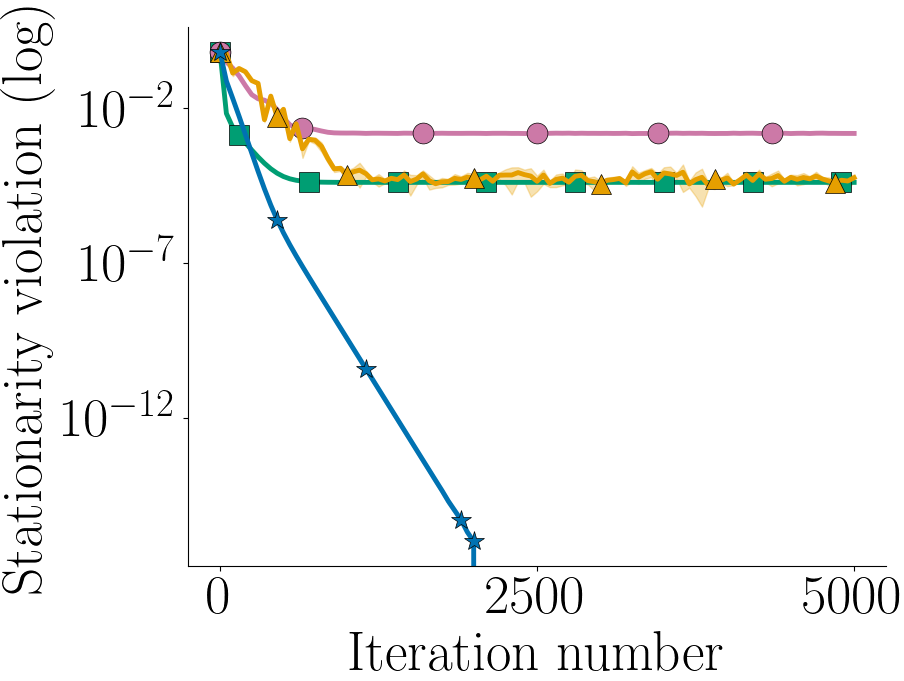} &

    \includegraphics[width=0.225\textwidth]{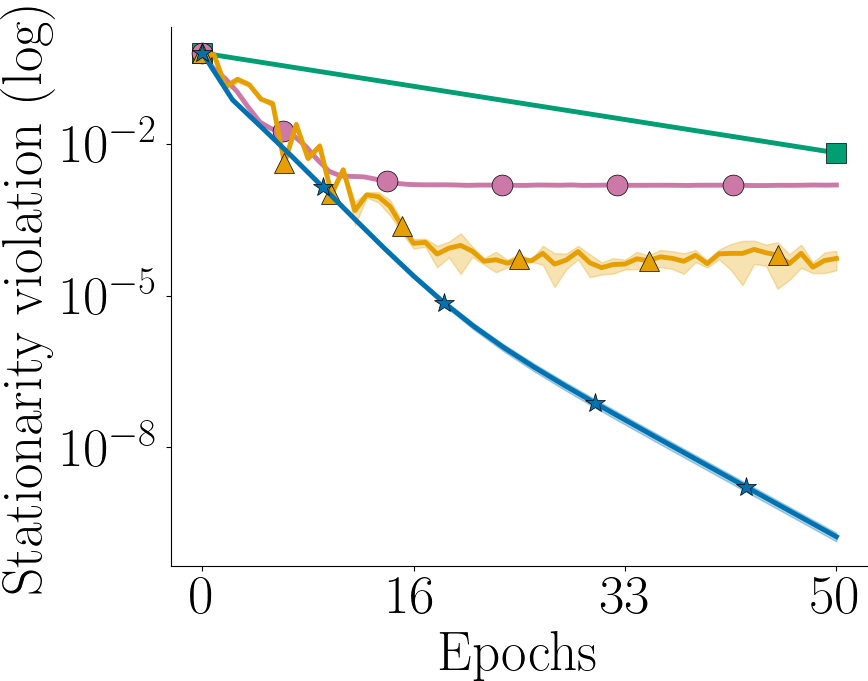} &

	\includegraphics[width=0.225\textwidth]{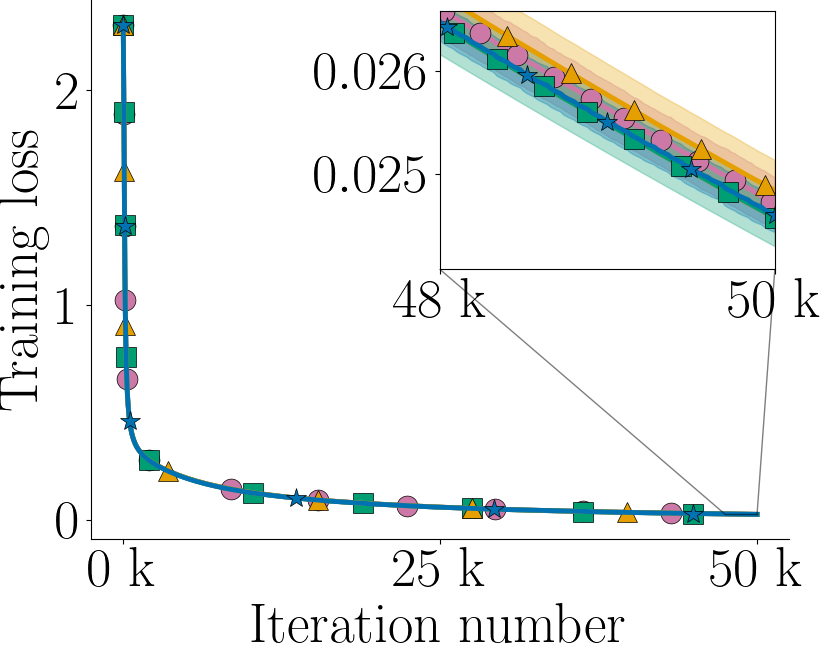} &

    \includegraphics[width=0.225\textwidth]{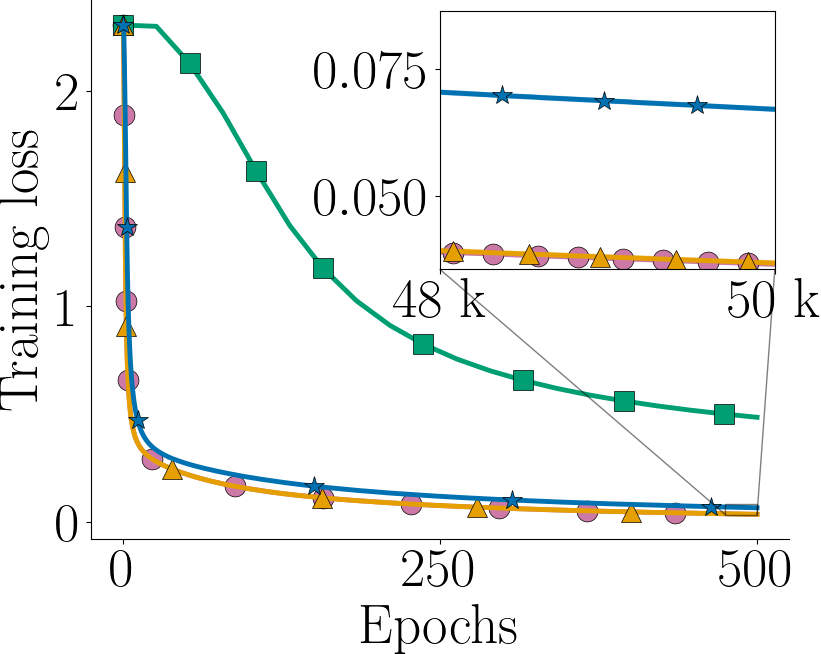} 

    \\
    
	\includegraphics[width=0.225\textwidth]{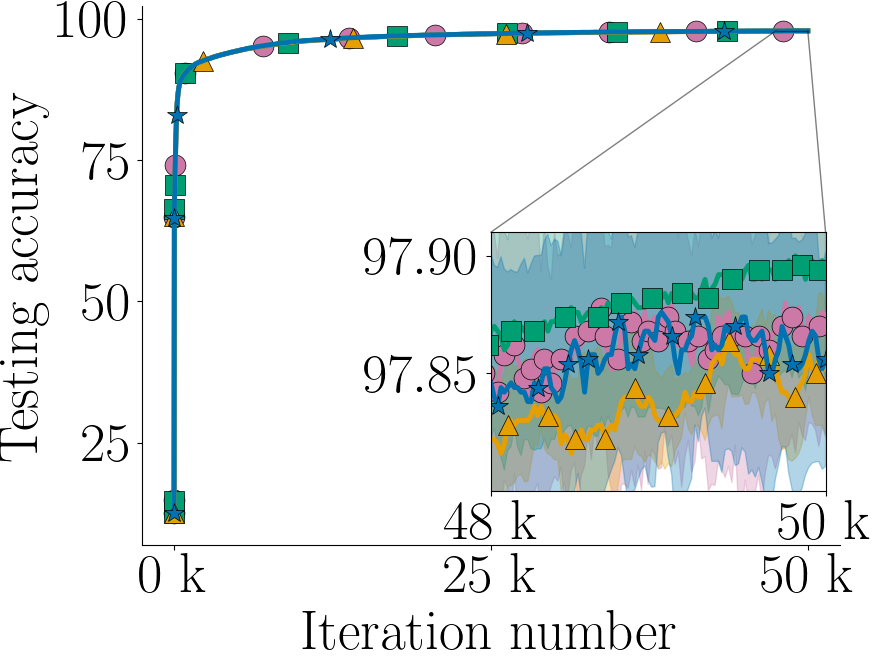} &

    \includegraphics[width=0.225\textwidth]{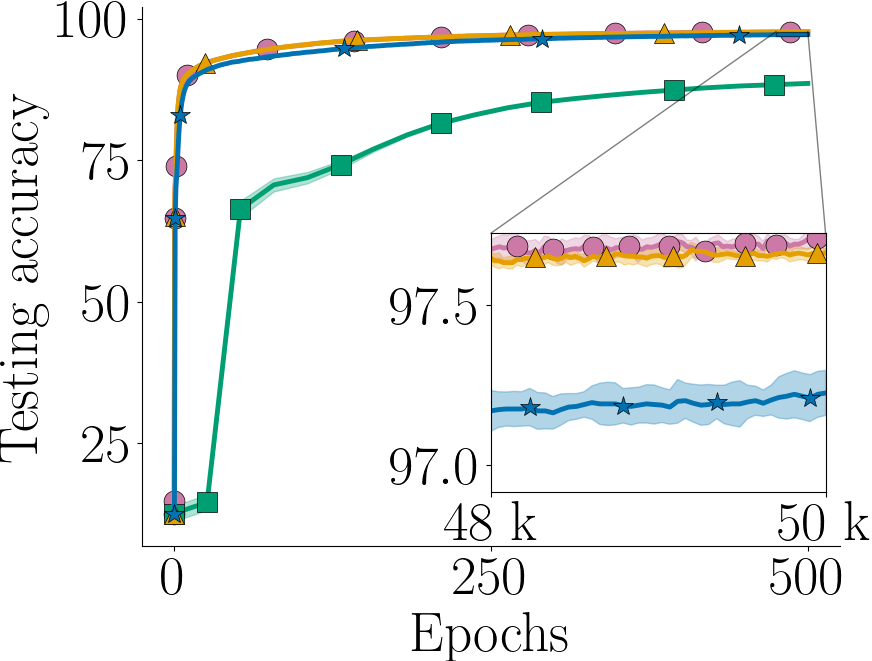} &
 
	\includegraphics[width=0.225\textwidth]{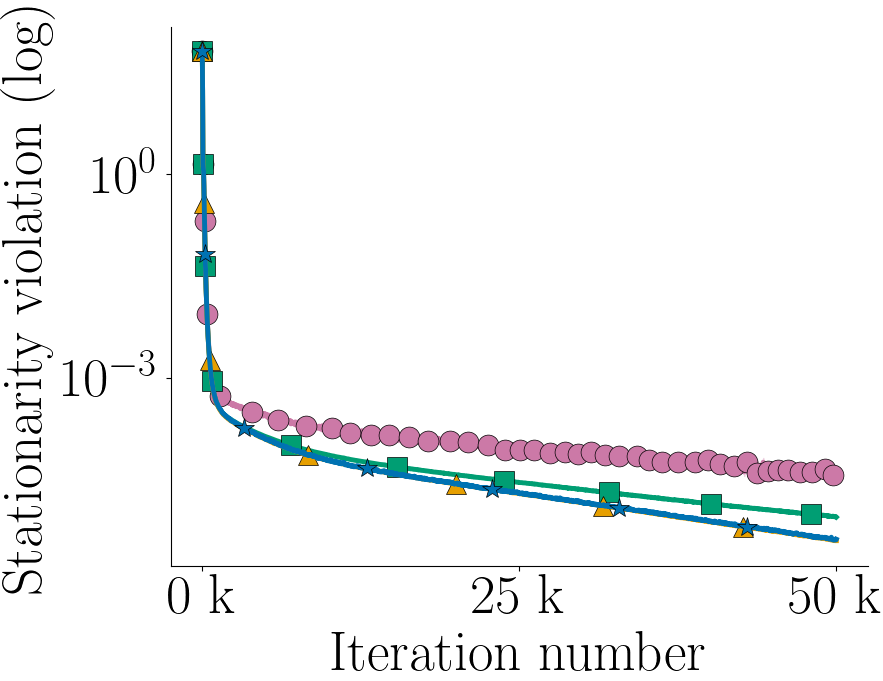}  &

    \includegraphics[width=0.225\textwidth]{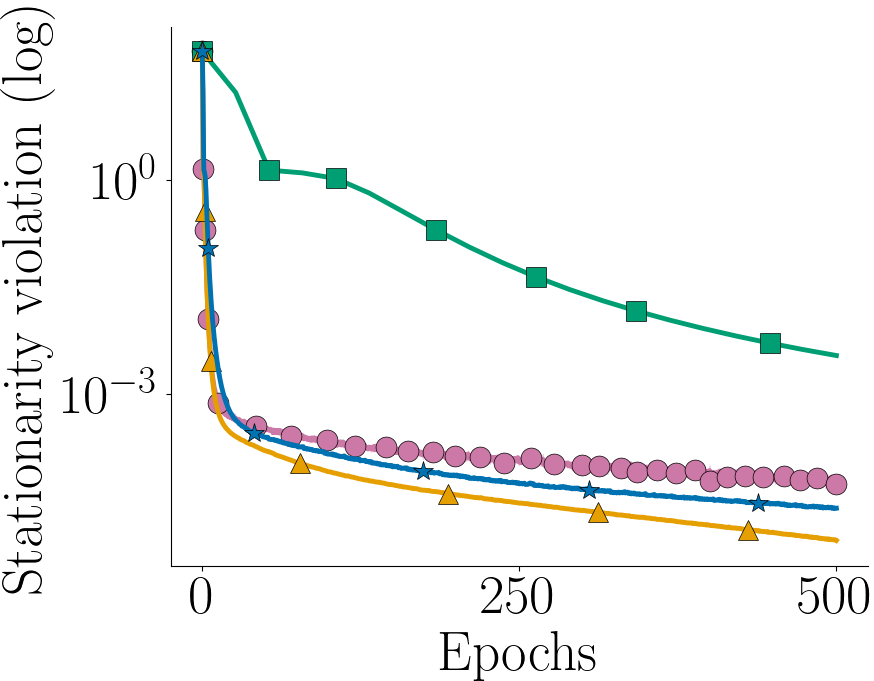} 
    
    \\
    \multicolumn{4}{c}{\includegraphics[width=0.5\textwidth]{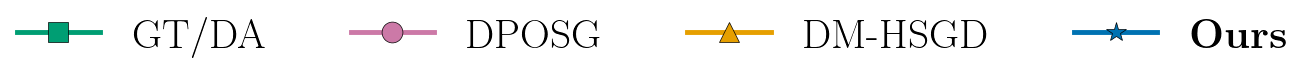}}
	
  \end{tabular}

  \end{center}
  \caption{  Pictures 1-6 are for the PL game~\eqref{experiments:pl}. 
  Pictures 7-10 for the robust non-convex linear regression model~\eqref{experiments:lr}; the first two correspond to the a9a dataset, while the last two correspond to the ijcnn1 dataset.  Pictures 11-16 for the robust neural network training problem~\eqref{experiments:nn}. The arrangement of these pictures follows a left-to-right, then top-to-bottom order.
  }
  \label{fig:all_exps}
\end{figure*}

 \begin{figure*}[ht]
  \setlength\tabcolsep{1pt}
  \begin{center}
    \includegraphics[width=0.245\textwidth]{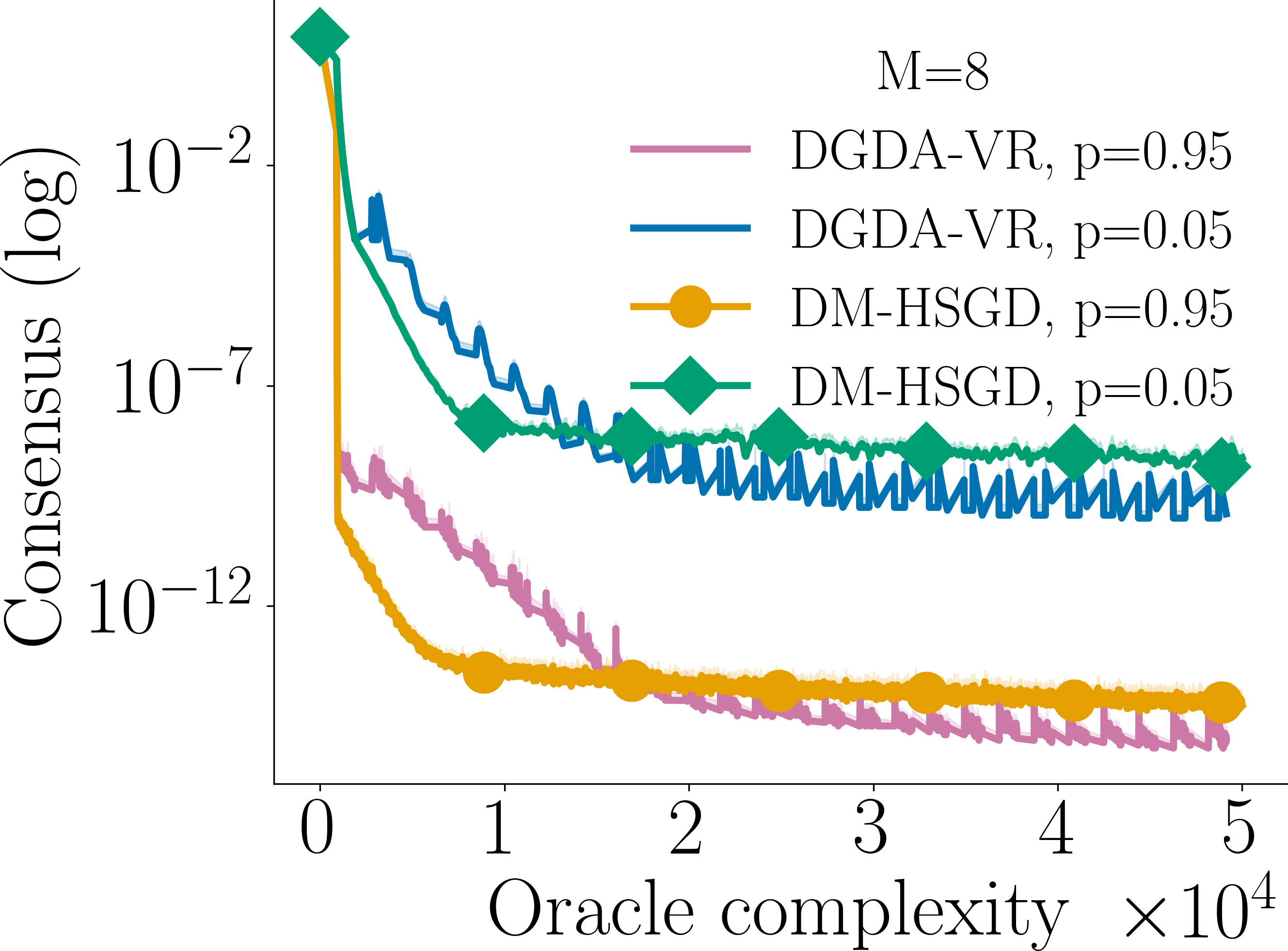} 
    \includegraphics[width=0.245\textwidth]{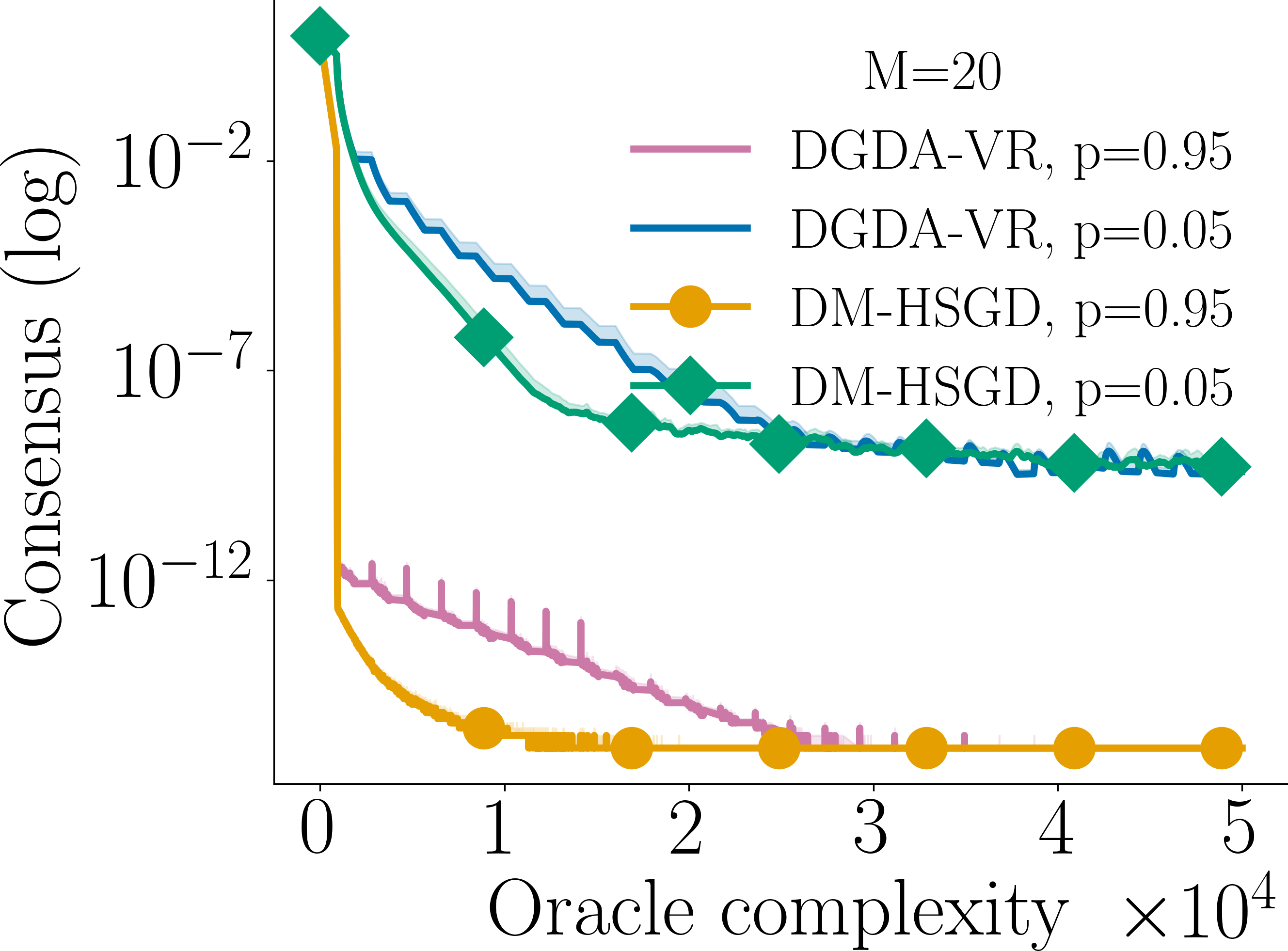} 
    \\
    \includegraphics[width=0.245\textwidth]{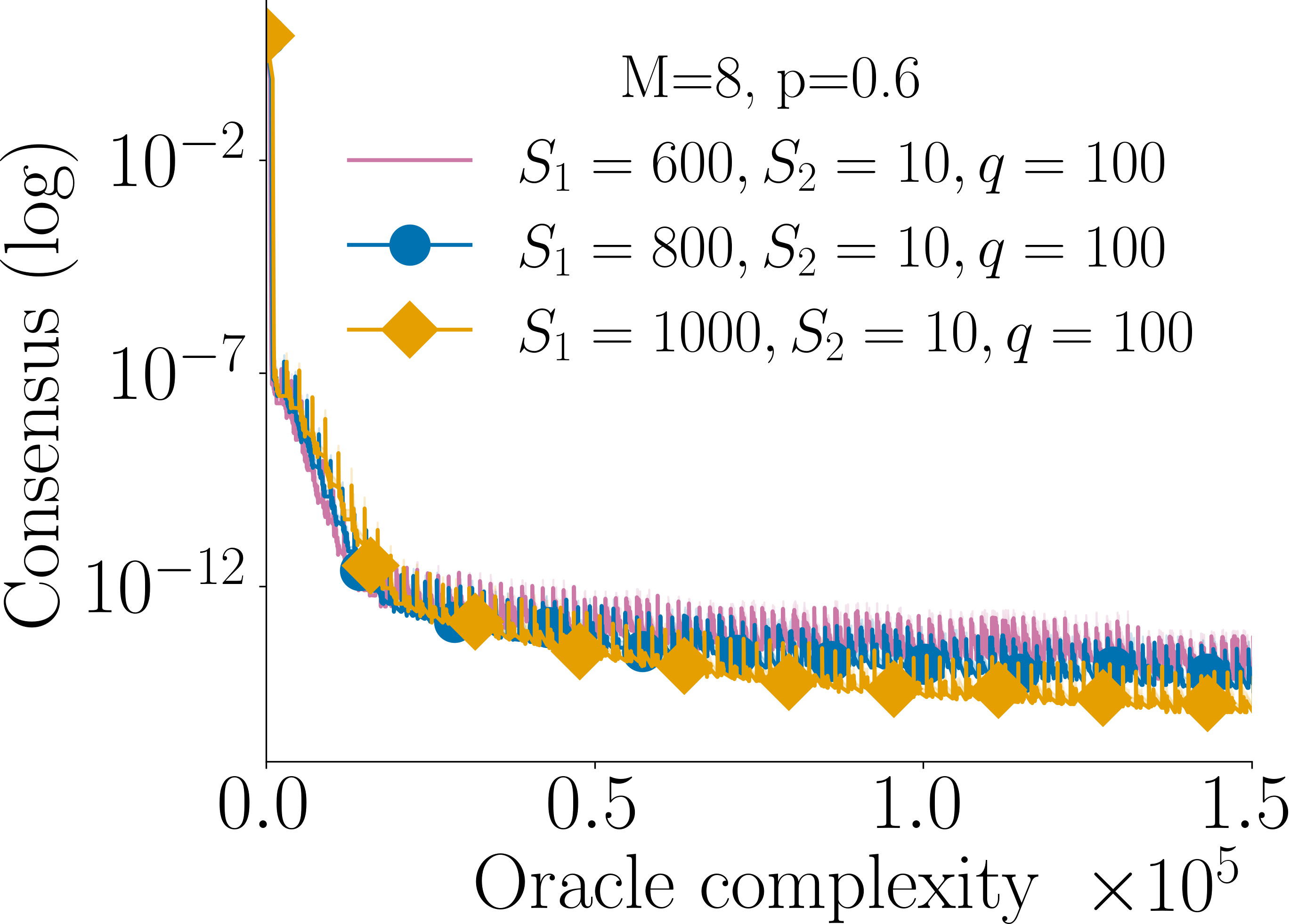} 
    \includegraphics[width=0.245\textwidth]{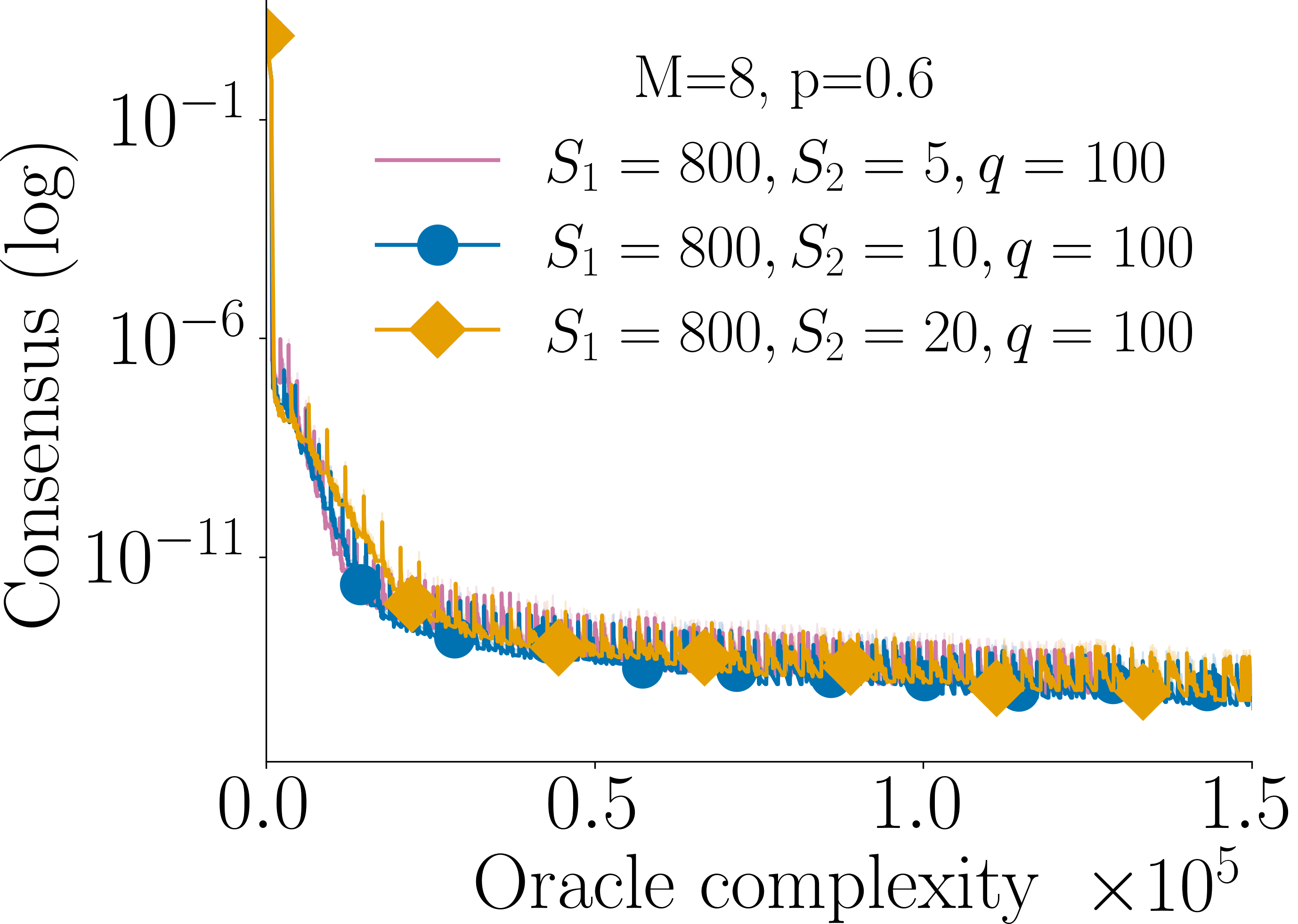} 
    \includegraphics[width=0.245\textwidth]{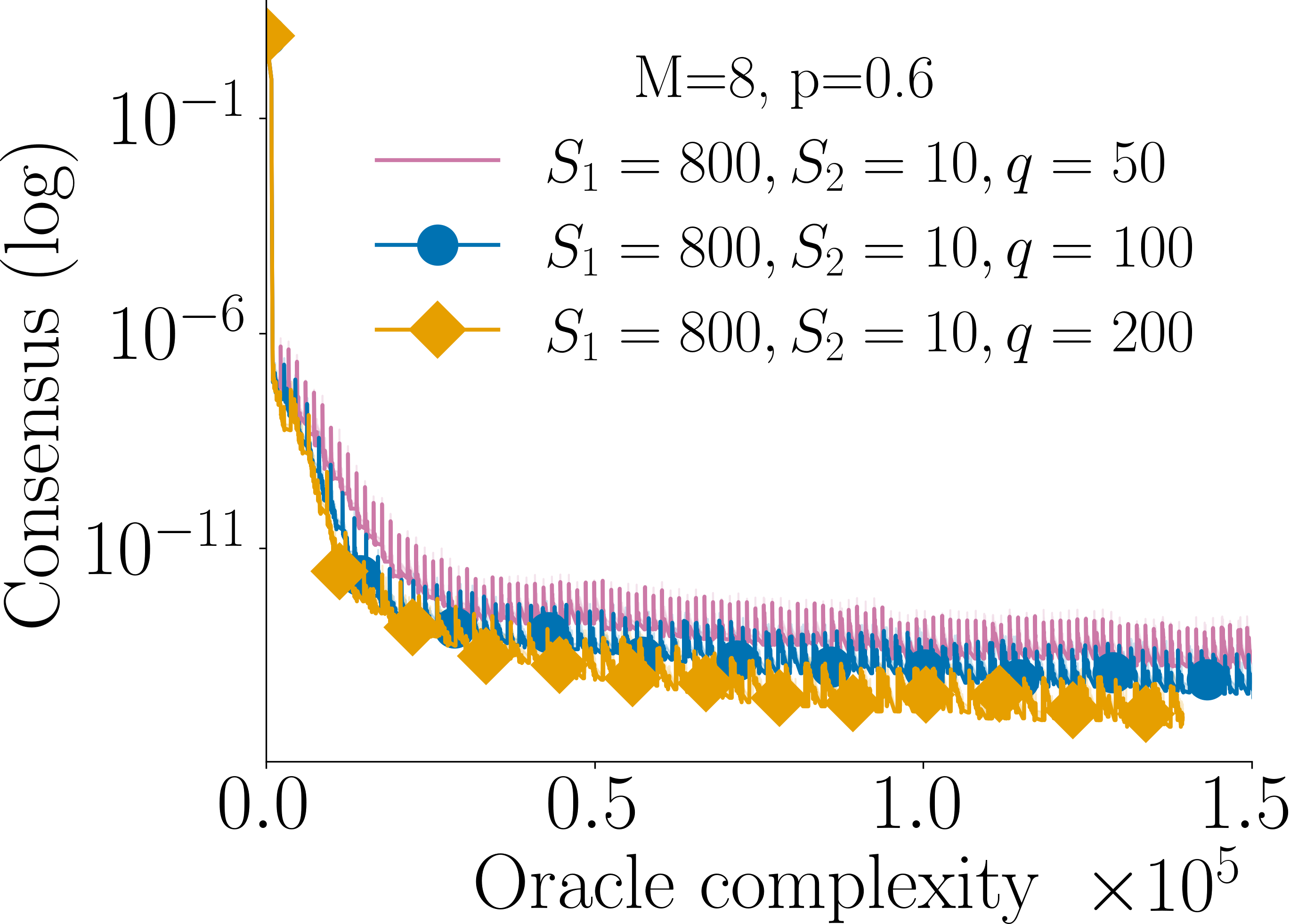} 
\end{center}
  \caption{\xzf{Sensitivity analysis results for the PL game~\eqref{experiments:pl}. The first two plots show the sensitivity analysis in terms of graph connectivity $\rho$, while the last three show the sensitivity analysis in terms of the batchsizes $S_1,S_2$ and the frequency $q$.} \gmbf{Here, oracle complexity refers to data points visited.}}
  \label{fig:sensitive-analysis}
\end{figure*}
\subsection{\sa{Main Results}}
\sa{In \gmb{a} multi-agent system, for any $\epsilon>0$, our aim for each agent-$i$ is to compute $\bz_i(\epsilon)=\Big(\bx_i(\epsilon),\by_i(\epsilon)\Big)$ such that} 
\vspace{-2.5mm}
{{\small
\begin{subequations}
\label{eq:stop_cond}
\begin{align}
&\sa{\mathbf{E}\Big[\|\grad\Phi(\bbx_\epsilon)\|\Big]\leq \epsilon,} \label{eq:stop_cond-a}\\
&\sa{\mathbf{E}\Big[\norm{\bby_\epsilon-\by^*(\bbx_\epsilon)}^2\Big] = \cO(\epsilon^2)},
\label{eq:dual-violation-bound}\\
&\sa{\mathbf{E}\Big[\sum_{i=1}^M\norm{\bz_i(\epsilon)-\bar{\bz}_\epsilon}^2\Big] = \cO(\epsilon^2),}
\label{eq:consensus-violation-bound}
\end{align}
\end{subequations}}
\vspace*{-3mm}
\\
}%
\vspace*{-1.5mm}
\sa{where $\bar{\bz}_\epsilon=(\bbx_\epsilon,\bby_\epsilon)\triangleq\frac{1}{M}\sum_{i=1}^M\big(\bx_i(\epsilon),\by_i(\epsilon)\big)$ and 
\begin{equation}
    \label{def:y-star}
    \by^*(\cdot)\triangleq \argmax_{\by} f(\cdot,\by)
    \vspace*{-2mm}
\end{equation}%
denotes the best-response function.}
\sa{
We 
show that \gda{} can indeed generate $\{\bz_i(\epsilon)\}_{i\in\cV}$ such that \eqref{eq:stop_cond} holds. More importantly, in the decentralized optimization context, let $T_\epsilon$ denote the minimum number of communication rounds required to 
compute $\{\bx_i(\epsilon)\}_{i\in\cV}$ satisfying \eqref{eq:stop_cond} in a 
\yx{decentralized} manner --- in each communication round, each agent-$i$ transmits \yx{two} vectors of size $(d+m)$ to its neighbors, i.e., $\bz_i^t$ and $\bd_i^t$. According to \gda{}, $T_\epsilon$ communication rounds require each agent-$i$ to make $C_\epsilon\triangleq\lceil \frac{T_\epsilon}{q}\rceil S_1+T_\epsilon S_2$ calls to its stochastic oracle $\tilde\grad f_i$. Our aim is to provide bounds on the expected communication and oracle complexities, i.e., $T_\epsilon$ and $C_\epsilon$.} \sa{Moreover, 
we will provide precise bounds on the dual suboptimality as in~\eqref{eq:dual-violation-bound} and on the consensus violation (the deviation from the average) for $\{\bz_i(\epsilon)\}_{i\in\cV}$ 
as in~\eqref{eq:consensus-violation-bound}. 
}
The result below shows our guarantee on 
\eqref{eq:stop_cond-a}.
\begin{theorem}\label{thm:main-result}
  Suppose \yx{Assumptions~
  \ref{ASPT:smooth-f}-\ref{ASPT:mixture-matrix}  hold}, and  $\{\eta_x,\eta_y\}$ and $\{S_1,S_2,q\}$  are chosen such that
  \vspace{-2mm}
  {\small
  \begin{equation}\label{eq:final-step-size-short}
  \small{\begin{aligned}
        &\eta_y = \Theta\left(\frac{1}{L}\min\Big\{(1-\rho)^2, \frac{1}{\kappa}\Big\}\right),\; \eta_x = \Theta\left(\frac{\eta_y}{\kappa^2}\right),\\
        \vspace{-1mm}
        &  \xzrr{S_1 = \Theta\left(\frac{\kappa^2\sigma^2}{\epsilon^2}\right)},\;
      S_2 \geq q,
      \quad
      q\geq 1.
  \end{aligned}}
  \end{equation}}%
  \vspace{-2mm}\\
  Given $\epsilon>0$, 
  \sa{there exists $T_\epsilon\in\mathbb{N}$ such that
  \vspace{-1.2mm}
  {\small
  \begin{equation*}
        T_\epsilon = \cO\Big(\max\left\{\frac{1}{\eta_x},\xzrr{\frac{L\kappa}{\eta_y}},
     \xzrr{\frac{L^2\kappa^2}{(1-\rho)M}}
     \right\}\epsilon^{-2} \Big),
     \vspace{-1.2mm}
  \end{equation*}}%
  and $\{X^t\}_{t=0}^T$ generated by \gda{} satisfies}
  \vspace{-1mm}
{\small
 \begin{equation}
 \label{eq:grad_bound}
 \sa{\frac{1}{T}\sum_{t=0}^{T-1}\mathbf{E}\Big[\|\grad\Phi(\bbx^t)\|\Big]\leq \epsilon,\quad\forall~T\geq T_\epsilon,}
 \vspace{-2mm}
 \end{equation}}%
 where $\bbx^t$ is defined in \eqref{def:bar-matrix}.
\end{theorem}
 \begin{remark}
    Without loss of generality, 
    $L\geq 1$. Indeed, \cref{ASPT:smooth-F-mean-squared} holds for all $\hat{L}$ such that $\hat{L}\geq L$; therefore, 
    \vspace{-0.5mm}
    {\small
    $$
    T_\epsilon=\cO\Big( \frac{\xzrr{\max\{L\kappa^2,L^2\kappa\}}}{\min\{1/\kappa,(1-\rho)^2\}}\cdot\frac{1}{\epsilon^2}\Big).
    \vspace{-0.5mm}
    $$}%
    Given $T=T_\epsilon, S_1$ and $S_2\geq q$, the optimal $q=\Theta(\sqrt{S_1})=\Theta\Big(\xzrr{\frac{\kappa\sigma}{\epsilon}}\Big)$ and \xzrr{$S_2= \Theta(q)$}, so $TS_2+TS_1/q \sim 2 TS_2 =\xzrr{ \cO\left( \frac{\max\{L\kappa^3,L^2\kappa^2\}\sigma}{\min\{1/\kappa,(1-\rho)^2\}\epsilon^3}\right)}.$
\end{remark}%
The result below shows our guarantee on \eqref{eq:dual-violation-bound} and \eqref{eq:consensus-violation-bound}.
\begin{theorem}\label{thm:sub-result}
\sa{Under the premise of \cref{thm:main-result},}
\vspace{-2.5mm}
{\small
\begin{align}
  &\sa{\frac{1}{T}\sum_{t=0}^{T-1}\mathbb{E}[\|Z^t_\perp\|_F^2]=
  \cO\Big(\xzrr{M\epsilon^2/(L^2\kappa^2)}\Big),
  \quad \forall~T\geq T_\epsilon},\\
  &\frac{1}{T}\sum_{t=0}^{T-1}\mathbb{E}[\norm{\by^{t}-\by^*(\bbx^{t})}^2]=\xzrr{ \cO(\epsilon^2/L^2)},\quad \forall~T\geq T_\epsilon,
\end{align}}%
\vspace{-3mm}
\\
where \sa{$\Lambda_0\triangleq\max\{\|Z^0_\perp\|^2_F,\|D^0_\perp\|^2_F\}$}.
\end{theorem}
\begin{remark}
Let $\tau_\epsilon$ be a random variable with a uniform distribution over $\{0,\ldots,T_\epsilon-1\}$. Then \eqref{eq:grad_bound} implies that $\mathbf{E}\Big[\|\grad\Phi(\bbx^{\tau_\epsilon})\|\Big]\leq \epsilon$. Furthermore, we also have $\mathbb{E}[\|Z^{\tau_\epsilon}_\perp\|_F]= \cO(\epsilon)$ and $\mathbb{E}[\norm{\by^{\tau_\epsilon}-\by^*(\bbx^{\tau_\epsilon})}] = \cO(\epsilon)$. 
\end{remark}

\begin{remark}
    \xzf{Since the final complexity bound
    depends 
    on the choice of 
    $\eta_x, \eta_y, S_1, S_2, \sa{q}$, we evaluate the tightness of our results by comparing these parameters with those in related work. Our selection of the VR parameters $S_1 = \mathcal{O}(\kappa^2\epsilon^{-2}), S_2 = \mathcal{O}(\kappa\epsilon^{-1}), q = \mathcal{O}(\kappa\epsilon^{-1})$ is consistent with the optimal choice in single-loop centralized VR methods, 
    e.g., 
     \sa{(Luo et al. 2020).} 
    The \sa{time-scale} ratio $\eta_y/\eta_x = \kappa^2$ aligns with the ratios used in existing works on GDA methods \sa{(Lin, Jin, and Jordan 2020).}
    To adapt the GDA to the decentralized setting, we have introduced a factor of $\frac{1}{\kappa}$ into the selection of $\eta_y$. DREAM can set $\eta_y=\frac{1}{L}$ but requires \saa{multi-communication rounds}. It is not yet clear if this cost can be further reduced, and whether $\frac{1}{\kappa}$ represents the optimal adjustment -- nevertheless, our analysis 
    \saa{seems to be tight when compared to the existing results.}} 
\end{remark}


 



  
  

\section{Numerical Experiments}
\label{sec:numerics}
We test our proposed method on three problems: a quadratic minimax problem, robust non-convex linear regression, and robust neural network training. For the first and third problem, we let $M=8$ such that each agent is represented by an NVIDIA Tesla V100 GPU. For the second problem, we test methods in a serial manner to facilitate more general reproducibility; here, we let $M=20$. In all cases, we 
use a ring \saa{(cycle)} graph \saa{with equal weights on edges including self loops, i.e.,  $w_{i,i-1}=w_{i,i}=w_{i,i+1}=1/3$ for all $i\in[M]$.} 
The learning rates for all tests are chosen such that $\eta_y\in\{10^{-1},10^{-2},10^{-3}\}$ and we tune the ratio $\frac{\eta_x}{\eta_y}\in\{1,10^{-1},10^{-2},10^{-3}\}$. We test our proposed method against 3 methods: DPSOG~\cite{Mingrui2020}, DM-HSGD~\cite{Wenhan2021}, and the deterministic GT/DA~\cite{Tsaknakis2020}. {The code is made available at \url{https://github.com/gmancino/DGDA-VR}.}

\subsection{A Polyak-Lojasiewicz game}


 
 
    

	
  
We consider a slightly modified version of the two-player Polyak-Lojasiewicz game from~\cite{Chen2022PL}. Namely, we make the problem decentralized by letting each agent $i\in[M]$ contain a dataset of triples $\small\{(\bp_{ij},\bq_{ij},\br_{ij})\}_{j=1}^{n}$ where each vector lies in $\mR^d$. 
\saa{For all $i\in[M]$, let ${\small f_i}:\reals^d\times\reals^d\to\reals$} such that
\vspace{-2mm}
    {\small \begin{equation}\label{experiments:pl}
        f_{i}(\bx_i,\by_i)=\frac{1}{2}(\bx_i)^\top\bP_i\bx_i-\frac{1}{2}(\by_i)^\top\bQ_i\by_i+(\bx_i)^\top\bR_i\by_i,
            \vspace{-2mm}
    \end{equation}}%
where
        {\small$\bP_i=\frac{1}{n}\sum_{j=1}^{n}\bp_{ij}\bp_{ij}^\top$, $\bQ_i=\frac{1}{n}\sum_{j=1}^{n}\bq_{ij}\bq_{ij}^\top+\alpha\bI$, $\bR_i=\frac{1}{n}\sum_{j=1}^{n}\br_{ij}\br_{ij}^\top$}
for some $\alpha>0$ which guarantees the problem is strongly-concave in $\by$; we choose $\alpha=1$ for these experiments. Data is generated in the same manner as in~\cite{Chen2022PL}\footnote{See~\url{https://github.com/TrueNobility303/SPIDER-GDA/blob/main/code/GDA/pl_data_generator.m}} to guarantee that $\small \bP_i$ is singular; hence, the problem is not strongly-convex in $\bx$. Here, $n=1{,}000$ and we fix the mini-batch size for all methods to be 1 (besides GT/DA). For our proposed method, we set $q=S_1=100$. We run each algorithm 
\saa{for} 50,000 iterations and 
plot the results 
of 50 epochs \saa{(one pass over the whole dataset through sampling is an epoch)} 
for each method. We measure the stationarity violation as \gmbf{$\norm{\sum_{i=1}^{M}\nabla_{\bx}f_{i}(\bar{\bx},\by^{(*)})}_2^2+\norm{\bX_{\perp}}_F^2+\norm{\bY_{\perp}}_F^2$, where $\by^{(*)}\triangleq \arg\max_{\by}\sum_{i=1}^Mf_i(\bar{\bx},\by)$ for $\bar{\bx}=\frac{1}{M}\sum_{i=1}^{M}\bx_{i}$.}
\saa{Results shown in Figure~\ref{fig:all_exps} 
demonstrate that 
\gda{} is competitive against SOTA for computing a stationary point.} 
\\
\textbf{Sensitivity Analysis}\quad
\gmbf{To 
\saa{assess} the influence of 
\textit{graph connectivity}, we 
\saa{compared \gda{} against DM-HSGD 
on 
random connected graphs, generated} such that there is an edge between any two nodes with probability $p\in\{0.05, 0.95\}$ \saa{--- corresponding to low and high connectivity scenarios,} respectively. For each $p$, we generate 15 random graphs of size $M\in\{8, 20\}$ 
\saa{-- the average value of $\rho$ over 15 realizations} is 0.94, 0.97, 0.16, 0.1 for $(p,M)$ combinations $(0.05,8)$, $(0.05,20)$, $(0.95,8)$, $(0.95,20)$, respectively. The first two plots in Fig.~\ref{fig:sensitive-analysis} report the sum of 
squared norms of the $\bx$ and $\by$ consensus violations 
\saa{against the oracle complexity.} 
In addition, we generate 15 random graphs for $M=8$ and $p=0.6$, i.e., \saa{moderate connectivity with} $\rho\approx 0.63$, \saa{to test \gda{} using low, moderate, high levels for each parameter $S_1,S_2,q$ while fixing the other two at the moderate level.}
Results are reported in the last three plots of \saa{Fig.}~\ref{fig:sensitive-analysis} which show that our method is not sensitive to the choice of hyper-parameters $S_1,S_2,q$.}

\subsection{Robust Machine Learning}

We consider two 
robust machine learning problems: non-convex linear regression with tabular data and neural network training with image data. Let each agent $i\in[M]$ contain a dataset of points and labels denoted by $\{(\ba_{ij},b_{ij})\}_{j=1}^{n}$ where $b_{ij}$ is the class label of data point $\ba_{ij}$. For these problems, $\by_{i}^{(*)}\triangleq \arg\max_{\by_i}f_i(\bx,\by_{i})$ is not easily computable; as a proxy, we report the stationarity violation using
   \vspace{-1.5mm}
    {\small \begin{equation}\label{eq:numerical:stationarity2}
		\norm{
  \sum_{i=1}^{M}\nabla f_{i}(\bar{\bx},\bar{\by})}_2^2+\norm{\bX_{\perp}}_F^2+\norm{\bY_{\perp}}_F^2.
    \vspace{-1.5mm}
	\end{equation}}

\paragraph{Robust Non-convex Linear Regression}
We consider training a robust version of the non-convex linear regression classifier from~\cite{Sun2020}. 
\saa{For $i\in [M]$, let}
        \vspace{-1.5mm}
    {\small \begin{equation}\label{experiments:lr}
        f_{i}(\bx_i,\by_i)=\frac{1}{n}\sum_{j=1}^n\ln\left(\left(b_{ij}-\bx_{i}^\top(\ba_{ij}+\by_{i})\right)^2/2+1\right)-\frac{\alpha}{2}\|\by_i\|_2^2
        \vspace{-1mm}
    \end{equation} }%
where $b_{ij}\in\{-1,+1\}$ and $\alpha>0$ is a penalty term which guarantees 
\saa{that $f_i$ is strongly-concave in $\by_i$} --we 
\saa{set $\alpha=1$} for these experiments. The $\by$ variable acts as a perturbation to the data; \saa{hence, we seek to minimize the loss on the worst-case data perturbation.} 
\saa{We 
test \gda{} on two datasets:} a9a and ijcnn1\footnote{See:~\url{https://www.csie.ntu.edu.tw/ cjlin/libsvmtools/datasets/}}. \gmbf{We fix the mini-batch to be 32 for all methods beside GT/DA and set $S_1=1{,}000,q=32$ for our method. We} run each method 
\saa{for} 5{,}000 iterations and 
plot \saa{the results 
of 50 epochs 
for each method. Results 
shown in Figure~\ref{fig:all_exps} demonstrate that 
in contrast to \gda{},} the main bottleneck for other methods \saa{is to} achieve consensus among agents. 

\paragraph{Robust Neural Network Training}
We consider a slightly modified version of the robust neural network training problem from~\cite{Deng2022,Sharma2022}. For all $i\in [M]$, let
\vspace{-4mm}
   {\small \begin{equation}\label{experiments:nn}
        f_{i}(\bx_i,\by_i)=\frac{1}{n}\sum_{j=1}^n\ell\left(g_{\bx_i}(\ba_{ij}+\by_i),b_{ij}\right)-\frac{\alpha}{2}\|\by_i\|_2^2,
        \vspace{-3mm}
    \end{equation}}%
where $g_{\bx_i}$ is a neural network parameterized by $\bx_i$, $\ell$ is the cross-entropy loss function, and $\alpha>0$ is a penalty parameter which guarantees that 
\saa{$f_i$ is strongly-concave in $\by_i$ --we set $\alpha=1$} for these experiments. Inspired by~\cite{Deng2022}, we \saa{adopt $g_{\bx_i}$ corresponding to} a two-layer network (200 hidden units) with a tanh activation function, 
and we use the MNIST~\cite{lecun1998mnist} dataset for training. We fix the mini-batch size for all methods to be 100 (besides GT/DA). For 
\gda{}, we set $q=100$ and 
$S_1=7,500$. We run each algorithm to 50{,}000 iterations and 
\saa{plot the results 
of 500 epochs 
for each method. Results shown in Figure~\ref{fig:all_exps} verify that 
\gda{} 
is competitive against the stochastic methods and still outperforms the deterministic method in terms of data passes required to compute a near stationary point.}

\section{Conclusion}
In this work, we proposed a \texttt{D}ecentralized \texttt{G}radient \texttt{D}ecent \texttt{A}scent - \texttt{V}ariance \texttt{R}eduction method, \gda{}, for solving the stochastic nonconvex strongly-concave minimax problem over a connected network of $M$ computing agents. Under the assumption that the computing agents only have access to stochastic first-order oracles, our method incorporates variance reduction and gradient tracking to jointly optimize the sample and communication complexities to be $\bigO{\epsilon^{-3}}$ and $\bigO{\epsilon^{-2}}$, respectively, for reaching an $\epsilon$-accurate solution. For the class of problems considered here, this is the first work which does not require multiple coordinated communications in each iteration to achieve these optimal complexities.

\section*{Acknowledgements}
This work is partly supported by NSF Grant DMS-2208394 and the ONR \saa{grants N00014-21-1-2271 and} N00014-22-1-2573, and also by the Rensselaer-IBM AI Research Collaboration, part of the IBM AI Horizons Network.

\bibliography{reference}

\appendix
\onecolumn
\section{Notation}\label{sec:notation-proof}
To aid readability, we list the frequently used notation in the proof as follows:
\begin{itemize}
\item  We define 
\begin{equation}
\label{def:deltat}
    \delta_t = \|\by^*(\bbx^t) - \bby^t\|^2.
\end{equation}
\item We defined $\Phi(\cdot) = \max_{\by}f(\cdot,\by)$ and the condition number $\kappa$ in Definition \ref{def:Phi-condition-number}.
\item We defined $\by^*(\cdot)=\argmax_{y}f(\cdot,\by)$ in \eqref{def:y-star}.
\item The matrices and vector notations are defined in Definitions \ref{def:matrices-1} and \ref{def:matrices-2}.
The stochastic gradient estimate error $\{E^t\}$ and $\{\be^t_i\}$ are defined as
\begin{equation}\label{def:Et-et}
    \be_i^t\triangleq \bv_i^t-\grad f_i(\bx_i^t,\by_i^t),\quad\sa{\forall~i\in[M],}
\end{equation}
\item We use subscripts $x,y$ to represent the components of vectors or matrices, i.e.,, we define
\begin{equation}\label{def:e_x,e_y}
    \be_{x,i}^t\triangleq \bv_{x,i}^t-\grad_x f_i(\bx_i^t,\by_i^t),\;\be_{y,i}^t\triangleq \bv_{y,i}^t-\grad_y f_i(\bx_i^t,\by_i^t)\;, \text{ for } i\in [M].
\end{equation}
\end{itemize}


\section{The Overview of the Proof}
\xz{In general, the proof of our main result \cref{thm:main-result} is based on the fundamental inequality:}
\xz{\begin{equation}
    \begin{aligned}
        \bE[\Phi&(\bbx^{(n+1)q})]  
         - \bE[\Phi(\bbx^{nq})] \leq \Theta(q\eta_x\sigma^2/S_1)
         \\
         &
         - \sum_{t=nq}^{(n+1)q-1}\Big(\Theta(\eta_x)\|\grad \Phi(\bbx^t)\|^2-\Theta(\eta_x L^2)\delta_t
         -\Theta(\frac{\eta_x L^2}{M})\|Z^t_\perp \|^2_F -\Theta(\frac{\eta_x\eta_y^2L^2}{M})\|D^t_\perp\|_F^2\Big)
    \end{aligned}
\end{equation}}%
\xzrr{which is shown in \cref{lemma:bound-Phi-final}. From this inequality, we bound $\delta_t$, $\|Z^t_\perp\|_F^2$, and $\|D^t_\perp\|^2_F$  in \cref{lemma:bound-deltat-final,lemma:bound-Z-D-perp-final}. To accomplish this, we also need to bound the error $\|E^t\|_F^2$ that is caused by the stochastic oracles and our \emph{variance reduction} gradient estimator. The bound of $\|E^t\|$ is provided in \cref{lemma:bound-Et}, and is frequently used in other parts of the proof.
After this, we invoke \cref{lemma:bound-Z-D-perp-final,lemma:bound-deltat-final} within 
\cref{lemma:bound-Phi-final}, and then obtain the general convergence results of $\|\grad \Phi(\bbx^t)\|$ in \cref{thm:main-convergence} through the following inequality:}
\xzrr{\begin{equation}\label{eq:final-overview}
             \sum_{t=0}^{T-1}\bE[\|\grad \Phi(\bbx^t)\|^2]\leq  \cO\Big(\frac{1}{\eta_x}+\frac{L\kappa\delta_0}{\eta_y}+ \frac{L^2\kappa^2\Lambda_0}{M} + T\cdot\frac{\kappa^2\sigma^2}{S_1}\Big).
\end{equation}%
}%
\xzrr{To achieve this concise bound, we also provide the parameter analysis in 
\xzref{the section Complexity Analysis}
to simplify the complicated terms in the proof. In the last,  we provide the proper parameter choices and obtain the detailed \xzf{oracle complexity} and communication complexity given certain parameter choices in \cref{thm:main-result-detailed} for running \gda{}. Before beginning the analysis, we restate a useful Lemma from the literature that is commonly employed in the convergence analysis of first-order algorithms.}




\begin{lemma}[Proposition 1~\cite{chen2021proximal}]\label{lemma:smooth-Phi}
    Suppose Assumptions~\ref{ASPT:smooth-f},~ \ref{ASPT:SC},~and~\ref{ASPT:lower-bounded-Phi} hold. Then $\Phi(\cdot):\mR^d\rightarrow\mR$ is $(\kappa+1)L$-smooth and $\grad \Phi(\cdot) = f(\cdot,\by^*(\cdot))$. Moreover, $\by^*(\cdot):\mR^d\rightarrow\mR^m$ is $\kappa$-Lipschitz.
\end{lemma}

\section{Convergence Analysis}
\label{sec:main-proof-convergence-analysis}
We begin by analyzing the measure of the dual suboptimality sequence $\{\delta_t\}$. It is important to note that many of the equations in this proof will be reused in other parts of the paper.
\begin{lemma}\label{lemma:bound-deltat-1}
    Suppose Assumptions~\ref{ASPT:smooth-f},~\ref{ASPT:SC} and \ref{ASPT:mixture-matrix} hold. If $\eta_y\in (0,1/L]$, then the inequality
    \begin{equation}\label{eq:bound-deltat-4}
    \begin{aligned}
          \delta_t 
        \leq
        & \Big(1-\frac{1}{4}\eta_y\mu + 3\Big(\frac{4}{\eta_y\mu}-1\Big)\kappa^2\eta_x^2 L^2\Big)\delta_{t-1} 
        + 
         \Big(\frac{4-\eta_y\mu}{4-2\eta_y\mu}\frac{2\eta_y}{\mu}+
           3\Big(\frac{4}{\eta_y\mu}-1\Big)\kappa^2\eta_x^2\Big)
           \frac{2L^2}{M}
           \xz{\|Z^{t-1}_\perp \|^2_F }
        \\
        & +
         \frac{4-\eta_y\mu}{4-2\eta_y\mu}\frac{2\eta_y}{\mu}
           \frac{2}{M} \sum_{i=1}^M \|\be^{\sa{t-1}}_{y,i}\|^2
        \sa{+\Big(\frac{4}{\eta_y\mu}-1\Big)\kappa^2\eta_x^2 \frac{6}{M} \sum_{i=1}^M \|\be^{t-1}_{x,i}\|^2}
         +3\Big(\frac{4}{\eta_y\mu}-1\Big)\kappa^2\eta_x^2
         \|\grad \Phi(\bbx^{t-1})\|^2
    \end{aligned}
\end{equation}
holds for all $t\geq 1$, where $\delta_t = \|{\by}^*(\bbx^t)-\bby^{t}\|^2$ is defined in \eqref{def:deltat}. 
\end{lemma}
\begin{proof}
    Using the facts $\bby^{t+1} = \bby^t+\eta_y\bbd^t_y = \bby^t+\eta_y\bbv^t_y $ and $\grad_y f(\bbx^t,{\by}^*(\bbx^t))=0$, we have that for any $a>0$,
    \begin{equation}\label{eq:bound-y-star-1}
        \begin{aligned}
            & \|{\by}^*(\bbx^t)-\bby^{t+1}\|^2 = \|{\by}^*(\bbx^t)+\eta_y\grad_y f(\bbx^t,{\by}^*(\bbx^t)) - \bby^t - \eta_y\bbv^t\|^2
            \\
            \leq & (1+a)\|{\by}^*(\bbx^t) + \eta_y\grad_y f(\bbx^t,{\by}^*(\bbx^t))-\bby^t -\eta_y \grad_y f(\bbx^t,\bby^t)\|^2 + (\textstyle 1+\frac{1}{a})\eta_y^2\|\grad_y f(\bbx^t,\bby^t)-\bbv^t_y\|^2. 
        \end{aligned}
    \end{equation}
    We bound the first term on the right hand side of \eqref{eq:bound-y-star-1} as follows:
    \begin{equation}\label{eq:bound-y-star-2}
        \begin{aligned}
            &\|{\by}^*(\bbx^t) + \eta_y\grad_y f(\bbx^t,{\by}^*(\bbx^t))-\bby^t -\eta_y 
            \sa{\grad_y 
            f(\bbx^t,\bby^t)}\|^2
            \\
            =& \|{\by}^*(\bbx^t) - \bby^t\|^2 + 2\eta_y\left\langle {\by}^*(\bbx^t) - \bby^t, \grad_y f(\bbx^t,{\by}^*(\bbx^t)) - \grad_y f(\bbx^t,\bby^t)\right\rangle +   \eta_y^2\|\grad_y f(\bbx^t,{\by}^*(\bbx^t)) - \grad_y f(\bbx^t,\bby^t)\|^2\\
            \leq &
            \|{\by}^*(\bbx^t) - \bby^t\|^2 + (2\eta_y-\eta_y^2L)\left\langle {\by}^*(\bbx^t) - \bby^t, \grad_y f(\bbx^t,{\by}^*(\bbx^t)) - \grad_y f(\bbx^t,\bby^t)\right\rangle 
            \\
            \leq &
            (1-2\eta_y\mu+ \eta_y^2\mu L)\sa{\|{\by}^*(\bbx^t) - \bby^t\|^2},
        \end{aligned}
    \end{equation}
    where the first inequality is by \cref{ASPT:smooth-f} and \xz{the concavity of $f(\bx,\cdot)$}; the second inequality is by the strong concavity of \sa{$f(\bx,\cdot)$}.
    Therefore, plugging \eqref{eq:bound-y-star-2} into \eqref{eq:bound-y-star-1}, we have that for any $a>0$,
    \begin{equation}\label{eq:bound-y-star-3}
        \begin{aligned}
            \|{\by}^*(\bbx^t)-\bby^{t+1}\|^2 
            \leq &
            (1+a)(1-2\eta_y\mu+\eta_y^2\mu L)\|{\by}^*(\bbx^t) - \bby^t\|^2 + (1+\frac{1}{a})\eta_y^2\|\grad_y f(\bbx^t,\bby^t)-\bbv^t_y\|^2 
            \\
            \leq & (1+a)\left(1-\eta_y\mu\right)\delta_t + (1+\frac{1}{a})\eta_y^2\|\grad_y f(\bbx^t,\bby^t)-\bbv^t_y\|^2 
            \\
            = & (1-\frac{1}{2}\eta_y\mu)\delta_t
            + \left(\frac{2}{\eta_y\mu} - 1\right)\eta_y^2 \|\grad_y f(\bbx^t,\bby^t)-\bbv^t_y\|^2   
            \\
            \leq & \left(1-\frac{1}{2}\eta_y\mu\right)\delta_t
            + \frac{2\eta_y}{\mu} \|\grad_y f(\bbx^t,\bby^t)-\bbv^t_y\|^2 
        \end{aligned}
    \end{equation}
where \sa{
in the second inequality we use $\eta_y\in(0,1/L]$, and in the equality} we set $a = \frac{\eta_y\mu}{2-2\eta_y\mu}$.
Moreover, we can bound $\|\bbv^t_y-\grad_y f(\bbx^t,\bby^t)\|^2$ as follows:
\begin{equation}\label{eq:bound-diff-f-vy}
    \begin{aligned}
        \|\grad_y f(\bbx^t,\bby^t) - \bbv^t_y\|^2
        = & 
        \left\|\grad_y f(\bbx^t,\bby^t) - \frac{1}{M}\sum_{i=1}^M \grad_y f_i(\bx^t_i,\by^t_i) + \frac{1}{M}\sum_{i=1}^M \grad_y f_i(\bx^t_i,\by^t_i) - \bbv^t_y \right\|^2
        \\
        \leq & 2  \left\|\yy{\frac{1}{M}\sum_{i=1}^M}\grad_y \yy{f_i}(\bbx^t,\bby^t) - \frac{1}{M}\sum_{i=1}^M \grad_y f_i(\bx^t_i,\by^t_i)\right\|^2 + 2\left\|\frac{1}{M}\sum_{i=1}^M \grad_y f_i(\bx^t_i,\by^t_i) - \frac{1}{M}\sum_{i=1}^M \bv^t_{y,i} \right\|^2
        \\
        \leq &\frac{2}{M}\sum_{i=1}^M\Big(
        \|\grad_y \yy{f_i}(\bbx^t,\bby^t) - \grad_y f_i(\bx^t_i,\by^t_i)\|^2
        +
        \|\grad_y f_i(\bx^t_i,\by^t_i) - \bv^t_{y,i}\|^2
        \Big)
        \\
        \leq & \frac{2L^2}{M}( \|X^t_\perp \|^2_F + \|Y^t_\perp \|^2_F) + 
        \frac{2}{M}\sum_{i=1}^M \|\be^t_{y,i}\|^2
    \end{aligned}
\end{equation}
where the last inequality follows from \cref{ASPT:smooth-f}, and $X_\perp$, $Y_\perp$, and $\be^t_{y,i}$ are defined in \eqref{def:e_x,e_y}. Then, plugging \eqref{eq:bound-diff-f-vy} into \eqref{eq:bound-y-star-3}, we have
\begin{equation}
    \label{eq:bound-y-star-4}
    \begin{aligned}
           \|{\by}^*(\bbx^t)-\bby^{t+1}\|^2 
            \leq & \left(1-\frac{1}{2}\eta_y\mu\right)\delta_t
            + \frac{2\eta_y}{\mu}\Big(\frac{2L^2}{M}( \|X^t_\perp \|^2_F + \|Y^t_\perp \|^2_F) + 
        \frac{2}{M}\sum_{i=1}^M \|\be^t_{y,i}\|^2\Big).
    \end{aligned}
\end{equation}

Now, we are ready to show a proper bound for $\delta_t$ for all $t\geq 1$. Indeed, for $\forall a>0$, we have
\begin{equation}\label{eq:bound-deltat-1}
    \begin{aligned}
         & \delta_t =  \|{\by}^*(\bbx^t) - {\by}^*(\bbx^{t-1}) + {\by}^*(\bbx^{t-1}) - \bby^t\|^2\\
        \leq & 
        (1+a)\| {\by}^*(\bbx^{t-1}) - \bby^t \|^2
        +
         \left(1+\frac{1}{a}\right)\|{\by}^*(\bbx^t) - {\by}^*(\bbx^{t-1}) \|^2 
         \\
        \leq & 
        (1+a)\| {\by}^*(\bbx^{t-1}) - \bby^t \|^2
        +
         \left(1+\frac{1}{a}\right)\kappa^2\|\bbx^t - \bbx^{t-1} \|^2 
         \\
        \leq & (1+a)\left(
        \Big(1-\frac{1}{2}\eta_y\mu\Big)\delta_{t-1} + \frac{2\eta_y}{\mu}\frac{2L^2}{M} ( \|X^{t-1}_\perp \|^2_F + \|Y^{t-1}_\perp \|^2_F) + 
        \frac{2\eta_y}{\mu}\frac{2}{M}\sum_{i=1}^M \|\be^{t-1}_{y,i}\|^2
        \right)
        + \left(1+\frac{1}{a}\right)\kappa^2\|\bbx^t - \bbx^{t-1}\|^2,
    \end{aligned}
\end{equation}
where the second inequality is by \xz{\cref{lemma:smooth-Phi}}; the last inequality is by \eqref{eq:bound-y-star-4}.
Next, letting $a = \frac{\eta_y\mu}{4-2\eta_y\mu}$ and using the fact $\bbx^t=\bbx^{t-1}-\eta_x\bbv^{t-1}$, we obtain 
\begin{equation}\label{eq:bound-deltat-2}
    \begin{aligned}
          \delta_t 
        \leq & \left(1-\frac{1}{4}\eta_y\mu\right)\delta_{t-1} 
        + 
        \frac{4-\eta_y\mu}{4-2\eta_y\mu}\frac{2\eta_y}{\mu}\frac{2L^2}{M} \left( \|X^{t-1}_\perp \|^2_F + \|Y^{t-1}_\perp \|^2_F\right) 
        \\
        & +
         \frac{4-\eta_y\mu}{4-2\eta_y\mu}\frac{2\eta_y}{\mu}\frac{2}{M} \sum_{i=1}^M \|\be^{t-1}_{y,i}\|^2
         +
         \left(\frac{4}{\eta_y\mu}-1\right)\kappa^2\eta_x^2\|\bbv^{t-1}_x\|^2.
    \end{aligned}
\end{equation}
In the following part, we first provide \sa{an upper bound on} $\|\bbv^t_x\|^2$, and then use it within \eqref{eq:bound-deltat-2}. 
Indeed, we have
\begin{equation}\label{eq:bound-vtx-1}
\begin{aligned}
     \|\bbv^t_x\|^2 = &\|\bbv^t_x-\grad_x f(\bbx^t,\bby^t) + \grad_x f(\bbx^t,\bby^t) - \grad \Phi(\bbx^t) + \grad \Phi(\bbx^t) \|^2 
     \\
     \leq & 3 \|\bbv^t_x-\grad_x f(\bbx^t,\bby^t)\|^2 + 3\|\grad_x f(\bbx^t,\bby^t) - \grad \Phi(\bbx^t)\|^2 + 3\|\grad \Phi(\bbx^t) \|^2 
    \\
     \leq & 3 \|\bbv^t_x-\grad_x f(\bbx^t,\bby^t)\|^2 + 3L^2\delta_t + 3\|\grad \Phi(\bbx^t) \|^2, 
\end{aligned}
\end{equation}
where the last inequality is by  \cref{ASPT:smooth-f} and \xz{$\Phi(\bx)=\max_{\by\in \mathbb{R}^m}f(\bx,\by)$ is defined in \cref{lemma:smooth-Phi}}; \xz{specifically, the definition of $\Phi(\bx)$ implies $\grad \Phi(\bbx^t)=\grad_x f(\bbx^t,\by^*(\bbx^t))$}. Furthermore, \sa{similar to \cref{eq:bound-diff-f-vy}, we also have}
\begin{equation}\label{eq:bound-diff-f-vtx-1}
    \begin{aligned}
        \|\grad_x f(\bbx^t,\bby^t) - \bbv^t_x\|^2
        \leq  \frac{2L^2}{M}( \|X^t_\perp \|^2_F + \|Y^t_\perp \|^2_F) + 
        \frac{2}{M}\sum_{i=1}^M \|\be^t_{x,i}\|^2.
    \end{aligned}
\end{equation}
If we use \eqref{eq:bound-diff-f-vtx-1} within \eqref{eq:bound-vtx-1}, it follows that
\begin{equation}\label{eq:bound-vtx-2}
\begin{aligned}
     \|\bbv^t_x\|^2
     \leq & \frac{6L^2}{M}( \|X^t_\perp \|^2_F + \|Y^t_\perp \|^2_F) + 
        \frac{6}{M}\sum_{i=1}^M \|\be^t_{x,i}\|^2 + 3L^2\delta_t + 3\|\grad \Phi(\bbx^t) \|^2. 
\end{aligned}
\end{equation}
Next, using \eqref{eq:bound-vtx-2} within \eqref{eq:bound-deltat-2} gives
\begin{equation}\label{eq:bound-deltat-3}
    \begin{aligned}
          \delta_t 
        \leq & \Big(1-\frac{1}{4}\eta_y\mu\Big)\delta_{t-1} 
        + 
        \frac{4-\eta_y\mu}{4-2\eta_y\mu}\frac{2\eta_y}{\mu}\frac{2L^2}{M} ( \|X^{t-1}_\perp \|^2_F + \|Y^{t-1}_\perp \|^2_F) 
        +
         \frac{4-\eta_y\mu}{4-2\eta_y\mu}\frac{2\eta_y}{\mu}\frac{2}{M} \sum_{i=1}^M \|\be^{t-1}_{y,i}\|^2
         \\
         &
         +
         \Big(\frac{4}{\eta_y\mu}-1\Big)\kappa^2\eta_x^2
         \Big(
         \frac{6L^2}{M}( \|X^{t-1} _\perp \|^2_F + \|Y^{t-1} _\perp \|^2_F) + 
        \frac{6}{M}\sum_{i=1}^M \|\be^{t-1} _{x,i}\|^2 + 3L^2\delta_{t-1}  + 3\|\grad \Phi(\bbx^{t-1} ) \|^2 
         \Big)
         \\
        = & \Big(1-\frac{1}{4}\eta_y\mu + 3\Big(\frac{4}{\eta_y\mu}-1\Big)\kappa^2\eta_x^2 L^2\Big)\delta_{t-1} 
        + 
         \Big(\frac{4-\eta_y\mu}{4-2\eta_y\mu}\frac{2\eta_y}{\mu}+
           3\Big(\frac{4}{\eta_y\mu}-1\Big)\kappa^2\eta_x^2\Big)
           \frac{2L^2}{M}
           ( \|X^{\yy{t-1}}_\perp \|^2_F + \|Y^{\yy{t-1}}_\perp \|^2_F) 
        \\
        & +
         \frac{4-\eta_y\mu}{4-2\eta_y\mu}\frac{2\eta_y}{\mu}
           \frac{2}{M} \sum_{i=1}^M \|\be^{\sa{t-1}}_{y,i}\|^2
        \sa{+\Big(\frac{4}{\eta_y\mu}-1\Big)\kappa^2\eta_x^2 \frac{6}{M} \sum_{i=1}^M \|\be^{t-1}_{x,i}\|^2}
         +3\Big(\frac{4}{\eta_y\mu}-1\Big)\kappa^2\eta_x^2
         \|\grad \Phi(\bbx^{t-1})\|^2.
         \\
    \end{aligned}
\end{equation}
\xz{Now using the fact $\|Z^t_\perp \|^2_F = \|X^t_\perp \|^2_F + \|Y^t_\perp \|^2_F$ completes the proof.}
\end{proof}
\xz{We temporarily stop the analysis of $\delta_t$ here. In the subsection 
\xzref{bound of the dual suboptimality}
we will continue to investigate the bound of $\delta_t$. The reason why we presented this part of the proof first is that it includes many technical equations that are used in other parts of the analysis. Next, we will proceed to analyze the error of the stochastic gradient oracles $\|E^t\|^2_F$.}

\subsection{Bound of
Stochastic Gradient Estimate}\label{sec:bound-Et}
In this subsection, we establish a suitable upper bound for  $\bE[\|E^t\|^2_F]$. The bound will be frequently used in the analysis of $\{\grad \Phi(\bbx^t), \delta_t, \|Z^t_\perp\|^2_F, \|D^t_\perp\|^2_F\}$.
\begin{lemma}\label{lemma:bound-eti}
Suppose Assumptions~\ref{ASPT:SC},~\ref{ASPT:general-sto-grad},~\ref{ASPT:smooth-F-mean-squared}~and~\ref{ASPT:mixture-matrix} hold. Given $q\geq 1$, the inequality
    \begin{equation}
        \label{eq:bound-ei-3}
        \begin{aligned}
            \bE[\|\be^{t}_i\|^2] 
            \leq &
            \bE[\|\be^{t-1}_i\|^2]
            +  \frac{L^2}{S_2}\bE[\|\bz^{t}_i-\bz^{t-1}_i\|^2],\quad \sa{\forall~i\in [M]}
        \end{aligned}
    \end{equation}
holds for all $t\geq 1$ \sa{such that} \sa{$\rm{mod}(t,q)\neq 0$}. Moreover, if $\rm{mod}(t,q)=0$, then $ \bE[\|\be^{t}_i\|^2] 
            \leq \frac{\sigma^2}{S_1}$.
\end{lemma}
\begin{proof}
    \sa{Recall that $\be_i^t= \bv_i^t-\grad f_i(\bx_i^t,\by_i^t)$ for $i=1,\ldots,M$ and \xz{$\bv_i^t$ is defined in \cref{def:matrices-1}}. Thus,}
    \sa{Given $i\in [M]$ and $t\geq 1$ such that $\rm{mod}(t, q) > 0$, it follows from the definition of $\bv^t_i$ and $\be^t_i$ that}
    \begin{equation}
        \label{eq:bound-ei-1}
        \begin{aligned}
            \bE\left[\left\|\|\be^t_i\|^2] = \right.\right.& \bE \Bigg[\Bigg\|\frac{1}{|\gmb{\cB_{i}^{t}}|}\sum_{\xi\in \gmb{\cB_{i}^{t}}}\Big(\xz{\tgrad f_i}(\bx^t_i,\by^t_i;\xi) - \xz{\tgrad f_i}(\bx^{t-1}_i,\by^{t-1}_i;\xi) \Big) + \bv^{t-1}_i - \grad f_i(\bx^t_i,\by^t_i))\Bigg\|^2\Bigg]
            \\
            = & \bE \left[\left\|\frac{1}{|\gmb{\cB_{i}^{t}}|}\sum_{\xi\in \gmb{\cB_{i}^{t}}}\Big(\xz{\tgrad f_i}(\bx^t_i,\by^t_i;\xi) - \xz{\tgrad f_i}(\bx^{t-1}_i,\by^{t-1}_i;\xi) \Big) + \be^{t-1}_i - \grad f_i(\bx^t_i,\by^t_i) +
             \grad f_i(\bx^{t-1}_i,\by^{t-1}_i)\right\|^2\right]
            \\
            = &
            \bE[\|\be^{t-1}_i\|^2] + \bE\left[\left\|\frac{1}{|\gmb{\cB_{i}^{t}}|}\sum_{\xi\in \gmb{\cB_{i}^{t}}}\Big(\xz{\tgrad f_i}(\bx^t_i,\by^t_i;\xi) - \xz{\tgrad f_i}(\bx^{t-1}_i,\by^{t-1}_i;\xi) \Big)  - \grad f_i(\bx^t_i,\by^t_i) + \grad f_i(\bx^{t-1}_i,\by^{t-1}_i) \right\|^2\right]
            \\ 
            = &
            \bE[\|\be^{t-1}_i\|^2]
            +
            \frac{1}{|\gmb{\cB_{i}^{t}}|^2}\sum_{\xi\in \gmb{\cB_{i}^{t}}}\bE\left[\|\xz{\tgrad f_i}(\bx^t_i,\by^t_i;\xi) - \xz{\tgrad f_i}(\bx^{t-1}_i,\by^{t-1}_i;\xi)   - \grad f_i(\bx^t_i,\by^t_i) + \grad f_i(\bx^{t-1}_i,\by^{t-1}_i) \|^2\right].
        \end{aligned}
    \end{equation}
    The last two equalities above follow from the unbiasedness of the stochastic oracle in Assumption~\ref{ASPT:general-sto-grad} and \sa{the independence of the elements in $\gmb{\cB_{i}^{t}}$}, 
    which implies
    $$
   \bE[ \xz{\tgrad f_i}(\bx^t_i,\by^t_i;\xi) - \xz{\tgrad f_i}(\bx^{t-1}_i,\by^{t-1}_i;\xi)   - \grad f_i(\bx^t_i,\by^t_i) + \grad f_i(\bx^{t-1}_i,\by^{t-1}_i)] = 0, \;\forall \xi\in \gmb{\cB_{i}^{t}}.
    $$
     Furthermore, \sa{since} for any
given random variable $\zeta$ with the finite second order moment, 
$\bE[\|\zeta - \bE[\zeta]\|^2]\leq \bE[\|\zeta\|^2]$ \sa{holds, invoking this inequality for $\zeta = \xz{\tgrad f_i}(\bx^t_i,\by^t_i;\xi) - \xz{\tgrad f_i}(\bx^{t-1}_i,\by^{t-1}_i;\xi)$, \eqref{eq:bound-ei-1} implies that}
\begin{equation}
        \label{eq:bound-ei-2}
        \begin{aligned}
            \bE[\|\be^t_i\|^2] 
            \leq &
            \bE[\|\be^{t-1}_i\|^2]
            +
            \frac{1}{|\gmb{\cB_{i}^{t}}|^2}\sum_{\xi\in \gmb{\cB_{i}^{t}}}\bE[\|\xz{\tgrad f_i}(\bx^t_i,\by^t_i;\xi) - \xz{\tgrad f_i}(\bx^{t-1}_i,\by^{t-1}_i;\xi)   \|^2] 
            \\
            \leq &
            \bE[\|\be^{t-1}_i\|^2]
            +  \frac{L^2}{|\gmb{\cB_{i}^{t}}|}\bE[\|\bx^{t}_i-\bx^{t-1}_i\|^2+\|\by^{t}_i-\by^{t-1}_i\|^2]
            \\
            = & \bE[\|\be^{t-1}_i\|^2]
            +  \frac{L^2}{|\gmb{\cB_{i}^{t}}|}\bE[\|\bz^{t}_i-\bz^{t-1}_i\|^2],
        \end{aligned}
    \end{equation}
    where the last inequality follows from \cref{ASPT:smooth-F-mean-squared}. \sa{Since for $\rm{mod}(t,q)>0$ we set $|\gmb{\cB_{i}^{t}}|=S_2$ for all $i\in[M]$, \eqref{eq:bound-ei-3} follows immediately.} 
    
    \sa{On the other hand,} given $i\in[M]$, when $\rm{mod}(t,q)=0$, it directly follows from the definition of $\be^t_i$ and Assumption~\ref{ASPT:general-sto-grad} 
    that $ \bE[\|\be^{t}_i\|^2] 
            \leq \frac{\sigma^2}{S_1}$, which completes the proof.
\end{proof}

\begin{lemma}\label{lemma:bound-ZF}
Suppose Assumptions~\ref{ASPT:SC},~\ref{ASPT:general-sto-grad},~\ref{ASPT:smooth-F-mean-squared}~and~\ref{ASPT:mixture-matrix} hold, and \xzrr{$\eta_x\leq\eta_y$}, then the inequality
    \begin{equation}\label{eq:bound-ZF-3}
    \begin{aligned}
         \xz{\bE[}\|Z^{t+1}-Z^t\|^2_F\xz{]} 
        \leq \frac{S_2}{L^2}\Delta_t + \xzrr{24\eta_y^2}\bE[\|E^t\|_F^2]
    \end{aligned}
\end{equation}
\sa{holds for all $t\geq 0$,} where \xz{we define $E^t  = [{\be^{t}_1}, \be^{t}_2, ..., \be^{t}_M]^\top$, and} 
    \begin{equation}
    \label{eq:def-Delta-t}
    \begin{aligned}
        \Delta_t = \frac{L^2}{S_2}\bE\Big[\Big( \xz{8}+\xzrr{40\eta_y^2L^2}  \Big)\|Z^t_\perp \|^2_F + 4\eta_y^2 \|D^t_\perp \|^2_F 
        + \xzrr{20\eta_y^2L^2M\delta_t} + 12M\eta_x^2\|\grad \Phi(\bbx^t)\|^2
        \Big].
    \end{aligned}
    \end{equation}
\end{lemma}
\begin{proof}
    From the update \xz{in \cref{Alg:main-alg}} and \cref{ASPT:mixture-matrix}, we have 
    \begin{equation*}
        \begin{aligned}
            &\|X^{t+1}-X^t\|_F^2 = \|WX^t-X^t-\eta_xD^t_x\|^2_F
            = \|(W-I)(X^t-\Bar{X}^t)-\eta_xD^t_x\|^2
            \\
            \leq
            &
            2\|(W-I)X^t_{\perp}\|_F^2 + 2\eta_x^2\|D^t_x\|^2_F
            \leq
            \xz{8}\|X^t_{\perp}\|_F^2 + 2\eta_x^2\|D^t_x\|^2_F.
        \end{aligned}
    \end{equation*}%
    Moreover, by \xz{Young's inequality} and the fact that $D^t_x = \Bar{D}^t_x + D^t_{x\perp}$ and $\Bar{D}^t_x =\Bar{V}^t_x$, we have
    \begin{equation}
        \label{eq:bound-XF}
        \|X^{t+1}-X^t\|_F^2  \leq \xz{8}\|X^t_{\perp}\|_F^2 + 4\eta_x^2\|\Bar{D}^t_x\|^2_F + 4\eta_x^2\|D^t_{x\perp}\|^2_F 
        \sa{=} \xz{8}\|X^t_{\perp}\|_F^2 + 4\eta_x^2\|\Bar{V}^t_x\|^2_F + 4\eta_x^2\|D^t_{x\perp}\|^2_F.
    \end{equation}
    Similarly, we have
    \begin{equation}
    \label{eq:bound-YF}
        \begin{aligned}
             \|Y^{t+1}-Y^t\|_F^2 \leq \xz{8}\|Y^t_{\perp}\|_F^2 + 4\eta_y^2\|\Bar{V}^t_y\|^2_F + 4\eta_y^2\|D^t_{y\perp}\|^2_F.
        \end{aligned}
    \end{equation}
Therefore, combining \cref{eq:bound-XF,eq:bound-YF}, we have 
\begin{equation}\label{eq:bound-ZF-1}
\begin{aligned}
     & \|Z^{t+1}-Z^t\|^2_F = \|X^{t+1}-X^t\|_F^2 +  \|Y^{t+1}-Y^t\|_F^2
     \\
     \leq &  \xz{8}\|X^t_{\perp}\|_F^2 
     + 4\eta_x^2\|D^t_{x\perp}\|^2_F
     + \xz{8}\|Y^t_{\perp}\|_F^2 
     + 4\eta_y^2\|D^t_{y\perp}\|^2_F
     + 4\eta_x^2\|\Bar{V}^t_x\|^2_F
     + 4\eta_y^2\|\Bar{V}^t_y\|^2_F 
     \\
     \leq & \xz{8}\|Z^t_{\perp}\|_F^2 
     + 4\eta_y^2\|D^t_{\perp}\|^2_F
     + 4\eta_x^2\|\Bar{V}^t_x\|^2_F
     + 4\eta_y^2\|\Bar{V}^t_y\|^2_F, 
     \end{aligned}  
\end{equation}
where we use the condition $\eta_x\leq \eta_y$ in the last inequality. 
\sa{Next,} we will bound $\|\Bar{V}^t_x\|^2_F$ and $\|\Bar{V}^t_y\|^2_F$ separately. Specifically, it follows from \eqref{eq:bound-vtx-2} that
\begin{equation}\label{eq:bound-vtx-3}
\begin{aligned}
     \|\bar{V}^t_x\|_F^2
     \leq & 6L^2\|Z^t_\perp \|^2_F+ 
        6\sum_{i=1}^M \|\be^t_{x,i}\|^2 + 3ML^2\delta_t + 3M\|\grad \Phi(\bbx^t) \|^2. 
\end{aligned}
\end{equation}
On the other hand, we can bound \sa{$\|\Bar{V}^t_y\|_{F}^2$} as follows
\xzrr{
\begin{equation}\label{eq:bound-vty-1}
    \begin{aligned}
        &\|\bar{V}^t_y\|_F^2 = M \|\bbv^t\|^2\\
        =& M \|\bbv^t - \grad_y f(\bbx^t,\bby^t) + \grad_y f(\bbx^t,\bby^t) - \grad_y  f(\bbx^t,\by_*(\bbx^t)) \|^2 \\
        \leq & 2M \Big( \|\bbv^t - \grad_y f(\bbx^t,\bby^t) \|^2+ \|\grad_y f(\bbx^t,\bby^t) - \grad_y  f(\bbx^t,\by_*(\bbx^t)) \|^2\Big)\\
        \leq &2L^2M\delta_t + 4L^2\|Z^t_\perp\|^2 + 4\sum_{i=1}^M\|\be^t_{y,i}\|^2
    \end{aligned}
\end{equation}
}%
\xzrr{where the last inequality follows from \cref{eq:bound-diff-f-vy} and \cref{ASPT:smooth-f}.}
Therefore, if we use \cref{eq:bound-vtx-3,eq:bound-vty-1} within \cref{eq:bound-ZF-1} and then use the condition \xzrr{$\eta_x\leq\eta_y$}, we obtain that
\begin{equation}
\label{eq:bound-ZF-2}
    \begin{aligned}
         \|Z^{t+1}-Z^t\|^2_F 
        \leq&
        \xz{8}\|Z^t_{\perp}\|_F^2 
     + 4\eta_y^2\|D^t_{\perp}\|^2_F 
     \\
     & +
     24\eta_x^2L^2\|Z^t_\perp \|^2_F + 
        24\sa{\eta_x^2}\sum_{i=1}^M \|\be^t_{x,i}\|^2 + 12\eta_x^2ML^2\delta_t + 12\eta_x^2M\|\grad \Phi(\bbx^t) \|^2 
        \\
        &
        \xzrr{
            +  8\eta_y^2L^2M\delta_t + 16\eta_y^2L^2\|Z^t_\perp\|^2 + 16\eta_y^2\sum_{i=1}^M\|\be^t_{y,i}\|^2
        }
        \\
        \leq &
            \Big( \xz{8}+\xzrr{40\eta_y^2L^2}  \Big)\|Z^t_\perp \|^2_F + 4\eta_y^2 \|D^t_\perp \|^2_F + \xzrr{24\eta_y^2}\|E^t\|_F^2
        \\
        &
        + \xzrr{20\eta_y^2L^2M\delta_t} + 12M\eta_x^2\|\grad \Phi(\bbx^t)\|^2.
    \end{aligned}
\end{equation}
\xz{Now taking the expectation of the above inequality} completes the proof.
\end{proof}

\begin{lemma}\label{lemma:bound-Et}
    Suppose Assumptions~\ref{ASPT:SC},~\ref{ASPT:general-sto-grad},~\ref{ASPT:smooth-F-mean-squared}~and~\ref{ASPT:mixture-matrix} hold and \xzrr{$S_2\geq 24L^2\eta_y^2$}. Then, for $q\geq 1$, the inequality
    \begin{equation}\label{eq:bound-Et-1}
                 \bE[\norm{E^t}^2_F] \leq \frac{2M\sigma^2}{S_1} + \sum_{j=nq}^{t-1}(1+c_1)^{t-1-j} \Delta_{j}
    \end{equation}
    holds for all $t\geq 1$ such that $mod(t,q)>0$, and $n\in\mathbb{N}$ such that $nq<t<(n+1)q$, where 
     \xzrr{$c_1 = \frac{24\eta_y^2L^2}{S_2}$}, 
     $E^t$ is defined in \cref{def:Et-et} and $\Delta_t$ is defined in \cref{eq:def-Delta-t}. 
    Moreover, if $mod(t,q)=0$, then $\bE[\|E^t\|^2_F]\leq \frac{M\sigma^2}{S_1}$.
\end{lemma}
\begin{proof}
    Given $t\geq 1$, when $mod(t,q)=0$, it directly follows from \cref{lemma:bound-eti} and \cref{ASPT:general-sto-grad} that $\bE[\|E^t\|^2_F]\leq \frac{M\sigma^2}{S_1}$. Moreover, when $mod(t,q)>0$,
    we take $n\in \mathbb{N}^+$ such that $nq<t<(n+1)q$. If we invoke  \cref{lemma:bound-ZF} within \cref{lemma:bound-eti}, we obtain that     
    \begin{equation}\label{eq:bound-Et-2}
        \bE[\|E^t\|^2_F]\leq (1+\xzrr{\frac{24\eta_y^2 L^2}{S_2}}) \bE[\|E^{t-1}\|^2_F] + \Delta_{t-1},\quad \forall\; \sa{t:}\ nq<t < (n+1)q.
    \end{equation}
    If we apply \cref{eq:bound-Et-2} recursively from $nq+1$ to $t$, it follows that
    \begin{equation*}
        \begin{aligned}
            \bE[\norm{E^t}^2_F] \leq & \Delta_{t-1} + (1+c_1)\Delta_{t-2} + \cdots + (1+c_1)^{t-nq-1}\Delta_{nq} + (1+c_1)^{t-nq} \bE[\norm{E^{nq}}^2_F] 
            \\
            = &\sum_{j=nq}^{t-1}(1+c_1)^{t-1-j} \Delta_{j} + (1+c_1)^{t-nq} \bE[\norm{E^{nq}}^2_F] 
            \\
            \leq & \sum_{j=nq}^{t-1}(1+c_1)^{t-1-j} \Delta_{j} + (1+c_1)^{t-nq} M\frac{\sigma^2}{S_1},\quad \forall\; \sa{t:}\ nq<t < (n+1)q,
        \end{aligned}
    \end{equation*}%
    where \xzrr{$c_1 = \frac{24\eta_y^2 L^2}{S_2}$}. 
    If we pick $q \sa{\leq} \lfloor 1/(2c_1) \rfloor$, then  $(1+c_1)^q\leq (1+c_1)^{1/(2c_1)}\leq2$ \sa{since $1/c_1\geq 1$ for 
    \xzrr{$\eta_y\leq\sqrt{\frac{S_2}{24}}\cdot\frac{1}{L} $}}, it follows that
    \begin{equation*}
         \bE[\norm{E^t}^2_F] \leq \frac{2M\sigma^2}{S_1} + \sum_{j=nq}^{t-1}(1+c_1)^{t-1-j} \Delta_{j},\quad \forall\; \sa{t:}\ nq<t < (n+1)q,
    \end{equation*}
    which completes the proof.
\end{proof}


\subsection{Fundamental Inequality}\label{sec:fundamentalresult}
\xz{In this subsection, we display the fundamental analysis of the sequences $\{\Phi(\bbx^t)\}$ and $\{\grad\Phi(\bbx^t)\}$. The analysis in \cref{lemma:bound-Phi-final} will be utilized to derive the final convergence result in \cref{thm:main-convergence} by constructing a telescoping sum.}
\begin{lemma}\label{lemma:bound-Phi-1}
Suppose \xz{\cref{ASPT:smooth-f} and \cref{ASPT:SC}} hold. Then the inequality
    \begin{equation}
    \label{eq:bound-phi-3}
    \begin{aligned}
        \Phi(\bbx^{t+1})  
         \leq  & \Phi(\bbx^t) - (\frac{\eta_x}{2} - \frac{3(\kappa +1)L\eta_x^2}{2})\|\grad \Phi(\bbx^t)\|^2
         +\Big(\eta_x + \frac{3(\kappa +1)L\eta_x^2}{2}\Big)\Big(L^2\delta_t + 
        \frac{2}{M}\sum_{i=1}^M \|\be^t_{x,i}\|^2\Big)
        \\
         &  + \frac{2L^2}{M}\Big(\eta_x + \frac{3(\kappa +1)L\eta_x^2}{2}\Big) \|Z^t_\perp \|^2_F 
    \end{aligned}
\end{equation}
holds for all $t\geq 0$.
\end{lemma}
\begin{proof}
For any $t\geq 0$, it follows from \cref{lemma:smooth-Phi} that
\begin{equation}\label{eq:bound-phi-1}
    \begin{aligned}
        \Phi(\bbx^{t+1}) & \leq \Phi(\bbx^{t}) + \langle \grad \Phi(\bbx^{t}), \bbx^{t+1} - \bbx^t \rangle   + \frac{(\kappa+1)L}{2}\|\bbx^{t+1} - \bbx^{t}\|^2 \\
        & = \Phi(\bbx^t) - \eta_x \|\grad \Phi(\bbx^t)\|^2 + \frac{(\kappa+1)L\eta_x^2}{2}\|\bbv^t_x\|^2 + \eta_x \langle \grad \Phi(\bbx^t) - \bbv^t_x, \grad \Phi(\bbx^t)\rangle. 
    \end{aligned}
\end{equation}
We bound the last inner product term in the above inequality as follows
\begin{equation}\label{eq:bound-diff-grad-1}
    \begin{aligned}
        \langle \grad \Phi(\bbx^t)-\bbv^t_x, \Phi(\bbx^t)\rangle
        = &  \langle \grad \Phi(\bbx^t)-\grad_x f(\bbx^t,\bby^t), \Phi(\bbx^t)\rangle + \langle \sa{\grad_x}f(\bbx^t,\bby^t) - \bbv^t_x, \Phi(\bbx^t)\rangle
        \\
        \leq & \frac{1}{2}\|\grad \Phi(\bbx)^t\|^2 + \|\grad \Phi(\bbx^t) - \grad_x f(\bbx^t,\bby^t)\|^2 + \|\grad_x f(\bbx^t,\bby^t) - \bbv^t_x\|^2,
    \end{aligned}
\end{equation}
where we use Young's inequality for \cref{eq:bound-diff-grad-1}, i.e.,  $\langle \ba,\bb \rangle\leq \frac{1}{4}\|\ba\|^2 + \|\bb\|^2,\;\forall \ba,\bb\in\mR^d$. 
Then, using \cref{eq:bound-vtx-1,eq:bound-diff-grad-1} 
within \cref{eq:bound-phi-1} leads to
\begin{equation}\label{eq:bound-phi-2}
    \begin{aligned}
        \Phi(\bbx^{t+1})  \leq  & \Phi(\bbx^t) - (\frac{\eta_x}{2} - \frac{3(\kappa +1)L\eta_x^2}{2})\|\grad \Phi(\bbx^t)\|^2
        \\
        & + \Big(\eta_x + \frac{3(\kappa +1)L\eta_x^2}{2}\Big) (\|\bbv^t_x-\grad_x f(\bbx^t,\bby^t)\|^2 + \|\grad_x f(\bbx^t,\bby^t) - \grad \Phi(\bbx^t)\|^2 )
        \\
         \leq  & \Phi(\bbx^t) - (\frac{\eta_x}{2} - \frac{3(\kappa +1)L\eta_x^2}{2})\|\grad \Phi(\bbx^t)\|^2
        \\
         &  + \Big(\eta_x + \frac{3(\kappa +1)L\eta_x^2}{2}\Big) \Big(\|\bbv^t_x-\grad_x f(\bbx^t,\bby^t)\|^2 + L^2\delta_t\Big),
    \end{aligned}
\end{equation}
where  the last inequality uses \cref{ASPT:smooth-f} and recalls that $\delta_t = \|{\by}^*(\bbx^t)-\bby^t\|^2$ and $\grad \Phi(\bbx^t)=\grad_x f(\bbx^t,\by^*(\bbx^t))$. 
Moreover, we can bound $\|\bbv^t_x-\grad_x f(\bbx^t,\bby^t)\|^2$ by \cref{eq:bound-diff-f-vtx-1}.
Next, if we plug \cref{eq:bound-diff-f-vtx-1} into \cref{eq:bound-phi-2}, we have that
\begin{equation*}
    \begin{aligned}
        \Phi(\bbx^{t+1})  
         \leq  & \Phi(\bbx^t) - (\frac{\eta_x}{2} - \frac{3(\kappa +1)L\eta_x^2}{2})\|\grad \Phi(\bbx^t)\|^2
         +\Big(\eta_x + \frac{3(\kappa +1)L\eta_x^2}{2}\Big)\Big(L^2\delta_t + 
        \frac{2}{M}\sum_{i=1}^M \|\be^t_{x,i}\|^2\Big)
        \\
         &  + \frac{2L^2}{M}\Big(\eta_x + \frac{3(\kappa +1)L\eta_x^2}{2}\Big) ( \|X^t_\perp \|^2_F + \|Y^t_\perp \|^2_F).
    \end{aligned}
\end{equation*}
Then using the fact $\|Z^t_\perp \|^2_F = \|X^t_\perp \|^2_F + \|Y^t_\perp \|^2_F$ completes the proof.
\end{proof}

\begin{lemma}\label{lemma:bound-Phi-final}
 Suppose Assumptions~\ref{ASPT:SC},~\ref{ASPT:general-sto-grad},~\ref{ASPT:smooth-F-mean-squared}~and~\ref{ASPT:mixture-matrix} hold. Then the inequality 
\begin{equation}\label{eq:bound-Phi-final}
    \begin{aligned}
        \bE[\Phi(\bbx^{(n+1)q})  ]
         \leq  & \bE\Big[\Phi(\bbx^{nq}) 
         + \Big(\eta_x + \frac{3(\kappa +1)L\eta_x^2}{2}\Big)\frac{4q\sigma^2}{S_1}
         - c_\Phi\sum_{t=nq}^{(n+1)q-1}\|\grad \Phi(\bbx^t)\|^2
          + c_\delta\sum_{t=nq}^{(n+1)q-1}\delta_t
          \\
          &
          + c_Z \sum_{t=nq}^{(n+1)q-1}\|Z^t_\perp \|^2_F
         +\frac{4(q-1)}{M}\Big(\eta_x + \frac{3(\kappa +1)L\eta_x^2}{2}\Big)\frac{L^2}{S_2}4\eta_y^2\sum_{t=nq}^{(n+1)q-2}\|D^t_\perp\|_F^2\Big],
    \end{aligned}
\end{equation}
holds for all $n\in\mathbb{N}$, where
\begin{equation}\label{eq:def-c-coeff}
    \begin{aligned}
        & c_\Phi \triangleq  (\frac{\eta_x}{2} - \frac{3(\kappa +1)L\eta_x^2}{2})
        -\frac{4(q-1)}{M}\Big(\eta_x + \frac{3(\kappa +1)L\eta_x^2}{2}\Big)\frac{L^2}{S_2}12M\eta_x^2,
        \\
        & c_\delta \triangleq \Big(\eta_x + \frac{3(\kappa +1)L\eta_x^2}{2}\Big)L^2 + \frac{4(q-1)}{M}\Big(\eta_x + \frac{3(\kappa +1)L\eta_x^2}{2}\Big)\frac{L^2}{S_2}\xzrr{20\eta_y^2L^2M},
        \\
        & c_Z \triangleq \frac{2L^2}{M}\Big(\eta_x + \frac{3(\kappa +1)L\eta_x^2}{2}\Big)
        +
        \frac{4(q-1)}{M}\Big(\eta_x + \frac{3(\kappa +1)L\eta_x^2}{2}\Big)\frac{L^2}{S_2}\xzrr{(8+40\eta_y^2L^2)}.
    \end{aligned}
\end{equation}
\end{lemma}
\begin{proof}
    For given $t\geq 0$ \sa{such that} $mod(t,q)>0$, we take $n\in \mathbb{N}$ \sa{such that} $nq<t<(n+1)q$. Then it follows from \cref{lemma:bound-Phi-1} and \cref{lemma:bound-Et} that
    \begin{equation}\label{eq:bound-phi-4}
    \begin{aligned}
        \bE[\Phi(\bbx^{t+1})  ]
         \leq  & \bE\Big[\Phi(\bbx^t) - (\frac{\eta_x}{2} - \frac{3(\kappa +1)L\eta_x^2}{2})\|\grad \Phi(\bbx^t)\|^2
         \\
         & +\Big(\eta_x + \frac{3(\kappa +1)L\eta_x^2}{2}\Big)\Big(L^2\delta_t + 
        \frac{4\sigma^2}{S_1} + \frac{2}{M}\sum_{j=nq}^{t-1}(1+c_1)^{t-1-j} \Delta_{j} \Big)
        \\
         &  + \frac{2L^2}{M}\Big(\eta_x + \frac{3(\kappa +1)L\eta_x^2}{2}\Big) \|Z^t_\perp \|^2_F\Big].
    \end{aligned}
\end{equation}
When $t\geq 0$ and $mod(t,q)=0$, we assume $t=nq$ for some $n\in\mathbb{N}$. Then \cref{eq:bound-phi-4} is also satisfied for $t=nq$ according to \cref{lemma:bound-Et}. Therefore, if we sum up \cref{eq:bound-phi-4} over $t=nq$ to $(n+1)q-1$, we obtain
\begin{equation}\label{eq:bound-phi-5}
    \begin{aligned}
       \bE[ \Phi(\bbx^{(n+1)q}) ] 
         \leq  & \bE\Big[\Phi(\bbx^{nq}) 
         + \Big(\eta_x + \frac{3(\kappa +1)L\eta_x^2}{2}\Big)\frac{4q\sigma^2}{S_1}
         - (\frac{\eta_x}{2} - \frac{3(\kappa +1)L\eta_x^2}{2})\sum_{t=nq}^{(n+1)q-1}\|\grad \Phi(\bbx^t)\|^2
         \\
         & + \Big(\eta_x + \frac{3(\kappa +1)L\eta_x^2}{2}\Big)L^2\sum_{t=nq}^{(n+1)q-1}\delta_t
          + \frac{2L^2}{M}\Big(\eta_x + \frac{3(\kappa +1)L\eta_x^2}{2}\Big) \sum_{t=nq}^{(n+1)q-1}\|Z^t_\perp \|^2_F\Big]
         \\
         & +\frac{2}{M}\Big(\eta_x + \frac{3(\kappa +1)L\eta_x^2}{2}\Big)\sum_{t=nq}^{(n+1)q-1}\sum_{j=nq}^{t-1}(1+c_1)^{t-1-j} \Delta_{j}
    \end{aligned}
\end{equation}
for all $n\in\mathbb{N}$. 
Note that
\begin{equation}\label{eq:bound-double-sum}
\begin{aligned}
        & \sum_{t=nq}^{(n+1)q-1}\sum_{j=nq}^{t-1}(1+c_1)^{t-1-j} \Delta_{j} =\sum_{j=nq}^{(n+1)q-2}\sum_{t=j+1}^{(n+1)q-1}(1+c_1)^{t-1-j} \Delta_{j}
        \\
   \leq & \sum_{j=nq}^{(n+1)q-2}(q-1)(1+c_1)^{q-2} \Delta_{j} \leq 2(q-1)\sum_{j=nq}^{(n+1)q-2} \Delta_j,
\end{aligned}
\end{equation}
where the last inequality is by $q\leq\frac{1}{2c_1}$ and \cref{lemma:parameter-ineq}.
Therefore, if we use \cref{eq:bound-double-sum} within \cref{eq:bound-phi-5}, we obtain that
\begin{equation}\label{eq:bound-phi-6}
    \begin{aligned}
        \bE[\Phi(\bbx^{(n+1)q}) ] 
         \leq  & \bE\Big[\Phi(\bbx^{nq}) 
         + \Big(\eta_x + \frac{3(\kappa +1)L\eta_x^2}{2}\Big)\frac{4q\sigma^2}{S_1}
         - (\frac{\eta_x}{2} - \frac{3(\kappa +1)L\eta_x^2}{2})\sum_{t=nq}^{(n+1)q-1}\|\grad \Phi(\bbx^t)\|^2
         \\
         & + \Big(\eta_x + \frac{3(\kappa +1)L\eta_x^2}{2}\Big)L^2\sum_{t=nq}^{(n+1)q-1}\delta_t
          + \frac{2L^2}{M}\Big(\eta_x + \frac{3(\kappa +1)L\eta_x^2}{2}\Big) \sum_{t=nq}^{(n+1)q-1}\|Z^t_\perp \|^2_F\Big]
         \\
         & +\frac{4(q-1)}{M}\Big(\eta_x + \frac{3(\kappa +1)L\eta_x^2}{2}\Big)\sum_{t=nq}^{(n+1)q-2}\Delta_{t}
    \end{aligned}
\end{equation}
holds for all $n\in\mathbb{N}$. Furthermore, \sa{using $\Delta_t$ defined in \cref{eq:def-Delta-t}, i.e.,}
\xzrr{\begin{equation*}
    \begin{aligned}
        \Delta_t =  \frac{L^2}{S_2}\bE\Big[\Big( \xz{8}+\xzrr{40\eta_y^2L^2}  \Big)\|Z^t_\perp \|^2_F + 4\eta_y^2 \|D^t_\perp \|^2_F 
        + \xzrr{20\eta_y^2L^2M\delta_t} + 12M\eta_x^2\|\grad \Phi(\bbx^t)\|^2
        \Big],
    \end{aligned}
    \end{equation*}}
\sa{within \cref{eq:bound-phi-6},} we obtain \cref{eq:bound-Phi-final} and complete the proof.
\end{proof}

\subsection{Bound of the Dual Suboptimality}\label{sec:Bound-of-the-Dual-Suboptimality}
\xz{In this subsection, we display the proper bound of the measure of suboptimality $\bE[\delta_t]$. \sa{The result of this analysis will be 
combined with \cref{lemma:bound-Phi-final}} to derive the final convergence result in \cref{thm:main-convergence} by constructing a telescoping sum.}
\begin{lemma}
\label{lemma:bound-deltat-final}
Suppose Assumptions~\ref{ASPT:SC},~\ref{ASPT:general-sto-grad},~\ref{ASPT:smooth-F-mean-squared}~and~\ref{ASPT:mixture-matrix} hold. Then the inequality
\begin{equation}\label{eq:bound-deltat-final}
    \begin{aligned}
        \bE\left[\delta_{(n+1)q} + a_\delta\sum_{t=nq+1}^{(n+1)q-1}\delta_{t} \right]
        \leq &\bE\left[ 
        \xz{(1-a_\delta)}\delta_{nq} 
         + \frac{16\eta_y q\sigma^2}{\mu S_1} 
        + a_\Phi
         \sum_{t=nq}^{(n+1)q-1}\|\grad \Phi(\bbx^{t})\|^2 \right.
         \\
         &
         \left.+ a_Z
           \sum_{t=nq}^{(n+1)q-1} \|Z^t_\perp \|^2_F 
            +\frac{16(q-1)\eta_y}{\mu M} \frac{L^2}{S_2}4\eta_y^2\sum_{t=nq}^{(n+1)q-2}\|D^t_\perp\|_F^2\right],
    \end{aligned}
\end{equation}
holds for all $n\in \mathbb{N}$,
where
\begin{equation}\label{eq:def-a-coeff}
    \begin{aligned}
        & a_\delta \triangleq  \frac{1}{4}\eta_y\mu - \frac{12}{\eta_y\mu}\kappa^2\eta_x^2 L^2
        - \xzrr{\frac{320(q-1)L^4\eta_y^3}{\mu S_2}},
        \\ 
        & a_\Phi \triangleq  \frac{12}{\eta_y\mu}\kappa^2\eta_x^2 + \frac{16(q-1)\eta_y}{\mu}\frac{L^2}{S_2}12\eta_x^2,
        \\
        & a_Z \triangleq \Big(\frac{4\eta_y}{\mu}+
          \frac{12}{\eta_y\mu}\kappa^2\eta_x^2\Big)
           \frac{2L^2}{M}
           +
           \frac{16(q-1)\eta_y}{\mu M}\frac{L^2}{S_2}  (8+\xzrr{40\eta_y^2L^2}).
    \end{aligned}
\end{equation}
\end{lemma}
\begin{proof}
     For given $t\in \mathbb{N}^+$, we take $n\in \mathbb{N}$ such that $nq<t\leq(n+1)q$. Then it follows from \cref{lemma:bound-deltat-1} and 
  $\eta_x\leq \frac{1}{\sqrt{3}\kappa}\eta_y$, $\eta_y\leq \frac{1}{L}$ together with \cref{lemma:parameter-ineq} that
     \begin{equation*}
    \begin{aligned}
              \delta_t 
        \leq
        & \Big(1-\frac{1}{4}\eta_y\mu + 3\Big(\frac{4}{\eta_y\mu}-1\Big)\kappa^2\eta_x^2 L^2\Big)\delta_{t-1} 
        +3\Big(\frac{4}{\eta_y\mu}-1\Big)\kappa^2\eta_x^2
         \|\grad \Phi(\bbx^{t-1})\|^2
        \\
        &
         + 
         \Big(\frac{4\eta_y}{\mu}+
           3\Big(\frac{4}{\eta_y\mu}-1\Big)\kappa^2\eta_x^2\Big)
           \frac{2L^2}{M}
          \|Z^{\yy{t-1}}_\perp \|^2_F
           +
         \frac{8\eta_y}{\mu M}
         \|E^{t-1}\|^2
         .
    \end{aligned}
\end{equation*}
Moreover, if we take the expectation of the above inequality and then use \cref{lemma:bound-Et}, we obtain
     \begin{equation}\label{eq:bound-deltat-5}
    \begin{aligned}
             \bE[ \delta_t ]
        \leq
        &\bE\Big[ \Big(1-\frac{1}{4}\eta_y\mu + 3\Big(\frac{4}{\eta_y\mu}-1\Big)\kappa^2\eta_x^2 L^2\Big)\delta_{t-1} 
        + 
         3\Big(\frac{4}{\eta_y\mu}-1\Big)\kappa^2\eta_x^2
         \|\grad \Phi(\bbx^{t-1})\|^2
         \\
         & +
         \Big(\frac{4\eta_y}{\mu}+
           3\Big(\frac{4}{\eta_y\mu}-1\Big)\kappa^2\eta_x^2\Big)
           \frac{2L^2}{M}
           \|Z^{\yy{t-1}}_\perp \|^2_F \Big]
         +
         \frac{8\eta_y}{\mu}\Big(\frac{2\sigma^2}{S_1} + \frac{1}{M}\sum_{j=nq}^{t-2}(1+c_1)^{t-2-j} \Delta_{j}\Big).
    \end{aligned}
\end{equation}
Next, if we sum up \cref{eq:bound-deltat-5} over $t=nq+1$ to $(n+1)q$, we obtain that
\begin{equation}
    \label{eq:bound-deltat-6}
    \begin{aligned}
        & \bE\Big[\delta_{(n+1)q} + \Big(\frac{1}{4}\eta_y\mu - 3\Big(\frac{4}{\eta_y\mu}-1\Big)\kappa^2\eta_x^2 L^2\Big)\sum_{t=nq+1}^{(n+1)q-1}\delta_{t} \Big]
        \\
        \leq &\bE\Big[ \Big(1-\frac{1}{4}\eta_y\mu + 3\Big(\frac{4}{\eta_y\mu}-1\Big)\kappa^2\eta_x^2 L^2\Big)\delta_{nq} + \frac{16\eta_y q\sigma^2}{\mu S_1} 
        +  3\Big(\frac{4}{\eta_y\mu}-1\Big)\kappa^2\eta_x^2
         \sum_{t=nq}^{(n+1)q-1}\|\grad \Phi(\bbx^{t})\|^2
         \\
         & + \Big(\frac{4\eta_y}{\mu}+
           3\Big(\frac{4}{\eta_y\mu}-1\Big)\kappa^2\eta_x^2\Big)
           \frac{2L^2}{M}
           \sum_{t=nq}^{(n+1)q-1}\|Z^t_\perp \|^2_F \Big]
           + \frac{8\eta_y}{\mu M} \sum_{t=nq+1}^{(n+1)q}\sum_{j=nq}^{t-2}(1+c_1)^{t-2-j} \Delta_{j},
    \end{aligned}
\end{equation}
\sa{where we set $\sum_{j=nq}^{nq-1}\Delta_j=0$ which arises for $t=nq+1$ in the above double summation.}
Next, if we use \cref{eq:bound-double-sum} within \cref{eq:bound-deltat-6}, we obtain that
\begin{equation*}
    \begin{aligned}
        & \bE\Big[\delta_{(n+1)q} + \Big(\frac{1}{4}\eta_y\mu - 3\Big(\frac{4}{\eta_y\mu}-1\Big)\kappa^2\eta_x^2 L^2\Big)\sum_{t=nq+1}^{(n+1)q-1}\delta_{t} \Big]
        \\
        \leq &\bE\Big[ \Big(1-\frac{1}{4}\eta_y\mu + 3\Big(\frac{4}{\eta_y\mu}-1\Big)\kappa^2\eta_x^2 L^2\Big)\delta_{nq} + \frac{16\eta_y q\sigma^2}{\mu S_1} 
        +  3\Big(\frac{4}{\eta_y\mu}-1\Big)\kappa^2\eta_x^2
         \sum_{t=nq}^{(n+1)q-1}\|\grad \Phi(\bbx^{t})\|^2
         \\
         & + \Big(\frac{4\eta_y}{\mu}+
           3\Big(\frac{4}{\eta_y\mu}-1\Big)\kappa^2\eta_x^2\Big)
           \frac{2L^2}{M}
           \sum_{t=nq}^{(n+1)q-1} \|Z^t_\perp \|^2_F \Big]
            + \frac{16(q-1)\eta_y}{\mu M} \sum_{t=nq}^{(n+1)q-2}\Delta_{t}.
    \end{aligned}
\end{equation*}
The above inequality further implies
\begin{equation}\label{eq:bound-deltat-7}
    \begin{aligned}
        & \bE\Big[\delta_{(n+1)q} + \Big(\frac{1}{4}\eta_y\mu - \frac{12}{\eta_y\mu}\kappa^2\eta_x^2 L^2\Big)\sum_{t=nq+1}^{(n+1)q-1}\delta_{t}\Big] 
        \\
        \leq & \bE\Big[\Big(1-\frac{1}{4}\eta_y\mu + \frac{12}{\eta_y\mu}\kappa^2\eta_x^2 L^2\Big)\delta_{nq} + \frac{16\eta_y q\sigma^2}{\mu S_1} 
        +  \frac{12}{\eta_y\mu}\kappa^2\eta_x^2
         \sum_{t=nq}^{(n+1)q-1}\|\grad \Phi(\bbx^{t})\|^2
         \\
         & + \Big(\frac{4\eta_y}{\mu}+
          \frac{12}{\eta_y\mu}\kappa^2\eta_x^2\Big)
           \frac{2L^2}{M}
           \sum_{t=nq}^{(n+1)q-1} \|Z^t_\perp \|^2_F \Big]
            + \frac{16(q-1)\eta_y}{\mu M} \sum_{t=nq}^{(n+1)q-2}\Delta_{t}.
    \end{aligned}
\end{equation}
Furthermore, recall from \eqref{eq:def-Delta-t} that
\begin{equation*}
    \begin{aligned}
        \Delta_t =  \frac{L^2}{S_2}\bE\Big[\Big( \xz{8}+\xzrr{\xzrr{40\eta_y^2L^2}}  \Big)\|Z^t_\perp \|^2_F + 4\eta_y^2 \|D^t_\perp \|^2_F 
        + \xzrr{20\eta_y^2L^2M\delta_t} + 12M\eta_x^2\|\grad \Phi(\bbx^t)\|^2
        \Big].
    \end{aligned}
    \end{equation*}
If we plug the above equality into \cref{eq:bound-deltat-7}, we obtain \cref{eq:bound-deltat-final} and complete the proof.
\end{proof}


\subsection{Bound of Consensus Error}\label{eq:bound-DZ}
\xz{In this subsection, we display the proper bound of the measure of consensus error $\bE[\|Z^t_\perp\|^2_F]$ and $\bE[\|D^t_\perp\|^2_F]$. The analysis in \cref{lemma:bound-Z-D-perp-final} will be utilized within \cref{lemma:bound-Phi-final} to derive the final convergence result in \cref{thm:main-convergence} by constructing a telescoping sum.}
\begin{lemma}\label{lemma:bound-D-perp-1}
      Suppose Assumptions~\ref{ASPT:SC},~\ref{ASPT:general-sto-grad},~\ref{ASPT:smooth-F-mean-squared}~and~\ref{ASPT:mixture-matrix} hold. Then
        \begin{equation}\label{eq:bound-D-perp-3}
    \begin{aligned}
        \bE[\|D_{\perp}^{t}\|^2_F] 
         \leq & \bE\Bigg[(\rho + \frac{4L^2\yy{\rho^2}\eta_y^2}{1-\rho})\|D^{t-1}_{\perp}\|_F^2 + \frac{L^2\yy{\rho^2}}{1-\rho}\Big( \xzrr{20\eta_y^2L^2M}\delta_{t-1} + 12M\eta_x^2\|\grad \Phi(\bbx^{t-1})\|^2 
         \\
         & +\xzrr{24\eta_y^2}\|E^{t-1}\|_F^2
         +
        \xzrr{(8+40\eta_y^2L^2)}\|Z^{t-1}_\perp \|^2_F
         \Big)\Bigg]
    \end{aligned}
    \end{equation}
    holds for all $t\geq 1$ and $mod(t,q)>0$. Moreover, when $t\geq 1$ and $mod(t,q)=0$, i.e., $t=nq$ for some $n\in\mathbb{N}$, it holds
     \begin{equation}
          \label{eq:bound-D-perp-4}
        \begin{aligned}
                    \bE[\|D_{\perp}^{t}\|^2_F ]
            \leq & 
            \bE\Big[\rho\|D^{t-1}_{\perp}\|_F^2 + \frac{\yy{\rho^2}}{1-\rho} \frac{6M\sigma^2}{S_1} 
            +
            \frac{3S_2\yy{\rho^2}}{1-\rho}\Delta_{t-1} + \frac{3\yy{\rho^2}}{1-\rho}\sum_{j = (n-1)q+1}^{t-1}\Delta_{j-1}
            \\
            & + \frac{\xzrr{72\eta_y^2}L^2\yy{\rho^2}}{1-\rho}\|E^{t-1}\|^2_F + \frac{\xzrr{72\eta_y^2}L^2\yy{\rho^2}}{(1-\rho)S_2}\sum_{j = (n-1)q+1}^{t-1}\|E^{j-1}\|^2_F\Big].
        \end{aligned}
      \end{equation}
\end{lemma}
\begin{proof}
     Recall that $D^t_\perp = D^t-\Bar{D}^t$ \sa{and}
     $D^{t+1} = W(D^{t} + V^{t+1} -V^{t})$, \sa{which implies that} $\bar{D}^{t+1} = \bar{D}^t + \bar{V}^{t+1} - \bar{V}^t$ for $t\geq 0$.
     Therefore, for 
     any constant $a>0$, we have that
        \begin{equation}\label{eq:bound-D-perp-1}
        \begin{aligned}
                    \|D_{\perp}^{t}\|^2_F 
        = & \|D^{t} - \Bar{D}^{t}\|^2_F = \|W(D^{t-1} + V^{t} -V^{t-1}) - (\bar{D}^{t-1} + \bar{V}^{t} - \bar{V}^{t-1})\|_F^2
        \\
        =&\|(W-\Pi)D^{t-1}_\perp +(W-\Pi)({V}_\perp ^t - {V}_\perp ^{t-1})\|^2_F
        \\
        \leq &
        (1+a)\|(W-\Pi)D_\perp ^{t-1}\|_F^2 + \textstyle (1+\frac{1}{a})\|(W-\Pi)({V}_\perp ^t - {V}_\perp ^{t-1})\|^2_F
        \\
        \leq & 
        (1+a)\|W-\Pi\|^2_2\|D_\perp ^{t-1}\|_F^2 + \textstyle (1+\frac{1}{a})\|W-\Pi\|^2_2\|{V}_\perp ^t - {V}_\perp ^{t-1}\|^2_F
        \\
        \leq & \rho\|D^{t-1}_{\perp}\|_F^2 + \frac{\yy{\rho^2}}{1-\rho}\|V^t - V^{t-1}\|^2_F
        \end{aligned}
    \end{equation}
    holds for all $t\geq 1$,
     where the first inequality is by Young's inequality and $\Pi$ is the average operator such that $\Pi\triangleq\frac{1}{M}\mathbf{1}\mathbf{1}^\top\in\reals^{M\times M}$, \sa{the second inequality uses $\norm{AB}_F^2\leq \norm{A}_2^2\norm{B}_F^2$}, the third inequality is by \cref{ASPT:mixture-matrix}, and  letting $a = \frac{1}{\rho}-1$ and the fact $\rho\yy{<}1$. In the following part, we will analyze $\|V^t - V^{t-1}\|^2_F$ when  $mod(t,q)>0$ and $mod(t,q)=0$. Indeed, when ${mod(t,q)>0}$ and $t\geq 1$, we have that
    \begin{equation}\label{eq:V-diff-1}
        \begin{aligned}
            \bE[\|V^t-V^{t-1}\|^2_F ]
            = &  \bE\left[\sum_{i=1}^M \left\|{\frac{1}{|\gmb{\cB_{i}^{t}}|}\sum_{\xi\in \gmb{\cB_{i}^{t}}}\Big(\tgrad f_i(x^t_i,y^t_i;\xi) - \tgrad f_i(x^{t-1}_i,y^{t-1}_i;\xi)\Big)}\right\|^2\right]
            \\
            \leq &  \bE\left[ \sum_{i=1}^M \frac{1}{|\gmb{\cB_{i}^{t}}|}\sum_{\xi\in \gmb{\cB_{i}^{t}}} \|\tgrad f_i(x^t_i,y^t_i;\xi) - \tgrad f_i(x^{t-1}_i,y^{t-1}_i;\xi)\|^2\right]
            \\
            \leq & L^2  \bE[\|Z^t-Z^{t-1}\|^2_F]
        \end{aligned}
    \end{equation}
    where the first inequality is by Young's inequality, and the second inequality is by \cref{ASPT:smooth-F-mean-squared}.  Moreover, if we take the expectation of \cref{eq:bound-D-perp-1} and then use \cref{eq:V-diff-1} and \cref{lemma:bound-ZF}, it follows that \cref{eq:bound-D-perp-3}
    holds for all $t\geq 1$ and $mod(t,q)>0$. Therefore, 
    \sa{the desired result holds for all $t\geq 1$ such that $mod(t,q)>0$.}

    Next, \sa{for $t\geq 1$ such that ${mod(t,q)=0}$,}  we upper bound $\norm{\bv^t_i - \bv^{t-1}_i}^2$ 
    as follows:
    \begin{equation}\label{eq:V-diff-2}
        \begin{aligned}
            &
            \bE[\norm{\bv^t_i - \bv^{t-1}_i}^2]
            \\
             =&
            \bE[\norm{\bv^t_i - \grad f_i(\bx^t_i,\by^t_i) + \grad f_i(\bx^t_i,\by^t_i)
            -  \grad f_i(\bx^{t-1}_i,\by^{t-1}_i) + \grad f_i(\bx^{t-1}_i,\by^{t-1}_i) - \bv^{t-1}_i}^2]
            \\
            \leq &  3\bE[\norm{\bv^t_i - \grad f_i(\bx^t_i,\by^t_i)}^2] + 3\bE[\norm{\grad f_i(\bx^t_i,\by^t_i)
            -  \grad f_i(\bx^{t-1}_i,\by^{t-1}_i)}^2] + 3\bE[\norm{  \bv^{t-1}_i- \grad f_i(\bx^{t-1}_i,\by^{t-1}_i)}^2]
            \\
            = &  3\bE[\norm{\be^t_i}^2] + 3\bE[\norm{\grad f_i(\bx^t_i,\by^t_i)
            -  \grad f_i(\bx^{t-1}_i,\by^{t-1}_i)}^2] + 3\bE[\norm{\be^{t-1}_i}^2]
            \\
            \leq &
            \frac{3\sigma^2}{S_1} + 3L^2\bE[\norm{\bz^t_i-\bz^{t-1}_i}^2] + 3\bE[\norm{\be^{t-1}_i}^2],\quad \forall i\in [M],
        \end{aligned}
    \end{equation}
    where the last inequality is by \cref{lemma:bound-eti} and \cref{ASPT:smooth-F-mean-squared}. Next, we assume that $t=nq$ for some $n\in \mathbb{N}^+$; \sa{since $mod(t-1,q)>0$, \cref{lemma:bound-eti} 
    implies that}
    \begin{equation}\label{eq:ei-deduction}
        \begin{aligned}
            \bE[\|\be^{t-1}_i\|^2] 
            \leq &
            \bE[\|\be^{t-2}_i\|^2]
            +  \frac{L^2}{S_2}\bE[\|\bz^{t-1}_i-\bz^{t-2}_i\|^2]
            \\
            \leq & \bE[\|\be^{(n-1)q}_i\|^2]
            +  \frac{L^2}{S_2}\sum_{j = (n-1)q+1}^{t-1}\bE[\|\bz^{j}_i-\bz^{j-1}_i\|^2]
            \\
            \leq & \frac{\sigma^2}{S_1} +  \frac{L^2}{S_2}\sum_{j = (n-1)q+1}^{t-1}\bE[\|\bz^{j}_i-\bz^{j-1}_i\|^2],\quad \forall i\in [M],
        \end{aligned}
    \end{equation}
    holds for all $t\geq 1$ and $mod(t,q)=0$.
    If we plug \cref{eq:ei-deduction} into \cref{eq:V-diff-2}, it follows that
    \begin{equation*}
        \begin{aligned}
            \bE[\norm{\bv^t_i - \bv^{t-1}_i}^2]
            \leq &
            \frac{6\sigma^2}{S_1} + 3L^2\bE[\norm{\bz^t_i-\bz^{t-1}_i}^2] 
            +  \frac{3L^2}{S_2}\sum_{j = (n-1)q+1}^{t-1}\bE[\|\bz^{j}_i-\bz^{j-1}_i\|^2],\quad \forall i\in [M],
        \end{aligned}
    \end{equation*}
     holds for all $t\geq 1$ and $mod(t,q)=0$; thus, we  obtain that 
     \begin{equation}\label{eq:V-diff-3}
         \bE[\norm{V^t-V^{t-1}}^2_F] \leq  \frac{6M\sigma^2}{S_1} + 3L^2\bE[\norm{Z^t-Z^{t-1}}^2_F] 
            +  \frac{3L^2}{S_2}\sum_{j = (n-1)q+1}^{t-1}\bE[\norm{Z^j-Z^{j-1}}^2_F]
     \end{equation}
      holds for all $t\geq 1$ \sa{such that $mod(t,q)=0$, i.e., $t=nq$ for some $n\in\mathbb{N}$.} Moreover, if we plug \cref{eq:V-diff-3} into \cref{eq:bound-D-perp-1}, and then use \cref{lemma:bound-ZF}, we obtain that \cref{eq:bound-D-perp-4}
       holds for all $t\geq 1$ \sa{such that} $mod(t,q)=0$, \sa{i.e., there exists $n\in\mathbb{N}$ such that $t=nq$.}
\end{proof}

\begin{lemma}\label{lemma:bound-D-perp-2}
  Suppose Assumptions~\ref{ASPT:SC},~\ref{ASPT:general-sto-grad},~\ref{ASPT:smooth-F-mean-squared}~and~\ref{ASPT:mixture-matrix} hold. Then the inequality
    \begin{equation}\label{eq:bound-D-perp-15}
 \begin{aligned}
        & \sa{\bE\Big[(1-\alpha)\|D^{Nq-1}_\perp\|^2_F 
        + \alpha \sum_{j=1}^{Nq-1} \|D_\perp^{j}\|^2_F}\Big]\\
        \sa{=} & \bE\Big[\|D^{Nq-1}_\perp\|^2_F 
        + \sum_{n=1}^{N-1}\xz{\alpha} \|D_\perp^{nq-1}\|^2_F
        +\sum_{n=1}^{N-1}\sum_{j=nq}^{(n+1)q-2}\xz{\alpha} \|D^j_\perp\|^2_F 
        +
        \xz{\alpha}  \sum_{j=1}^{q-2}\|D_\perp^{\sa{j}}\|^2\Big]
        \\
        \leq &
        \bE\Big[\Big(
         (\frac{3L^2}{(1-\rho)S_2} +  \xzrr{\frac{288\tilde{q}\eta_y^2L^4}{(1-\rho)S_2}})
        4\yy{\rho^2} \eta_y^2
        +
        ( \rho + \frac{4L^2\yy{\rho^2}\eta_y^2}{1-\rho} + \xzrr{\frac{24\eta_y^2L^2\yy{\rho^2}}{1-\rho}} \frac{8(q-1)L^2\eta_y^2}{S_2})\Big)\|D^0_\perp\|^2_F
        \\
        &
        +
        \sum_{j=0}^{q-2}
        \Big(
        {b}_\delta\delta_j
        +
        {b}_\Phi \|\grad \Phi(\bbx^j)\|^2
        +
        {b}_Z\|Z^j_\perp\|^2_F 
        \Big)
        +
        \sum_{n=1}^{N-1}\sum_{j=(n-1)q}^{(n+1)q-2}
        \Big(
        \hat{b}_\delta\delta_j
        +
        \hat{b}_\Phi \|\grad \Phi(\bbx^j)\|^2
        +
        \hat{b}_Z\|Z^j_\perp\|^2_F 
        \Big)\Big]
        \\
        &
      +
        (N-1)\frac{2M\sigma^2\yy{\rho^2} }{(1-\rho)S_1}
          \Big(
          ((q-1) +6\Tilde{q}) \xzrr{24L^2\eta_y^2} + 3
          \Big) + \xzrr{\frac{24\eta_y^2L^2\yy{\rho^2}}{1-\rho}}\frac{2M\sigma^2}{S_1}(q-1),
        \end{aligned}
\end{equation}
holds for all $N\in\mathbb{N}^+$,
where  
\begin{equation}\label{eq:def-tilde-q}
    \tilde{q} \triangleq \max \left\{1,\frac{q-1}{S_2}\right\}
\end{equation}
and
\begin{equation}\label{eq:def-alpha}
    \begin{aligned}
    & \alpha \triangleq \min\{\alpha_1,\alpha_2\},\;\alpha_1 \triangleq 1-\rho - \frac{12L^2\yy{\rho^2}\eta_y^2}{1-\rho},\\
        & \alpha_2 \triangleq 1-\rho  - \frac{4L^2\yy{\rho^2}\eta_y^2}{1-\rho} - \xzrr{\frac{24\eta_y^2L^2\yy{\rho^2}}{1-\rho}} \frac{8(q-1)L^2\eta_y^2}{S_2} -  \yy{\rho^2}\Big(\frac{3L^2}{(1-\rho)S_2} +  \xzrr{\frac{288\tilde{q}\eta_y^2L^4}{(1-\rho)S_2}}\Big)
        4\eta_y^2,
    \end{aligned}
\end{equation}
and
\begin{equation}\label{def:b-coeff}
    \begin{aligned}
        & b_\delta \triangleq \frac{L^2\yy{\rho^2}}{1-\rho}\xzrr{20ML^2\eta_y^2} + \xzrr{\frac{24\eta_y^2L^2\yy{\rho^2}}{1-\rho} \frac{2(q-1)L^2}{S_2}20ML^2\eta_y^2},
        \\
        & b_\Phi \triangleq \frac{12\rho^2 ML^2\eta_x^2}{1-\rho}
        + \xzrr{\frac{24\rho^2 \eta_y^2L^2}{1-\rho}}\frac{24(q-1)ML^2\eta_x^2}{S_2},
        \\
        & b_Z \triangleq \frac{\yy{\rho^2}L^2}{1-\rho} (\xzrr{8+40L^2\eta_y^2})(1+\xzrr{\frac{48(q-1)L^2\eta_y^2}{S_2}}),
        \end{aligned}
\end{equation}
and
\begin{equation}
    \label{eq:def-b-hat-coeff}
    \begin{aligned}
        &
        \hat{b}_\delta \triangleq
        \max\{
        b_\delta, \frac{3L^2\yy{\rho^2}}{1-\rho}\xzrr{20ML^2\eta_y^2}, \yy{\rho^2}\Big(\frac{3L^2}{(1-\rho)S_2} +  \xzrr{\frac{288\tilde{q}\eta_y^2L^4}{(1-\rho)S_2}}\Big) 
           \xzrr{20ML^2\eta_y^2}
        \},
        \\
        & \hat{b}_\Phi \triangleq 
        \max\{
        b_\Phi, \frac{36ML^2\yy{\rho^2}\eta_x^2}{1-\rho},
         \yy{\rho^2}\Big(\frac{3L^2}{(1-\rho)S_2} +  \xzrr{\frac{288\tilde{q}\eta_y^2L^4}{(1-\rho)S_2}}\Big)
            12M\eta_x^2
        \},
        \\
        &\hat{b}_Z \triangleq
        \max\{
        b_Z, \frac{3\yy{\rho^2}L^2}{1-\rho}
         \xzrr{(8+40\eta_y^2L^2)},
          \yy{\rho^2}\Big(\frac{3L^2}{(1-\rho)S_2} +  \xzrr{\frac{288\tilde{q}\eta_y^2L^4}{(1-\rho)S_2}}\Big)
           \xzrr{(8+40\eta_y^2L^2)}      
        \}.
    \end{aligned}
\end{equation}
\end{lemma}
\begin{proof}
     For all $t\geq 1$ \sa{such that} $mod(t,q)>0$, it follows from \cref{lemma:bound-D-perp-1} and \cref{lemma:bound-Et} that
    \begin{equation}\label{eq:bound-D-perp-5}
    \begin{aligned}
        \bE[\|D_{\perp}^{t}\|^2_F] 
         \leq & (\rho + \frac{4L^2\yy{\rho^2}\eta_y^2 }{1-\rho})\bE[\|D^{t-1}_{\perp}\|_F^2] + \frac{L^2\yy{\rho^2} }{1-\rho}\Big( \xzrr{20\eta_y^2L^2M}\delta_{t-1} + 12M\eta_x^2\|\grad \Phi(\bbx^{t-1})\|^2 
         \\
         & 
         +
        \xzrr{(8+40\eta_y^2L^2)}\|Z^{t-1}_\perp \|^2_F   
         \Big)
         + \frac{\xzrr{24\eta_y^2}L^2\yy{\rho^2} }{1-\rho}\Big(\frac{2M\sigma^2}{S_1} + \sum_{j=nq}^{t-2}(1+c_1)^{t-2-j} \Delta_{j}\Big).
    \end{aligned}
    \end{equation}
    Moreover, \sa{recall $\Delta_t$ defined in~\eqref{eq:def-Delta-t}, i.e.,}
\begin{equation*}
    \begin{aligned}
        \Delta_t = &\frac{L^2}{S_2}\bE\Big[
         \xzrr{(8+40\eta_y^2L^2)}\|Z^t_\perp \|^2_F + 4\eta_y^2 \|D^t_\perp \|^2_F 
        + \xzrr{20\eta_y^2L^2M}\delta_t + 12M\eta_x^2\|\grad \Phi(\bbx^t)\|^2\Big].
    \end{aligned}
    \end{equation*}
    If we plug $\Delta_t$ into \cref{eq:bound-D-perp-5}, we obtain that
    \begin{equation}\label{eq:bound-D-perp-6}
            \begin{aligned}
        \bE[\|D_{\perp}^{t}\|^2_F] 
         \leq & \bE\Bigg[(\rho + \frac{4L^2\yy{\rho^2}\eta_y^2 }{1-\rho})\|D^{t-1}_{\perp}\|_F^2] 
         +
         \frac{\xzrr{24\eta_y^2}L^2\yy{\rho^2} }{1-\rho}\frac{L^2}{S_2}4\eta_y^2\sum_{j=nq}^{t-2}(1+c_1)^{t-2-j}\|D^j_\perp\|_F^2
         \\
         &
          + \frac{L^2\yy{\rho^2}}{1-\rho}\xzrr{20M\eta_y^2L^2}\delta_{t-1}
          +
         \frac{\xzrr{24\eta_y^2}L^2\yy{\rho^2}}{1-\rho}\frac{L^2}{S_2}\xzrr{20\eta_y^2L^2M}\sum_{j=nq}^{t-2}(1+c_1)^{t-2-j}
          \delta_{j}
         \\
         &
         + \frac{12M L^2\yy{\rho^2}\eta_x^2}{1-\rho}\|\grad \Phi(\bbx^{t-1})\|^2 
         +
          \frac{\xzrr{24\eta_y^2}L^2\yy{\rho^2}}{1-\rho}\frac{12ML^2\eta_x^2}{S_2}\sum_{j=nq}^{t-2}(1+c_1)^{t-2-j}
          \|\grad \Phi(\bbx^{j})\|^2 
         \\
         & 
         +
        \frac{L^2\yy{\rho^2}}{1-\rho}\xzrr{(8+40\eta_y^2L^2)}\|Z^{t-1}_\perp \|^2_F   
        \\
        &+
          \xzrr{\frac{24\eta_y^2L^2\yy{\rho^2}}{1-\rho}}\frac{L^2}{S_2}\xzrr{(8+40\eta_y^2L^2)}\sum_{j=nq}^{t-2}(1+c_1)^{t-2-j}\|Z^{j}_\perp \|^2_F  
          \\
          & +  \xzrr{\frac{24\eta_y^2L^2\yy{\rho^2}}{1-\rho}}\frac{2M\sigma^2}{S_1}\Bigg]
    \end{aligned}
    \end{equation}
     holds for all $t\geq 1$ \sa{such that} $mod(t,q)>0$. Moreover, if we assume $nq<t\sa{\leq}(n+1)q-1$ for some $n\in\mathbb{N}$ and sum \cref{eq:bound-D-perp-6} over $t=nq+1$ to $(n+1)q-1$, and use  the fact
     that
\begin{equation*}
\begin{aligned}
        & \sum_{t=nq+1}^{(n+1)q-1}\sum_{j=nq}^{t-2}(1+c_1)^{t-2-j}A_j =\sum_{j=nq}^{(n+1)q-3}\sum_{t=j+2}^{(n+1)q-1}(1+c_1)^{t-2-j}A_j \leq 2(q-1) \sum_{t=nq}^{(n+1)q-3}A_t
\end{aligned}
\end{equation*}
holds for any nonegative number sequence $\{A_j\}$,
where the last inequality is by $q\leq\frac{1}{2c_1}$ and \cref{lemma:parameter-ineq}, we obtain that
\begin{equation}\label{eq:bound-D-perp-7}
    \begin{aligned}
        & \bE\Bigg[\sum_{t=nq+1}^{(n+1)q-1}\|D^t_\perp\|^2_F - \sum_{t=nq}^{(n+1)q-2}(\rho + \frac{4L^2\yy{\rho^2}\eta_y^2}{1-\rho}) 
        \|D^t_\perp\|^2_F \Bigg]
        \\
        \leq & \bE\Bigg[\xzrr{\frac{24\eta_y^2L^2\yy{\rho^2}}{1-\rho}} \frac{8(q-1)L^2\eta_y^2}{S_2}\sum_{t=nq}^{(n+1)q-3}\|D^t_\perp\|^2_F + b_\delta\sum_{t=nq}^{(n+1)q-2}\delta_t + b_\Phi \sum_{t=nq}^{(n+1)q-2} \|\grad \Phi(\bbx^t)\|^2
        \\
      & + b_Z\sum_{t=nq}^{(n+1)q-2}\|Z^t_\perp\|^2_F\Bigg] + \xzrr{\frac{24\eta_y^2L^2\yy{\rho^2}}{1-\rho}}\frac{2M\sigma^2}{S_1}(q-1),
        \end{aligned}
\end{equation}
holds for all $n\in\mathbb{N}$, where $b_\delta,b_\Phi,b_Z$ are defined in \cref{def:b-coeff}.

On the other hand, for all $t\geq 1$ and $mod(t,q)=0$, we assume $t=nq$ for some $n\in\mathbb{N}$. Then it follows from \cref{lemma:bound-D-perp-1} and \cref{lemma:bound-Et} that
     \begin{equation*}
        \begin{aligned}
                    \bE[\|D_{\perp}^{t}\|^2_F ]
            \leq & 
            \bE\Bigg[\rho\|D^{t-1}_{\perp}\|_F^2 + \frac{\yy{\rho^2}}{1-\rho} \frac{6M\sigma^2}{S_1} 
            +
            \frac{3\yy{\rho^2}S_2}{1-\rho}\Delta_{t-1} + \frac{3\yy{\rho^2}}{1-\rho}\sum_{j = (n-1)q+1}^{t-1}\Delta_{j-1}
            \\
            & + \frac{\xzrr{72}\eta_y^2L^2\yy{\rho^2}}{1-\rho}\Big(\frac{2M\sigma^2}{S_1} + \sum_{j=(n-1)q}^{t-2}(1+c_1)^{t-2-j} \Delta_{j}\Big) 
            \\
            & + \frac{\xzrr{72\eta_y^2}L^2\yy{\rho^2}}{(1-\rho)S_2}\sum_{s = (n-1)q+1}^{t-1}\Big(\frac{2M\sigma^2}{S_1} + \sum_{j=(n-1)q}^{s-2}(1+c_1)^{s-2-j} \Delta_{j}\Big)\Bigg]
        \end{aligned}
      \end{equation*}
       holds \sa{for all $t\geq 1$ such that} $t=nq$ for some $n\in\mathbb{N}$.
Note that
    \begin{equation*}
        \begin{aligned}
            \sum_{s = (n-1)q+1}^{t-1}\sum_{j=(n-1)q}^{s-2}(1+c_1)^{s-2-j} \Delta_{j} = \sum_{j=(n-1)q}^{t-3}\sum_{s=j+2}^{t-1}(1+c_1)^{s-2-j}\Delta_j \leq 2(q-1)\sum_{j=(n-1)q}^{t-3}\Delta_j,
        \end{aligned}
    \end{equation*}
    where the last inequality is by $q\leq  1/(2c_1) $ and \cref{lemma:parameter-ineq}, 
    \sa{implying} $(1+c_1)^{q}\leq 2$.
    Therefore, we further have that
        \begin{equation*}
        \begin{aligned}
                    \bE[\|D_{\perp}^{t}\|^2_F ]
            \leq & 
            \bE\Bigg[\rho\|D^{t-1}_{\perp}\|_F^2 + \frac{\yy{\rho^2}}{1-\rho} \frac{6M\sigma^2}{S_1} 
            +
            \frac{3\yy{\rho^2}S_2}{1-\rho}\Delta_{t-1}
             + \frac{3\yy{\rho^2}}{1-\rho}\sum_{j = (n-1)q}^{t-2}\Delta_{j}
            \\
            & + \frac{\xzrr{72\eta_y^2}L^2\yy{\rho^2}}{1-\rho}\Big(\frac{2M\sigma^2}{S_1} + 2\sum_{j=(n-1)q}^{t-2}\Delta_{j}\Big) 
            \\
            & + \frac{\xzrr{72\eta_y^2}L^2\yy{\rho^2}}{(1-\rho)S_2}\Big(\frac{2M\sigma^2}{S_1}(q-1) + 2(q-1)\sum_{j=(n-1)q}^{t-3}\Delta_{j}\Big)\Bigg]
        \end{aligned}
      \end{equation*}
      holds for \sa{all $t\geq 1$ such that} $t=nq$ for some $n\in\mathbb{N}$.
      Then, by $\tilde{q}=\yy{\max\{1,\frac{q-1}{S_2}\}}$ and rearranging terms, we have that
              \begin{equation}
          \label{eq:bound-D-perp-8}
        \begin{aligned}
                    \bE[\|D_{\perp}^{t}\|^2_F ]
            \leq & 
            \bE\Bigg[\rho\|D^{t-1}_{\perp}\|_F^2 + \frac{\yy{\rho^2}}{1-\rho} \frac{6M\sigma^2}{S_1} 
            +
            \frac{3\yy{\rho^2}S_2}{1-\rho}\Delta_{t-1}
            \\
            & 
            + \yy{\rho^2}\Big(\frac{3}{1-\rho} +\frac{\xzrr{72\eta_y^2}L^2}{1-\rho}\cdot2 
            +
            \frac{\xzrr{72\eta_y^2}L^2}{(1-\rho)S_2}\cdot2(q-1)\Big) \sum_{j=(n-1)q}^{t-2}\Delta_{j} \Bigg]
           \\
           &
            + 
            \frac{\xzrr{72\eta_y^2}\yy{\rho^2}L^2}{1-\rho}\frac{2M\sigma^2}{S_1}
            +
            \frac{\xzrr{72\eta_y^2}\yy{\rho^2}L^2}{(1-\rho)S_2}\frac{2M\sigma^2}{S_1}(q-1)
                        \\
            \leq & 
            \bE\Bigg[\rho\|D^{t-1}_{\perp}\|_F^2 + \frac{\yy{\rho^2}}{1-\rho} \frac{6M\sigma^2}{S_1} 
            +
            \frac{3\yy{\rho^2}S_2}{1-\rho}\Delta_{t-1}
            + \yy{\rho^2}\Big(\frac{3}{1-\rho} 
            +
            \xzrr{\frac{288\tilde{q}\eta_y^2L^2}{1-\rho}}
            \Big)
            \sum_{j=(n-1)q}^{t-2}\Delta_{j} \Bigg]
           \\
           &
            + 
            \frac{\xzrr{288\tilde{q}\eta_y^2L^2\yy{\rho^2}}}{1-\rho}\frac{M\sigma^2}{S_1}
        \end{aligned}
      \end{equation}
       holds for \sa{all $t\geq 1$ such that} $t=nq$ for some $n\in\mathbb{N}$,
      where the second inequality is by the fact $q\geq 1$ and $\frac{1}{S_1}\leq 1$.
       Moreover, 
       recall that \sa{for all $t\geq 1$,}
\begin{equation*}
    \begin{aligned}
        \Delta_t = &\frac{L^2}{S_2}\bE\Big[
         \xzrr{(8+40\eta_y^2L^2)}\|Z^t_\perp \|^2_F + 4\eta_y^2 \|D^t_\perp \|^2_F 
        + \xzrr{20\eta_y^2L^2M}\delta_t + 12M\eta_x^2\|\grad \Phi(\bbx^t)\|^2\Big].
    \end{aligned}
    \end{equation*}
\sa{Hence, for arbitrary $n\in\mathbb{N}$, setting $t=nq$ and substituting $\Delta_t$ into} \cref{eq:bound-D-perp-8}, we obtain that
               \begin{equation}
          \label{eq:bound-D-perp-10}
        \begin{aligned}
                    \bE[\|D_{\perp}^{nq}\|^2_F ]
            \leq & 
            \frac{\xzrr{288\tilde{q}\eta_y^2L^2\yy{\rho^2}}}{1-\rho}\frac{M\sigma^2}{S_1} + \bE\Bigg[\rho\|D^{nq-1}_{\perp}\|_F^2 + \frac{\yy{\rho^2}}{1-\rho} \frac{6M\sigma^2}{S_1} 
            \\
            &+
            \frac{3L^2\yy{\rho^2}}{1-\rho}
            \Big[
         \xzrr{(8+40\eta_y^2L^2)}\|Z^{nq-1}_\perp \|^2_F + 4\eta_y^2 \|D^{nq-1}_\perp \|^2_F
        + \xzrr{20\eta_y^2L^2M}\delta_{nq-1} + 12M\eta_x^2\|\grad \Phi(\bbx^{nq-1})\|^2\Big]
            \\
            & 
            +
            \yy{\rho^2}\Big(\frac{3L^2}{(1-\rho)S_2} +  \xzrr{\frac{288\tilde{q}\eta_y^2L^4}{(1-\rho)S_2}}\Big)
            \xzrr{(8+40\eta_y^2L^2)}
            \sum_{j=(n-1)q}^{nq-2}\|Z^{j}_\perp \|^2_F
            \\
            &+
            \yy{\rho^2}\Big(\frac{3L^2}{(1-\rho)S_2} +  \xzrr{\frac{288\tilde{q}\eta_y^2L^4}{(1-\rho)S_2}}\Big)
            4\eta_y^2 
            \sum_{j=(n-1)q}^{nq-2}\|D^{j}_\perp \|^2_F
            \\
            & 
            +
           \yy{\rho^2}\Big(\frac{3L^2}{(1-\rho)S_2} +  \xzrr{\frac{288\tilde{q}\eta_y^2L^4}{(1-\rho)S_2}}\Big) 
           \xzrr{20\eta_y^2L^2M}
            \sum_{j=(n-1)q}^{nq-2} \sa{\delta_j} 
            \\
            & +
            \yy{\rho^2}\Big(\frac{3L^2}{(1-\rho)S_2} +  \xzrr{\frac{288\tilde{q}\eta_y^2L^4}{(1-\rho)S_2}}\Big)
            12M\eta_x^2
            \sum_{j=(n-1)q}^{nq-2}\|\grad\Phi(\bbx^j) \|^2\Bigg] 
        \end{aligned}
      \end{equation}
       holds for $n\in\mathbb{N}^+$. Therefore, we have finished discussing \sa{the two cases depending on $t$}, i.e., $mod(t,q)=0$ and $mod(t,q)>0$. 
       
       Next, if we add \cref{eq:bound-D-perp-10} and \cref{eq:bound-D-perp-7}, it follows that
       \begin{equation}\label{eq:bound-D-perp-11}
    \begin{aligned}
        &\bE\Bigg[ \sum_{j=nq}^{(n+1)q-1}\|D^j_\perp\|^2_F - \sum_{j=nq}^{(n+1)q-2}(\rho + \frac{4L^2\yy{\rho^2}\eta_y^2}{1-\rho}) 
        \|D^j_\perp\|^2_F \Bigg]
        \\
        \leq & \bE\Bigg[\xzrr{\frac{24\eta_y^2L^2\yy{\rho^2}}{1-\rho}} \frac{8(q-1)L^2\eta_y^2}{S_2}\sum_{j=nq}^{(n+1)q-3}\|D^j_\perp\|^2_F 
        + 
        \Big(\frac{12L^2\yy{\rho^2}\eta_y^2}{1-\rho}+\rho\Big)
        \|D^{nq-1}_\perp\|^2_F
        \\
        &+
        \yy{\rho^2}\Big(\frac{3L^2}{(1-\rho)S_2} +  \xzrr{\frac{288\tilde{q}\eta_y^2L^4}{(1-\rho)S_2}}\Big)
        4\eta_y^2 
        \sum_{j=(n-1)q}^{nq-2}\|D^{j}_\perp \|^2_F
        \\
        &
        +
        b_\delta\sum_{j=nq}^{(n+1)q-2}\delta_j
        +
        \frac{3L^2\yy{\rho^2}}{1-\rho}\xzrr{20\eta_y^2L^2M}\delta_{nq-1}
        \\
        &+
           \yy{\rho^2}\Big(\frac{3L^2}{(1-\rho)S_2} +  \xzrr{\frac{288\tilde{q}\eta_y^2L^4}{(1-\rho)S_2}}\Big) 
           \xzrr{20\eta_y^2L^2M}
            \sum_{j=(n-1)q}^{nq-2} \sa{\delta_j} 
        \\
        & + b_\Phi \sum_{j=nq}^{(n+1)q-2} \|\grad \Phi(\bbx^j)\|^2
        +
        \frac{36ML^2\yy{\rho^2}\eta_x^2}{1-\rho} \|\grad \Phi(\bbx^{nq-1})\|^2
        \\
        & +
            \yy{\rho^2}\Big(\frac{3L^2}{(1-\rho)S_2} +  \xzrr{\frac{288\tilde{q}\eta_y^2L^4}{(1-\rho)S_2}}\Big)
            12M\eta_x^2
            \sum_{j=(n-1)q}^{nq-2}\|\grad\Phi(\bbx^j) \|^2
        \\
      & + b_Z\sum_{j=nq}^{(n+1)q-2}\|Z^j_\perp\|^2_F 
      + \frac{3L^2\yy{\rho^2}}{1-\rho}
         \xzrr{(8+40\eta_y^2L^2)}\|Z^{nq-1}_\perp \|^2_F
         \\
         & +
          \yy{\rho^2}\Big(\frac{3L^2}{(1-\rho)S_2} +  \xzrr{\frac{288\tilde{q}\eta_y^2L^4}{(1-\rho)S_2}}\Big)
           \xzrr{(8+40\eta_y^2L^2)}
            \sum_{j=(n-1)q}^{nq-2}\|Z^{j}_\perp \|^2_F\Bigg]
      \\
      &+\xzrr{\frac{24\eta_y^2L^2\yy{\rho^2}}{1-\rho}}\frac{2M\sigma^2}{S_1}(q-1)
     +
            \frac{\xzrr{288\tilde{q}\eta_y^2L^2\yy{\rho^2}}}{1-\rho}\frac{M\sigma^2}{S_1} + \frac{\yy{\rho^2}}{1-\rho}\frac{6M\sigma^2}{S_1}
        \end{aligned}
\end{equation}
holds for all $n\in\mathbb{N}^{+}$. Furthermore, if we use the definition of $\hat{b}_\delta,\hat{b}_\Phi,\hat{b}_Z$ that defined in \cref{eq:def-b-hat-coeff}
within the above inequality, we obtain 
       \begin{equation*}
    \begin{aligned}
        & \bE\Bigg[\sum_{j=nq}^{(n+1)q-1}\|D^j_\perp\|^2_F - \sum_{j=nq}^{(n+1)q-2}(\rho + \frac{4L^2\yy{\rho^2}\eta_y^2}{1-\rho}) 
        \|D^j_\perp\|^2_F \Bigg]
        \\
        \leq & \bE\Bigg[\xzrr{\frac{24\eta_y^2L^2\yy{\rho^2}}{1-\rho}} \frac{8(q-1)L^2\eta_y^2}{S_2}\sum_{j=nq}^{(n+1)q-3}\|D^j_\perp\|^2_F 
        + 
        \Big(\frac{12L^2\yy{\rho^2}\eta_y^2}{1-\rho}+\rho\Big)
        \|D^{nq-1}_\perp\|^2_F
        \\
        &+
        \yy{\rho^2}\Big(\frac{3L^2}{(1-\rho)S_2} +  \xzrr{\frac{288\tilde{q}\eta_y^2L^4}{(1-\rho)S_2}}\Big)
        4\eta_y^2 
        \sum_{j=(n-1)q}^{nq-2}\|D^{j}_\perp \|^2_F
        \\
        &
        +
        \sum_{j=(n-1)q}^{(n+1)q-2}
        \Big(
        \hat{b}_\delta\delta_j
        +
        \hat{b}_\Phi \|\grad \Phi(\bbx^j)\|^2
        +
        \hat{b}_Z\|Z^j_\perp\|^2_F 
        \Big)\Bigg]
      \\
      &+\xzrr{\frac{24\eta_y^2L^2\yy{\rho^2}}{1-\rho}}\frac{2M\sigma^2}{S_1}(q-1)
     +
            \frac{\xzrr{288\tilde{q}\eta_y^2L^2\yy{\rho^2}}}{1-\rho}\frac{M\sigma^2}{S_1} + \frac{\yy{\rho^2}}{1-\rho}\frac{6M\sigma^2}{S_1}.
        \end{aligned}
\end{equation*}
Then by rearranging terms, we obtain that
 \begin{equation}\label{eq:bound-D-perp-12}
    \begin{aligned}
        & \bE\Bigg[\|D^{(n+1)q-1}_\perp\|^2_F + \Big(1-\rho - \frac{4L^2\yy{\rho^2}\eta_y^2}{1-\rho} - \xzrr{\frac{24\eta_y^2L^2\yy{\rho^2}}{1-\rho}} \frac{8(q-1)L^2\eta_y^2}{S_2}\Big) \sum_{j=nq}^{(n+1)q-2}
        \|D^j_\perp\|^2_F\Bigg] 
        \\
        \leq & 
        \bE\Bigg[\Big(\frac{12L^2\yy{\rho^2}\eta_y^2}{1-\rho}+\rho\Big)
        \|D^{nq-1}_\perp\|^2_F
        +
        \yy{\rho^2}\Big(\frac{3L^2}{(1-\rho)S_2} +  \xzrr{\frac{288\tilde{q}\eta_y^2L^4}{(1-\rho)S_2}}\Big)
        4\eta_y^2 
        \sum_{j=(n-1)q}^{nq-2}\|D^{j}_\perp \|^2_F
        \\
        &
        +
        \sum_{j=(n-1)q}^{(n+1)q-2}
        \Big(
        \hat{b}_\delta\delta_j
        +
        \hat{b}_\Phi \|\grad \Phi(\bbx^j)\|^2
        +
        \hat{b}_Z\|Z^j_\perp\|^2_F 
        \Big)\Bigg]
      +
          \frac{2\yy{\rho^2}M\sigma^2}{(1-\rho)S_1}
          \Big(
          \Big((q-1) +6\Tilde{q}\Big) \xzrr{24L^2\eta_y^2} + 3
          \Big).
        \end{aligned}
\end{equation}
Moreover, by the definition of $\alpha_1$ and $\alpha_2$ in \eqref{eq:def-alpha}, 
\sa{if we sum} \cref{eq:bound-D-perp-12} over $n=1,2,...,N-1$ for some $N\in\mathbb{N}^+$, we obtain that
\begin{equation}\label{eq:bound-D-perp-13}
        \begin{aligned}
        & \bE\Bigg[\|D^{Nq-1}_\perp\|^2_F 
        + \sum_{n=2}^{N-1}\alpha_1 \|D_\perp^{nq-1}\|^2_F
        +\sum_{n=1}^{N-1}\sum_{j=nq}^{(n+1)q-2}\alpha_2\|D^j_\perp\|^2_F \Bigg]
        \\
        \leq & 
        \bE\Bigg[\Big(\frac{12L^2\yy{\rho^2}\eta_y^2}{1-\rho}+\rho\Big)
        \|D^{q-1}_\perp\|^2_F
        +
        \yy{\rho^2}\Big(\frac{3L^2}{(1-\rho)S_2} +  \xzrr{\frac{288\tilde{q}\eta_y^2L^4}{(1-\rho)S_2}}\Big)
        4\eta_y^2 
        \sum_{j=0}^{q-2}\|D^{j}_\perp \|^2_F
        \\
        &
        +
        \sum_{n=1}^{N-1}\sum_{j=(n-1)q}^{(n+1)q-2}
        \Big(
        \hat{b}_\delta\delta_j
        +
        \hat{b}_\Phi \|\grad \Phi(\bbx^j)\|^2
        +
        \hat{b}_Z\|Z^j_\perp\|^2_F 
        \Big)\Bigg]
      +
          (N-1)\frac{2\yy{\rho^2}M\sigma^2}{(1-\rho)S_1}
          \Big(
          \Big((q-1) +6\Tilde{q}\Big) \xzrr{24L^2\eta_y^2}+ 3
          \Big).
        \end{aligned}
\end{equation}
In addition, for $n=0$, it follows from \cref{eq:bound-D-perp-7} that
\begin{equation*}
    \begin{aligned}
        & \bE\Bigg[\sum_{j=1}^{q-1}\|D^j_\perp\|^2_F - \sum_{j=0}^{q-2}(\rho + \frac{4L^2\yy{\rho^2}\eta_y^2}{1-\rho}) 
        \|D^j_\perp\|^2_F \Bigg]
        \\
        \leq &\bE\Bigg[ \xzrr{\frac{24\eta_y^2L^2\yy{\rho^2}}{1-\rho}} \frac{8(q-1)L^2\eta_y^2}{S_2}\sum_{j=0}^{q-3}\|D^j_\perp\|^2_F + b_\delta\sum_{j=0}^{q-2}\delta_{\sa{j}} + b_\Phi \sum_{j=0}^{q-2} \|\grad \Phi(\bbx^{\sa{j}})\|^2 + b_Z\sum_{j=0}^{q-2}\|Z^{\sa{j}}_\perp\|^2_F\Bigg]
        \\
        &  
        + \xzrr{\frac{24\eta_y^2L^2\yy{\rho^2}}{1-\rho}}\frac{2M\sigma^2}{S_1}(q-1),
    \end{aligned}
\end{equation*}
which further implies that
\begin{equation}\label{eq:bound-D-perp-14}
    \begin{aligned}
        & \bE\Bigg[\|D^{q-1}_\perp\|^2_F + \sum_{j=1}^{q-2}\Big(1-\rho -\frac{4L^2\yy{\rho^2}\eta_y^2}{1-\rho} - \xzrr{\frac{24\eta_y^2L^2\yy{\rho^2}}{1-\rho}} \frac{8(q-1)L^2\eta_y^2}{S_2}\Big)
        \|D^j_\perp\|^2_F\Bigg]
        \\
        \leq & \bE\Bigg[\Big( \rho + \frac{4L^2\yy{\rho^2}\eta_y^2}{1-\rho} + \xzrr{\frac{24\eta_y^2L^2\yy{\rho^2}}{1-\rho}} \frac{8(q-1)L^2\eta_y^2}{S_2}\Big)\|D^0_\perp\|^2_F + b_\delta\sum_{j=0}^{q-2}\delta_{\sa{j}} + b_\Phi \sum_{j=0}^{q-2} \|\grad \Phi(\bbx^{\sa{j}})\|^2
        \\
        & + b_Z\sum_{j=0}^{q-2}\|Z^{\sa{j}}_\perp\|^2_F\Bigg] + \xzrr{\frac{24\eta_y^2L^2\yy{\rho^2}}{1-\rho}}\frac{2M\sigma^2}{S_1}(q-1).
    \end{aligned}
\end{equation}
Next, if we sum up \cref{eq:bound-D-perp-13} and \cref{eq:bound-D-perp-14}, then it follows that \cref{eq:bound-D-perp-15} 
holds for all $N\in\mathbb{N}^+$, which completes the proof.
\end{proof}

\begin{lemma}\label{lemma:bound-Z-D-perp-final}
      Suppose \xz{Assumptions~\ref{ASPT:SC},~\ref{ASPT:general-sto-grad}, \ref{ASPT:smooth-F-mean-squared}, and \ref{ASPT:mixture-matrix}}, \cref{param-condition:basic} and {$1- \rho > 
            \frac{ 2\eta_y^2 \hat{b}_Z }{\alpha(1-\rho )}$} hold. Then \sa{the inequalities below}
         \begin{equation}\label{eq:bound-Z-perp-final}
                \begin{aligned}
           & \bE\Big[ (1-  \rho-
            \frac{ 2\eta_y^2 \hat{b}_Z }{\alpha(1-\rho )} )\sum_{j=1}^T  \| Z^j_\perp \|^2_F\Big] \\
            \leq & \yy{C_{0,Z}}
        + \sa{\frac{\eta_y^2}{\alpha(1-\rho)}}
            \bE\bigg[
        \sum_{j=0}^{q-2}
        \Big(
        {b}_\delta\delta_j
        +
        {b}_\Phi \|\grad \Phi(\bbx^j)\|^2
        \Big)
        +
        \sum_{n=1}^{N-1}\sum_{j=(n-1)q}^{(n+1)q-2}
        \Big(
        \hat{b}_\delta\delta_j
        +
        \hat{b}_\Phi \|\grad \Phi(\bbx^j)\|^2
        \Big)
       \bigg]
        \end{aligned}
    \end{equation}
    \begin{equation}\label{eq:bound-D-perp-final}
        \begin{aligned}
            & \alpha\bE\Big[\sum_{j = 1}^{T-1}\|D^j_\perp\|^2_F\Big]
            \leq \yy{C_{0,D}}+ \sa{C_1}
            \bE\bigg[ 
        \sum_{j=0}^{q-2}
        \Big(
        {b}_\delta\delta_j
        +
        {b}_\Phi \|\grad \Phi(\bbx^j)\|^2
        \Big)
        +
        \sum_{n=1}^{N-1}\sum_{j=(n-1)q}^{(n+1)q-2}
        \Big(
        \hat{b}_\delta\delta_j
        +
        \hat{b}_\Phi \|\grad \Phi(\bbx^j)\|^2
        \Big)
       \bigg]
        \end{aligned}
    \end{equation}
      hold for all $T\geq 1$ \sa{such that} $T=Nq$ for some $N\in\mathbb{N}^+$,
      where \sa{$C_1\triangleq 
            1 +\frac{2\eta_y^2\hat{b}_Z}{\alpha(1-\rho)}\Big(1-\rho -  \frac{ 2\eta_y^2 \hat{b}_Z }{\alpha(1-\rho )}\Big)^{-1}
            $, and}
      \begin{equation}\label{eq:def-C-DZ}
      \begin{aligned}
            C_{0,Z} \triangleq & \bE\bigg[
          (\rho + \frac{2\eta_y^2\hat{b}_Z}{\alpha(1-\rho)})
            \|Z^0_\perp\|^2_F + \frac{\eta_y^2}{\alpha (1-\rho)}\Big((N-1)\frac{2\yy{\rho^2}M\sigma^2}{(1-\rho)S_1}
          \Big(
          \Big((q-1) +6\Tilde{q}\Big) \xzrr{24L^2\eta_y^2}+ 3
          \Big) +  \xzrr{\frac{24\eta_y^2L^2\yy{\rho^2}}{1-\rho}}\frac{2M\sigma^2}{S_1}(q-1)
          \Big)
            \\
            &
             + 
        \frac{\eta_y^2}{\alpha (1-\rho)}
        \Big(\sa{\alpha+}
         \Big(\frac{3L^2\yy{\rho^2}}{(1-\rho)S_2} +         \xzrr{\frac{288\tilde{q}\eta_y^2L^4\rho^2}{(1-\rho)S_2}}\Big)
        4\eta_y^2 
        +
        \Big( \rho + \frac{4L^2\yy{\rho^2}\eta_y^2}{1-\rho} + \xzrr{\frac{24\eta_y^2L^2\yy{\rho^2}}{1-\rho}} \frac{8(q-1)L^2\eta_y^2}{S_2}\Big)\Big)\|D^0_\perp\|^2_F\bigg],\\
        C_{0,D}\triangleq &\bE\bigg[2\hat{b}_Z\Big(
        \Big(1- \rho- \frac{ 2\eta_y^2 \hat{b}_Z }{\alpha(1-\rho )}\Big)^{-1}
            \Big(\rho + \frac{ 2\eta_y^2 \hat{b}_Z }{\alpha(1-\rho )}\Big)
            +
            1\Big)\|Z^0_\perp\|^2_F
            \\
            & +
            \sa{C_1}
             \Big(
         \sa{\alpha+}(\frac{3L^2\yy{\rho^2}}{(1-\rho)S_2} +         \xzrr{\frac{288\tilde{q}\eta_y^2L^4\rho^2}{(1-\rho)S_2}})
        4\eta_y^2 
         +
        ( \rho + \frac{4L^2\yy{\rho^2}\eta_y^2}{1-\rho} + \xzrr{\frac{24\eta_y^2L^2\yy{\rho^2}}{1-\rho}} \frac{8(q-1)L^2\eta_y^2}{S_2})\Big)\|D^0_\perp\|^2_F
        \\
        &
      +
        \sa{C_1}\Big((N-1)\frac{2\yy{\rho^2}M\sigma^2}{(1-\rho)S_1}
          \Big(
          \Big((q-1) +6\Tilde{q}\Big) \xzrr{24L^2\eta_y^2}+ 3
          \Big) +  \xzrr{\frac{24\eta_y^2L^2\yy{\rho^2}}{1-\rho}}\frac{2M\sigma^2}{S_1}(q-1)\Big)\bigg].
      \end{aligned}
      \end{equation}
\end{lemma}
\begin{proof}
  Recall that $X_\perp = X-\Bar{X}$, therefore, for all $t\geq 0$ and a constant $a>0$, we have
    \begin{equation*}
        \begin{aligned}
            \|X_{\perp}^{t+1}\|^2_F = & \|WX^t-\eta_xD^t_x - \Bar{X}^t + \eta_x\Bar{D}^t_x\|^2_F \\
            \leq
            &
            (1+a)\|WX^t - \Bar{X}^t\|^2_F + (1+\frac{1}{a})\eta_x^2\|D_{x\perp}^t\|^2_F
            \\
            = &
            (1+a)\|(W-\Pi)\sa{(X^t-\bar{X}^t)}\|^2_F + (1+\frac{1}{a})\eta_x^2\|D_{x\perp}^t\|^2_F
            \\
            \leq &
            (1+a)\rho^2 \|X_{\perp}^t\|_F^2 + (1+\frac{1}{a})\eta_x^2\|D^t_{x\perp}\|_F^2
            \\ 
            = & \rho\|X^t_{\perp}\|_F^2 + \frac{1}{1-\rho}\eta_x^2\|D^t_{x\perp}\|^2_F 
        \end{aligned}
    \end{equation*}%
    where the first inequality is by Young's inequality; \sa{the first equality follows from $(W-\Pi)(\bar{X}^t)=\mathbf{0}_{M\times n}$}; the second inequality is by \cref{ASPT:mixture-matrix}; the last equality is by letting $a = \frac{1}{\rho}-1$. Similarly, for all $t\geq 0$, we have
    \begin{equation}
        \|Y_{\perp}^{t+1}\|^2_F \leq \rho\|Y^t_{\perp}\|_F^2 + \frac{1}{1-\rho}\eta_y^2\|D^t_{y\perp}\|^2_F
    \end{equation}
    Because $\eta_x\leq \eta_y$, the above two inequalities further imply that
    \begin{equation*}
         \|Z_{\perp}^{t+1}\|^2_F \leq \rho\|Z^t_{\perp}\|_F^2 + \frac{1}{1-\rho}\eta_y^2\|D^t_{\perp}\|^2_F.
    \end{equation*}
    Moreover, if we sum up the above inequality from $t=0$ to $T-1$ \sa{for some $T\geq 1$ such that} $T=Nq$ for some $N\in\mathbb{N}^+$, it follows that
    \begin{equation*}
         \|Z_{\perp}^{T}\|^2_F + (1-\rho)\sum_{j=1}^{T-1}\|Z_{\perp}^{j}\|^2_F\leq \rho\|Z^0_{\perp}\|_F^2 + \frac{\eta_y^2}{1-\rho}\sum_{j=0}^{T-1}\|D^j_{\perp}\|^2_F.
    \end{equation*}
    \sa{Furthermore, since $\alpha \triangleq \min\{\alpha_1,\alpha_2\}<1$, using \cref{lemma:bound-D-perp-2} within the above inequality,} we obtain that
    \begin{equation*}
        \begin{aligned}
             \bE\Bigg[ \|Z_{\perp}^{T}\|^2_F  & + (1-\rho)\sum_{t=1}^{T-1}\|Z_{\perp}^{t}\|^2_F\Bigg]
             \leq  
             \bE\Bigg[\rho\|Z^0_{\perp}\|_F^2 + \frac{\eta_y^2}{1-\rho}\|D^0_{\perp}\|^2_F \Bigg]
            \\
        & + 
        \frac{\eta_y^2}{\alpha (1-\rho)}
        \bE\Bigg[
        \Big(
         (\frac{3L^2\yy{\rho^2}}{(1-\rho)S_2} +         \xzrr{\frac{288\tilde{q}\eta_y^2L^4\rho^2}{(1-\rho)S_2}})4\eta_y^2 
        +
        ( \rho + \frac{4L^2\yy{\rho^2}\eta_y^2}{1-\rho} + \xzrr{\frac{24\eta_y^2L^2\yy{\rho^2}}{1-\rho}} \frac{8(q-1)L^2\eta_y^2}{S_2})\Big)\|D^0_\perp\|^2_F
        \\
        &
        +
        \sum_{j=0}^{q-2}
        \Big(
        {b}_\delta\delta_j
        +
        {b}_\Phi \|\grad \Phi(\bbx^j)\|^2
        +
        {b}_Z\|Z^j_\perp\|^2_F 
        \Big)
        +
        \sum_{n=1}^{N-1}\sum_{j=(n-1)q}^{(n+1)q-2}
        \Big(
        \hat{b}_\delta\delta_j
        +
        \hat{b}_\Phi \|\grad \Phi(\bbx^j)\|^2
        +
        \hat{b}_Z\|Z^j_\perp\|^2_F 
        \Big)
        \\
        &
+
        (N-1)\frac{2M\sigma^2\yy{\rho^2} }{(1-\rho)S_1}
          \Big(
          ((q-1) +6\Tilde{q}) \xzrr{24L^2\eta_y^2} + 3
          \Big) + \xzrr{\frac{24\eta_y^2L^2\yy{\rho^2}}{1-\rho}}\frac{2M\sigma^2}{S_1}(q-1)\Bigg]
        \end{aligned}
    \end{equation*}
    holds for all $T\geq 1$ \sa{such that} $T=Nq$ for some $N\in\mathbb{N}^+$.
    Hence, we further have that 
    \begin{equation*}
        \begin{aligned}
            (1- & \rho-
            \frac{ 2\eta_y^2 \hat{b}_Z }{\alpha(1-\rho )} )\bE\Bigg[\sum_{t=1}^T  \| Z^t_\perp \|^2_F\Bigg]
            \leq 
           \bE\bigg[ (\rho + \frac{2\eta_y^2\hat{b}_Z}{\alpha(1-\rho)})
            \|Z^0_\perp\|^2_F + \frac{\eta_y^2}{1-\rho}\|D^0_\perp\|^2_F \bigg]
            \\
            &
             + 
        \frac{\eta_y^2}{\alpha (1-\rho)}
        \bE\Bigg[
        \Big(
         (\frac{3L^2\yy{\rho^2}}{(1-\rho)S_2} +         \xzrr{\frac{288\tilde{q}\eta_y^2L^4\rho^2}{(1-\rho)S_2}})
        4\eta_y^2 
        +
        ( \rho + \frac{4L^2\yy{\rho^2}\eta_y^2}{1-\rho} + \xzrr{\frac{24\eta_y^2L^2\yy{\rho^2}}{1-\rho}} \frac{8(q-1)L^2\eta_y^2}{S_2})\Big)\|D^0_\perp\|^2_F
        \\
        &
        +
        \sum_{j=0}^{q-2}
        \Big(
        {b}_\delta\delta_j
        +
        {b}_\Phi \|\grad \Phi(\bbx^j)\|^2
        \Big)
        +
        \sum_{n=1}^{N-1}\sum_{j=(n-1)q}^{(n+1)q-2}
        \Big(
        \hat{b}_\delta\delta_j
        +
        \hat{b}_\Phi \|\grad \Phi(\bbx^j)\|^2
        \Big)
        \\
        &
+
        (N-1)\frac{2M\sigma^2\yy{\rho^2} }{(1-\rho)S_1}
          \Big(
          ((q-1) +6\Tilde{q}) \xzrr{24L^2\eta_y^2} + 3
          \Big) + \xzrr{\frac{24\eta_y^2L^2\yy{\rho^2}}{1-\rho}}\frac{2M\sigma^2}{S_1}(q-1)\Bigg]
        \end{aligned}
    \end{equation*}
    holds for all $T\geq 1$ \sa{such that} $T=Nq$ for some $N\in\mathbb{N}^+$; \sa{hence, we obtain the desired result in~\eqref{eq:bound-Z-perp-final}.} 
    \xz{Next, \sa{$\alpha<1$,} it follows from  \cref{eq:bound-D-perp-15} and $\hat{b}_Z\geq b_Z$ that the inequality}
    \begin{equation}\label{eq:bound-D-perp-*}
 \begin{aligned}
        & \sa{\bE\Bigg[ \alpha \sum_{j=1}^{Nq-1} \|D_\perp^{j}\|^2_F\Bigg]}  
        \\
        \leq &
        \bE\Bigg[\Big(
         (\frac{3L^2}{(1-\rho)S_2} +  \xzrr{\frac{288\tilde{q}\eta_y^2L^4}{(1-\rho)S_2}})
        4\yy{\rho^2} \eta_y^2
        +
        ( \rho + \frac{4L^2\yy{\rho^2}\eta_y^2}{1-\rho} + \xzrr{\frac{24\eta_y^2L^2\yy{\rho^2}}{1-\rho}} \frac{8(q-1)L^2\eta_y^2}{S_2})\Big)\|D^0_\perp\|^2_F
        \\
        &
        +
        \sum_{j=0}^{q-2}
        \Big(
        {b}_\delta\delta_j
        +
        {b}_\Phi \|\grad \Phi(\bbx^j)\|^2
        \Big)
        +
        \sum_{n=1}^{N-1}\sum_{j=(n-1)q}^{(n+1)q-2}
        \Big(
        \hat{b}_\delta\delta_j
        +
        \hat{b}_\Phi \|\grad \Phi(\bbx^j)\|^2
        \Big)
        \\
        & 
        + 2\hat{b}_Z\|Z^0_\perp\|^2_F + 2\hat{b}_Z \sum_{j=1}^{Nq-2}\|Z^j_\perp\|^2_F 
        \\
        &
      +
        (N-1)\frac{2M\sigma^2\yy{\rho^2} }{(1-\rho)S_1}
          \Big(
          ((q-1) +6\Tilde{q}) \xzrr{24L^2\eta_y^2} + 3
          \Big) + \xzrr{\frac{24\eta_y^2L^2\yy{\rho^2}}{1-\rho}}\frac{2M\sigma^2}{S_1}(q-1)\Bigg]
        \end{aligned}
\end{equation}
\xz{holds for all $N\in\mathbb{N}^+$. Moreover,  we use \cref{eq:bound-Z-perp-final} together with {$1-\rho> \frac{2\hat{b}_Z\eta_y^2}{\alpha(1-\rho)}$} within the above inequality
    and obtain that
    \cref{eq:bound-D-perp-final}
     holds for all $T\geq 1$ \sa{such that} $T=Nq$ for some $N\in\mathbb{N}^+$, which completes the proof.}
     \end{proof}

\section{Parameter Conditions}
\xz{In this section, we present two parameter conditions that are employed in our analysis. Specifically, \cref{param-condition:basic} is utilized in the proof before this section. Subsequently, \cref{param-condition:basic-2} is employed to simplify the constant terms that appear in the above analysis. As a result, we obtain \cref{thm:main-convergence}. It is worth noting that  \cref{param-condition:basic-2} implies \cref{param-condition:basic} hold. 
}
\begin{param}\label{param-condition:basic}
Suppose $\eta_x,\eta_y > 0$ and $q, S_1, S_2\in\mathbb{N}^+$ satisfy the following conditions
\begin{enumerate}[label=\emph{(\roman*)}]
    \item \xzrr{$\eta_y \leq \frac{1}{L}$, and $\eta_x\leq \frac{1}{\sqrt{3}\kappa}\eta_y$} 
    \item $q\leq \frac{1}{2c_1}$ \sa{for} $c_1 \triangleq \frac{24\eta_y^2L^2}{S_2}$ \sa{such that $c_1\leq 1$},
    \item  $1- \rho > 
            \frac{ 2\eta_y^2 \hat{b}_Z }{\alpha(1-\rho )}$,
\end{enumerate}
where $\hat{b}_Z$ is defined in \cref{eq:def-b-hat-coeff}, $\alpha$ is defined in \cref{eq:def-alpha} and $\rho$ is defined in \cref{ASPT:mixture-matrix}.
\end{param}

\xz{In the next lemma, we summarize the frequently employed inequalities related to \cref{param-condition:basic} that was used in
\xzref{the section Convergence Analysis}
for improved  readability.}%
\begin{lemma}\label{lemma:parameter-ineq}
    \sa{If $\eta_x,\eta_y > 0$, and $q\in \mathbb{N}^+$ satisfy \cref{param-condition:basic}, then}
    \begin{enumerate}
        \item[(i)] 
        $\frac{4-\eta_y\mu}{4-2\eta_y\mu}\leq 2$;
        \item[(ii)] 
        $(1+c_1)^q\leq 2$;
        \item[(iii)] 
        $3(\frac{4}{\eta_y\mu}-1)\kappa^2\eta_x^2\leq \frac{4\eta_y}{\mu}$.
    \end{enumerate}
\end{lemma}%
\xz{The following parameter condition will be used in 
\xzref{the section Complexity Analysis}
to obtain the  convergence results in \cref{thm:main-convergence}. It implies \cref{param-condition:basic}. Specifically, we later show that $1- \rho > 
            \frac{ 2\eta_y^2 \hat{b}_Z }{\alpha(1-\rho )}$ hold if \cref{param-condition:basic-2} are satisfied in \cref{eq:hat-values}.
Therefore, all stepsizes assumptions of the previous analysis hold with \cref{param-condition:basic-2}.}
\begin{param}\label{param-condition:basic-2}
Suppose $\eta_x,\eta_y\geq 0$ and $q, S_1, S_2\in\mathbb{N}^+$ satisfy the following conditions
\begin{subequations}
\begin{align}
&\eta_y\leq \frac{1}{32\sqrt{5}L}\min\{(1-\rho)^2, \frac{1}{\kappa},\frac{1-\rho}{\sqrt{\kappa}}\}, \label{eq:stepsize-etay}\\
        & \eta_x\leq \min\{{\frac{1}{6(\kappa+1)L}},~{\frac{\eta_y}{8\sqrt{3}\kappa^2}},
        \frac{4}{5\mu}\eta_y,\frac{1}{64\kappa^2}\eta_y
    \},\\
        & S_2 \geq q,\quad \xz{q\geq 1.}
\end{align}
\end{subequations}
\begin{remark}
    \xz{The redundant conditions are kept for the purpose of facilitating the verification of the conditions used in the  analysis in 
    \xzref{the section Complexity Analysis}
    . Furthermore, the above-mentioned parameter conditions are condensed in our final result, as shown by \eqref{eq:final-step-size}.}
\end{remark}
\end{param}

\section{Complexity Analysis}\label{sec:main-proof-compelxity}
\xz{In this section, we first use \cref{param-condition:basic-2} to simplify the constant in our previous analysis, and then obtain the convergence result in \cref{thm:main-convergence}.}
\begin{lemma}\label{lemma:bound-alpha}
Suppose \cref{param-condition:basic-2} holds. Then it holds that $\alpha \geq \frac{1-\rho}{2}$,
    where $\alpha$ is defined in \cref{eq:def-alpha}.
\end{lemma}
\begin{proof}
    Because $S_2\geq q$, then we obtain
    \begin{equation}\label{eq:bound-tilde-q}
        \tilde{q} = 1
    \end{equation}
    where $\tilde{q}$ is defined in \cref{eq:def-tilde-q}. 
    Secondly, \cref{eq:stepsize-etay} implies $ \frac{12L^2\eta_y^2\rho^2}{1-\rho}\leq \frac{1-\rho}{2}$; therefore, we obtain that
    \begin{equation}\label{eq:bound-alpha1}
      \alpha_1 \geq \frac{(1-\rho)}{2},
    \end{equation}
    where $\alpha_1$ is defined in \cref{eq:def-alpha}. We \sa{now} continue to show $\alpha_2\geq \frac{1-\rho}{2}$, where $\alpha_2$ is defined in \cref{eq:def-alpha}. Indeed, it holds that
    \begin{equation}
    \begin{aligned}
               \alpha_2 =&1-\rho -\frac{4L^2\eta_y^2\rho^2}{1-\rho} - \xzrr{\frac{24\eta_y^2L^2\yy{\rho^2}}{1-\rho}} \frac{8(q-1)L^2\eta_y^2}{S_2} -
         \yy{\rho^2}\Big(\frac{3L^2}{(1-\rho)S_2} + \xzrr{\frac{288\tilde{q}\eta_y^2L^4}{(1-\rho)S_2}}\Big)
        4\eta_y^2 
        \\
        =&1-\rho - \frac{4L^2\eta_y^2\rho^2}{1-\rho}\Big(
        1 + \frac{48(q-1)L^2\eta_y^2}{S_2} + \frac{3}{S_2} +\frac{288\tilde{q}L^2\eta_y^2}{S_2}
        \Big)
        \\
        \geq & \frac{1-\rho}{2}
    \end{aligned}
    \end{equation}
    where use $\eta_y\leq \frac{1}{32\sqrt{5}}\cdot\frac{(1-\rho)^2}{ L\kappa}$ and $S_2\geq q\geq 1$. Therefore, we conclude that $ 
    \alpha = \min\{\alpha_1,\alpha_2\} \geq \frac{1-\rho}{2}.
    $
\end{proof}

\begin{lemma}\label{lemma:bhat-values}
Suppose \cref{param-condition:basic-2} holds. Then it holds that
    \begin{equation}
    \label{eq:hat-values}
        \hat{b}_\delta\leq\xzrr{\frac{120ML^4\rho^2\eta_y^2}{1-\rho}},\quad
        \hat{b}_\Phi\leq\frac{\xzrr{72}ML^2\yy{\rho^2}\eta_x^2}{1-\rho} ,\quad
        \hat{b}_Z\leq\xzrr{\frac{60L^2\rho^2}{1-\rho}},\quad
        \frac{2\hat{b}_Z\eta_y^2}{\alpha(1-\rho)}\leq\frac{1-\rho}{2},
    \end{equation}
    where $\{\hat{b}_\delta,\hat{b}_\Phi,\hat{b}_Z\}$ are defined in \cref{eq:def-b-hat-coeff}. 
\end{lemma}
\begin{proof}
\xzrr{It follows from the definition of $\hat{b}_\delta$ and $S_2\geq q,\eta_y\leq \xzy{\frac{1}{32\sqrt{5}}\cdot\frac{(1-\rho)^2}{ L}}$ that
\begin{equation}
    \hat{b}_\delta \leq  \yy{\rho^2}\Big(\frac{3L^2}{(1-\rho)} +  \xzrr{\frac{288\eta_y^2L^4}{(1-\rho)}}\Big) 
           \xzrr{20ML^2\eta_y^2} \leq  \yy{\rho^2}\cdot\frac{6L^2}{(1-\rho)}\cdot
           \xzrr{20ML^2\eta_y^2} = \frac{120ML^4\rho^2\eta_y^2}{1-\rho}
\end{equation}
}%
Similarly, it hold that
\begin{equation}
   \hat{b}_\Phi\leq \yy{\rho^2}\Big(\frac{3L^2}{(1-\rho)} +  \xzrr{\frac{288\eta_y^2L^4}{(1-\rho)}}\Big)
            12M\eta_x^2
            \leq \yy{\rho^2}\cdot\frac{6L^2}{(1-\rho)} \cdot
            12M\eta_x^2 = \frac{72ML^2\rho^2\eta_x^2}{1-\rho} 
\end{equation}
and
\begin{equation}
   \hat{b}_Z\leq \yy{\rho^2}\Big(\frac{3L^2}{(1-\rho)} +  \xzrr{\frac{288\eta_y^2L^4}{(1-\rho)}}\Big)
           \xzrr{(8+40\eta_y^2L^2)}  
           \leq 
           \yy{\rho^2}\cdot\frac{6L^2}{(1-\rho)}\cdot
           \xzrr{(8+40\eta_y^2L^2)}  \leq \frac{60L^2\rho^2}{1-\rho}.
\end{equation}
The last inequality in \cref{eq:hat-values} directly follows from $\xzy{\frac{1}{32\sqrt{5}}\cdot\frac{(1-\rho)^2}{ L}}$ and the above bound of $\hat{b}_Z$ and \cref{lemma:bound-alpha}.
\end{proof}

\begin{lemma}\label{lemma:bound:az-cz-a0-cphi}
\sa{If \cref{param-condition:basic-2} holds, then}
   \xzrr{\begin{equation}\label{eq:bound-ac-all}
        a_Z \leq \frac{176L\kappa\eta_y}{M}, \quad 
        a_\Phi\leq  \frac{13\eta_x^2\kappa^2}{\eta_y\mu}, \quad
        a_\delta \geq \frac{1}{8}\eta_y\mu,\quad
        c_Z\leq \frac{105L^2\eta_x}{2M},
        \quad
        c_\Phi\geq\frac{\eta_x}{12},\quad
       c_\delta\leq \frac{3}{2}\eta_x L^2
    \end{equation}}
    where \sa{$\{a_Z,~a_\Phi,~a_\delta\}$} are defined in \cref{eq:def-a-coeff}, $\{c_Z,c_\Phi,c_\delta\}$ are defined in \cref{eq:def-c-coeff}.
\end{lemma}
\begin{proof}
    We begin the proof by showing an upper bound of $a_Z$.
    First, it follows from  {$\eta_x\leq \frac{\eta_y}{\sqrt{3}\kappa}$} that 
    \begin{equation}\label{eq:a-z-part-1}
        \Big(\frac{4\eta_y}{\mu} + \frac{12}{\eta_y\mu}\kappa^2\eta_x^2\Big)\frac{2L^2}{M} = \Big(1 + \frac{3\kappa^2\eta_x^2}{\eta_y^2}\Big) \frac{8\kappa L\eta_y}{M}
        \leq  \frac{16L\kappa\eta_y}{M}.
    \end{equation}
    Secondly, it follows from $S_2\geq q$ and \xzy{$\eta_y\leq \frac{1}{32\sqrt{5}}\cdot\frac{1}{L\kappa}$} that
    \begin{equation}\label{eq:a-z-part-2}
    \begin{aligned}
                \frac{16(q-1)\eta_y}{\mu M}\frac{L^2}{S_2}  \xzrr{(8+40\eta_y^2L^2)} 
                \leq \frac{16}{ M}\cdot\eta_yL\kappa \cdot  \xzrr{(8+40\eta_y^2L^2)} 
                \leq   \frac{160L\kappa\eta_y}{ M}
    \end{aligned}
    \end{equation}
    Furthermore,  \cref{eq:a-z-part-1} together with \cref{eq:a-z-part-2} implies that 
    \begin{equation}
        \label{eq:bound-az}
        a_Z\leq \frac{176L\kappa\eta_y}{M}.
    \end{equation}
    Next, we continue to show an upper bound of $a_\Phi$. Indeed, it follows from  $S_2\geq q$ and \xzy{$\eta_y\leq \frac{1}{32\sqrt{5}}\cdot\frac{1}{ L\kappa}$} that
\begin{equation}
\begin{aligned}
        a_\Phi = & \frac{12}{\eta_y\mu}\kappa^2\eta_x^2 + \frac{16(q-1)\eta_y}{\mu}\frac{L^2}{S_2}12\eta_x^2
        \\
        \leq &\frac{12\eta_x^2}{\eta_y\mu}(\kappa^2 + 16L^2\eta_y^2\frac{q-1}{S_2}) 
        \\
        \leq & \frac{13\eta_x^2\kappa^2}{\eta_y\mu}
\end{aligned}
\end{equation}
    
  Next, we continue to an upper bound of $c_Z$. First, the condition {$\eta_x\leq \frac{1}{6(\kappa+1)L}$} implies
    \begin{equation}\label{eq:bound-etax-square}
        \frac{3(\kappa+1)L\eta_x^2}{2} \leq \xz{\frac{\eta_x}{4}}.
    \end{equation}
  Then, it follows from \xzy{$\eta_y\leq \frac{1}{32\sqrt{5}}\cdot\frac{1}{ L\kappa}$} and $S_2\geq q$ and \cref{eq:bound-etax-square} that
    \begin{equation}
    \begin{aligned}
                c_Z =& \frac{2L^2}{M}\Big(\eta_x + \frac{3(\kappa +1)L\eta_x^2}{2}\Big)
        +
        \frac{4(q-1)}{M}\Big(\eta_x + \frac{3(\kappa +1)L\eta_x^2}{2}\Big)\frac{L^2}{S_2}\xzrr{(8+40\eta_y^2L^2)} \\
        \leq &
        \frac{2L^2}{M}\frac{5\eta_x}{4}
        +
        \frac{4(q-1)}{M}\frac{5\eta_x}{4}\frac{L^2}{S_2}\cdot10
        \leq  \frac{5L^2\eta_x}{2M}
        +
        \frac{4}{M}\cdot\frac{5\eta_x}{4}\cdot L^2\cdot10
        \\
        \leq & \frac{105L^2\eta_x}{2M}.
    \end{aligned}
    \end{equation}
    Next, we continue to show a lower bound on $c_\Phi$. Indeed, it follows from  $S_2\geq q$, $\eta_x\leq \frac{1}{6(\kappa+1)L}\leq \frac{1}{12L}$ and  \cref{eq:bound-etax-square} that
    \begin{equation}\label{eq:bound-c_phi}
        \begin{aligned}
          c_\Phi = & (\frac{\eta_x}{2} - \frac{3(\kappa +1)L\eta_x^2}{2})
        -\frac{4(q-1)}{M}\Big(\eta_x + \frac{3(\kappa +1)L\eta_x^2}{2}\Big)\frac{L^2}{S_2}12M\eta_x^2
        \\
        \geq& \frac{\eta_x}{2} - \frac{\eta_x}{4} - 60L^2\eta_x^3
        \\
        \geq &\frac{\eta_x}{12},
        \end{aligned}
    \end{equation}%
    The rest of the proof follows a similar idea to the above proof. Indeed, we have that
    \begin{equation}
        \begin{aligned}
             a_\delta =&  \frac{1}{4}\eta_y\mu - \frac{12}{\eta_y\mu}\kappa^2\eta_x^2 L^2
        - \xzrr{\frac{320(q-1)L^4\eta_y^3}{\mu S_2}}
         \\
         = & \frac{1}{4}\eta_y\mu\Big(1-\frac{48\kappa^4\eta_x^2}{\eta_y^2} - \frac{1280(q-1)L^2\kappa^2\eta_y^2}{S_2}\Big) 
         \\
         \geq & \frac{1}{8}\eta_y\mu
    \end{aligned}
    \end{equation}
    where we use the condition that \xzy{$\eta_y\leq\frac{1}{32\sqrt{5}}\cdot\frac{1}{ L\kappa}$} and $S_2\geq q$ and $\eta_x\leq\frac{\eta_y}{8\sqrt{3}\kappa^2}$.
    Similarly, we have
    \begin{equation}\label{eq:diff-adelta-cdelta}
    \begin{aligned}
        c_\delta= &
         \Big(\eta_x + \frac{3(\kappa +1)L\eta_x^2}{2}\Big)L^2 + \frac{4(q-1)}{M}\Big(\eta_x + \frac{3(\kappa +1)L\eta_x^2}{2}\Big)\frac{L^2}{S_2}\xzrr{20ML^2\eta_y^2}
         \\
         \leq &
         \frac{5\eta_x}{4}L^2 +\frac{ \xzrr{100L^2\eta_x}}{5120}
         \\
         \leq &\frac{3}{2}\eta_xL^2.
    \end{aligned}
    \end{equation}
     where we use the condition that  $S_2\geq q$, \xzy{$\eta_y\leq \frac{1}{32\sqrt{5}}\cdot\frac{1}{ L\kappa}$} and \cref{eq:bound-etax-square}.
\end{proof}

\begin{lemma}
If \cref{param-condition:basic-2} holds, then
\begin{subequations}
    \begin{align}
       &   C_{0,Z}
         \leq
        \|Z_\perp^0\|_F^2+
          \|D^0_\perp\|^2_F
         + \frac{\eta_y^2 M\sigma^2}{S_1}\frac{\rho^2}{(1-\rho)^3}\cdot N(3q/80+16),\label{eq:bound-CZ-final}
         \\ 
         & 
        C_{0,D} \leq
        \frac{\xzrr{240}\rho^2L^2}{(1-\rho)^2}
        \|Z^0_\perp\|^2_F
         + 
         \frac{201}{100}\|D^0_\perp\|^2_F
         +
         \frac{M\sigma^2}{S_1}\frac{\rho^2}{(1-\rho)}\cdot N(3q/80+16),
        \label{eq:bound-CD-final}
        \\
        & N\Big[\Big(\eta_x + \frac{3(\kappa +1)L\eta_x^2}{2}\Big)\frac{4q\sigma^2}{S_1} + \frac{16\eta_yq\sigma^2}{\mu S_1}\Big]
         \leq  \frac{20N\eta_yq\sigma^2}{\mu S_1}
        \label{eq:bound-Nsigma-final}
    \end{align}
\end{subequations}
where $C_{0,Z}$ and $C_{0,D}$ are defined in \cref{eq:def-C-DZ}.
\end{lemma}
\begin{proof}
    In this proof, we will analyze each component of $C_{0,Z}$ and $C_{0,D}$ separately. It follow from \cref{eq:hat-values} that 
\begin{equation}\label{eq:bound-rho-etay-square-bhatz}
     \rho + \frac{2\eta_y^2\hat{b}_Z}{\alpha(1-\rho)} \leq  1.
    \end{equation}
    In addition, we can compute that
\begin{equation}\label{eq:bound-long-1}
        \begin{aligned}
            & 
         \left(\frac{3L^2\yy{\rho^2}}{(1-\rho)S_2} +         \xzrr{\frac{288\tilde{q}\eta_y^2L^4\rho^2}{(1-\rho)S_2}}\right)
        4\eta_y^2 
        =  \frac{12L^2\rho^2\eta_y^2}{(1-\rho)S_2} + \frac{\xzrr{1152}L^4\rho^2\eta_y^4}{(1-\rho)S_2}
        \\
        = & \frac{12L^2\rho^2}{(1-\rho)}\frac{\eta_y^2}{S_2}\Big(
        1 + 96L^2\eta_y^2\Big) \leq \frac{\rho^2}{\xzrr{400}}
        \end{aligned}
\end{equation}
     where the first equality is by \cref{eq:bound-tilde-q}, i.e., $\tilde{q}= 1$, the first inequality is by \xzy{$\eta_y\leq\frac{1}{32\sqrt{5}}\cdot\frac{(1-\rho)^2}{ L}$}. Furthermore,
        \begin{equation}\label{eq:bound-long-2}
        \begin{aligned}
            &\rho + \frac{4L^2\yy{\rho^2}\eta_y^2}{1-\rho} + \xzrr{\frac{24\eta_y^2L^2\yy{\rho^2}}{1-\rho}} \frac{8(q-1)L^2\eta_y^2}{S_2}
            \\
            = &\rho \Big[
            1 + \frac{4L^2\rho}{1-\rho}\eta_y^2\Big(
            1+
            \frac{48(q-1)L^2\eta_y^2}{S_2}
            \Big)
            \Big]
            \\
            \leq & \rho \Big[
            1 + \frac{4L^2\rho}{1-\rho}\eta_y^2(1+48L^2\eta_y^2)
            \Big]
            \\
            \leq & \frac{501}{500}\rho
        \end{aligned}
    \end{equation}
     where the first inequality is  by   $S_2\geq q$, the second inequality is by \xzy{$\eta_y\leq \frac{1}{32\sqrt{5}}\cdot\frac{(1-\rho)^2}{ L}$}. Therefore, by the definition of $\alpha$ in \cref{eq:def-alpha}, we conclude that
     \begin{equation}\label{eq:long-2.5}
         \sa{\alpha+}
    \Big(\frac{3L^2\yy{\rho^2}}{(1-\rho)S_2} +         \xzrr{\frac{288\tilde{q}\eta_y^2L^4\rho^2}{(1-\rho)S_2}}\Big)
        4\eta_y^2 
        +
        \Big( \rho + \frac{4L^2\yy{\rho^2}\eta_y^2}{1-\rho} + \xzrr{\frac{24\eta_y^2L^2\yy{\rho^2}}{1-\rho}} \frac{8(q-1)L^2\eta_y^2}{S_2}\Big) 
        \leq 1-\rho +\frac{\rho^2}{400} +\frac{501}{500}\rho\leq\frac{201}{200}
     \end{equation}
  In addition, we can compute that
\begin{equation}\label{eq:bound-long-3}
\begin{aligned}
        &  (N-1)\frac{2\yy{\rho^2}M\sigma^2}{(1-\rho)S_1}
          \Big(
          ((q-1) +6\Tilde{q}) \xzrr{24L^2\eta_y^2}+ 3
          \Big) +  \xzrr{\frac{24\eta_y^2L^2\yy{\rho^2}}{1-\rho}}\frac{2M\sigma^2}{S_1}(q-1)
            \\
            = 
            & \frac{2 M\sigma^2}{S_1}\frac{\rho^2}{1-\rho}\Big[
            (N-1)\Big(
          ((q-1) +6\tilde{q}) \xzrr{24L^2\eta_y^2} + \xz{3}
          \Big) + \xzrr{24L^2\eta_y^2}(q-1)\Big]
          \\
          \leq &  \frac{2 M\sigma^2}{S_1}\frac{\rho^2}{(1-\rho)}\cdot N\Big[
           \frac{3q}{640} + \frac{9}{320} + 3 + \frac{3q}{640}
           \Big]
            \\
            \leq & \frac{2M\sigma^2}{S_1}\frac{\rho^2}{(1-\rho)}\cdot N(\frac{3q}{320}+4)
\end{aligned}
\end{equation}
 where the first inequality is by \cref{eq:bound-tilde-q} and \xzy{$\eta_y\leq\frac{1}{32\sqrt{5}}\cdot\frac{1}{ L\kappa}$}.
 Hence, from the definition of $C_{0,Z}$ in \cref{eq:def-C-DZ},  combining \cref{eq:bound-rho-etay-square-bhatz,eq:bound-long-1,eq:bound-long-2,eq:long-2.5,eq:bound-long-3} and using 
\xzy{$\eta_y\leq\frac{1}{32\sqrt{5}}\cdot\frac{(1-\rho)^2}{ L}$} and \cref{lemma:bound-alpha} implies that
 \begin{equation}
     \begin{aligned}
        C_{0,Z}
         \leq&  
        \|Z_\perp^0\|_F^2+
          \|D^0_\perp\|^2_F
         + N(3q/80+16)\cdot\frac{\eta_y^2 M\sigma^2}{S_1}\frac{\rho^2}{(1-\rho)^3},
     \end{aligned}
 \end{equation}
which complete the proof for \cref{eq:bound-CZ-final}.

Next, we will follow a similar idea to prove \cref{eq:bound-CD-final}. Indeed, we already proved upper bounds of most components of $C_{0,D}$ defined in \cref{eq:def-C-DZ}.  It follows from \cref{eq:hat-values} that
\begin{equation}\label{eq:bound-hatbz-all}
   (1-\rho -  \frac{ 2\eta_y^2 \hat{b}_Z }{\alpha(1-\rho )})^{-1}\leq \frac{2}{1-\rho},\quad
            C_1 = 1 +\frac{2\eta_y^2\hat{b}_Z}{\alpha(1-\rho)}\Big / (1-\rho -  \frac{ 2\eta_y^2 \hat{b}_Z }{\alpha(1-\rho )})
            \leq 2
\end{equation}
where $C_1$ is defined in \cref{lemma:bound-Z-D-perp-final}.
Therefore, by \cref{eq:hat-values}, we further have
\begin{equation}\label{eq:bound-CD-long-1}
\begin{aligned}
            2\hat{b}_Z\Big(\Big(1- \rho- \frac{ 2\eta_y^2 \hat{b}_Z }{\alpha(1-\rho )}\Big)^{-1}
    \Big(\rho + \frac{ 2\eta_y^2 \hat{b}_Z }{\alpha(1-\rho )}\Big) +1\Big) 
    \leq  
    2\hat{b}_Z\Big(\frac{2}{1-\rho}
    \Big(\rho +\frac{1-\rho}{2}\Big) +1\Big) 
    =  \frac{4\hat{b}_Z}{1-\rho}\leq \frac{240L^2\rho^2}{(1-\rho)^2}
\end{aligned}
\end{equation}
Next, by \cref{eq:long-2.5,eq:bound-hatbz-all}, another term of $C_{0,D}$ can be bounded as
\begin{equation}
    \begin{aligned}\label{eq:bound-CD-long-2}
 \sa{C_1}
             \Big(
         \sa{\alpha+}(\frac{3L^2\yy{\rho^2}}{(1-\rho)S_2} +         \xzrr{\frac{288\tilde{q}\eta_y^2L^4\rho^2}{(1-\rho)S_2}})
        4\eta_y^2 
         +
        ( \rho + \frac{4L^2\yy{\rho^2}\eta_y^2}{1-\rho} + \xzrr{\frac{24\eta_y^2L^2\yy{\rho^2}}{1-\rho}} \frac{8(q-1)L^2\eta_y^2}{S_2})\Big) \leq \frac{201}{100}
    \end{aligned}
\end{equation}
Then using \cref{eq:bound-CD-long-1,eq:bound-long-3,eq:bound-CD-long-2} and the definition of $C_{0,D}$ in \cref{eq:def-C-DZ} implies 
\begin{equation}
    C_{0,D} \leq
        \frac{\xzrr{240}\rho^2L^2}{(1-\rho)^2}
        \|Z^0_\perp\|^2_F
         + 
         \frac{201}{100}\|D^0_\perp\|^2_F
         +
         \frac{M\sigma^2}{S_1}\frac{\rho^2}{(1-\rho)}\cdot N(\frac{3q}{80}+16),
\end{equation}
which completes the analysis for \cref{eq:bound-CD-final}. Next, we move to the proof for \cref{eq:bound-Nsigma-final}. Indeed, it follows from \cref{eq:bound-etax-square} and $\eta_x\leq \frac{4}{5\mu}\eta_y$ that
\begin{equation}
    \begin{aligned}
         &N\Big[\Big(\eta_x + \frac{3(\kappa +1)L\eta_x^2}{2}\Big)\frac{4q\sigma^2}{S_1} + \frac{16\eta_yq\sigma^2}{\mu S_1}\Big]\leq  N\Big[\frac{5q\sigma^2\eta_x}{S_1} + \frac{16\eta_yq\sigma^2}{\mu S_1}\Big]
         \leq \frac{20N\eta_yq\sigma^2}{\mu S_1},
    \end{aligned}
\end{equation}
which completes the proof.
\end{proof}

\subsection{The Proof of Main Result}\label{sec:proof-main-result}
\xzrr{Havingh provided the essential bounds for the parameters as discussed above, we are now ready to combine our analysis and show the final convergence results.}
\begin{theorem}\label{thm:main-convergence}
   \xzrr{Suppose Assumptions~\ref{ASPT:SC},~\ref{ASPT:lower-bounded-Phi},~\ref{ASPT:general-sto-grad},~\ref{ASPT:smooth-F-mean-squared} and \ref{ASPT:mixture-matrix} hold. Moreover, $\{\eta_x,\eta_y\}$ and $\{S_1,S_2,q\}$  are chosen such that \cref{param-condition:basic-2} are satisfied. Then the following inequality holds 
    \begin{equation}\label{eq:convergence-explicit}
    \frac{1}{T}\sum_{j=0}^{T-1}\mathbb{E}\left[\|\grad\Phi(\bbx^j)\|^2\right]\leq \frac{1}{T}\cdot\Big(
    \frac{50}{\eta_x}\mathbb{E}\left[\Phi(\bbx^0)-\Phi(\bbx^T)\right] 
    +
    \frac{900L\kappa}{\eta_y}\delta_0
    +\frac{50\cdot9772L^2\kappa^2}{M}\cdot\Lambda_0 \Big)+  \frac{50\cdot323\kappa^2\sigma^2}{S_1}.
    \end{equation}
  for all $T\geq 0$ such that $T = N q$ for some $N \in\mathbb{N}^+$, where $\Delta_\Phi \triangleq \Phi(\bbx^0) - \min_{\bx}\Phi(\bx)$
  and $\Lambda_0=\|Z^0_\perp\|^2_F+\|D^0_\perp\|^2_F$ and $\delta_0$ is defined in \cref{eq:def-Delta-t}.}
\end{theorem}

\begin{proof}
    If we sum up \cref{eq:bound-deltat-final} from $n=0$ to $N-1$, it holds that
    \begin{equation}\label{eq:bound-deltat-sum-final}
    \begin{aligned}
         \mathbb{E}\Big[\delta_{Nq}\Big]+ &a_\delta\bE\left[\sum_{n=0}^{N-1}\delta_{nq} + \sum_{n=0}^{N-1}\sum_{t=nq+1}^{(n+1)q-1}\delta_{t} \right]
        \leq \bE\left[ 
        \delta_{0} 
         + N\frac{16\eta_y q\sigma^2}{\mu S_1} 
        + a_\Phi\sum_{n=0}^{N-1}
         \sum_{t=nq}^{(n+1)q-1}\|\grad \Phi(\bbx^{t})\|^2 \right.
         \\
         &
         \left.+ a_Z\sum_{n=0}^{N-1}
           \sum_{t=nq}^{(n+1)q-1} \|Z^t_\perp \|^2_F 
            +\frac{16(q-1)\eta_y}{\mu M} \frac{L^2}{S_2}4\eta_y^2\sum_{n=0}^{N-1}\sum_{t=nq}^{(n+1)q-2}\|D^t_\perp\|_F^2\right].
    \end{aligned}
\end{equation}
Moreover, it follows from \cref{eq:bound-Z-perp-final,eq:bound-hatbz-all} and $\hat{b}_\delta\geq  b_\delta,\hat{b}_\Phi\geq b_\Phi$ that
    \begin{equation}\label{eq:bound-Z-perp-final-2} 
        \begin{aligned}
   & \bE\Big[\sum_{j=1}^T  \| Z^j_\perp \|^2_F\Big] \\ 
    \leq & \frac{2}{1-\rho}\Big\{\yy{C_{0,Z}} 
+ \sa{\frac{\eta_y^2}{\alpha(1-\rho)}}
    \bE\bigg[
\sum_{j=0}^{q-2}
\Big(
{b}_\delta\delta_j
+
{b}_\Phi \|\grad \Phi(\bbx^j)\|^2
\Big)
+
\sum_{n=1}^{N-1}\sum_{j=(n-1)q}^{(n+1)q-2}
\Big(
\hat{b}_\delta\delta_j
+
\hat{b}_\Phi \|\grad \Phi(\bbx^j)\|^2
\Big)
\bigg]\Big\} \\
\leq & 
\frac{2}{1-\rho}\Big\{\yy{C_{0,Z}} 
+ \sa{\frac{\eta_y^2}{\alpha(1-\rho)}}
    \bE\bigg[
\sum_{j=0}^{T-2}
\Big(
2\hat{b}_\delta\delta_j
+
2\hat{b}_\Phi \|\grad \Phi(\bbx^j)\|^2
\Big)
\bigg]\Big\}
\end{aligned}
\end{equation}
Similarly, \cref{eq:bound-D-perp-final} also implies that
\begin{equation}\label{eq:bound-D-perp-final-2}
      \begin{aligned}
            & \bE\Big[\sum_{j = 1}^{T-1}\|D^j_\perp\|^2_F\Big]
            \leq \frac{\yy{C_{0,D}}}{\alpha}+ \frac{\sa{C_1}}{\alpha}
            \bE\bigg[ 
        \sum_{j=0}^{T-2}
        \Big(
        2\hat{b}_\delta\delta_j
        +
        2\hat{b}_\Phi \|\grad \Phi(\bbx^j)\|^2
        \Big)
       \bigg]
        \end{aligned}
\end{equation}
Moreover, using \cref{eq:bound-Z-perp-final-2,eq:bound-D-perp-final-2} and the fact 
$T=Nq$ within \cref{eq:bound-deltat-sum-final}, it follows that
\begin{equation}\label{eq:bound-deltat-sum-final-1.5}
    \begin{aligned}
    \bE\left[\delta_T+ a_\delta\sum_{j=0}^{T-1}\delta_{j} \right]
        \leq &\bE\left[ 
        \delta_{0} 
         + N\frac{16\eta_y q\sigma^2}{\mu S_1} 
         +a_Z\|Z^0_\perp\|^2_F
         +\frac{16(q-1)\eta_y}{\mu M}\frac{L^2}{S_2}4\eta_y^2\|D_\perp^0\|^2_F \right.
         \\
         & + \frac{2a_Z}{1-\rho}C_{0,Z}
         +\frac{16(q-1)\eta_y}{\mu M}\frac{L^2}{S_2}4\eta_y^2\frac{C_{0,D}}{\alpha}
         \\
         &+ 
         \Big(
         a_\Phi + \frac{2a_Z}{1-\rho}\frac{\eta_y^2}{\alpha(1-\rho)}2\hat{b}_\Phi + 
         \frac{16(q-1)\eta_y}{\mu M}\frac{L^2}{S_2}4\eta_y^2\frac{C_1}{\alpha}2\hat{b}_\Phi
         \Big)
         \sum_{j=0}^{T-1}\|\grad \Phi(\bbx^j)\|^2
         \\
         &
         \left.
         +\Big(
          \frac{2a_Z}{1-\rho}\frac{\eta_y^2}{\alpha(1-\rho)}2\hat{b}_\delta + 
         \frac{16(q-1)\eta_y}{\mu M}\frac{L^2}{S_2}4\eta_y^2\frac{C_1}{\alpha}2\hat{b}_\delta
         \Big)
         \sum_{j=0}^{T-2}\delta_j
         \right],
    \end{aligned}
\end{equation}
In addition, we can compute that
\begin{equation}\label{eq:adelta-minus-long-tail}
    \begin{aligned}
        & a_\delta -  \frac{2a_Z}{1-\rho}\frac{\eta_y^2}{\alpha(1-\rho)}2\hat{b}_\delta -
         \frac{16(q-1)\eta_y}{\mu M}\frac{L^2}{S_2}4\eta_y^2\frac{C_1}{\alpha}2\hat{b}_\delta
         \\
         \geq & \frac{1}{8}\eta_y\mu - \frac{352L\kappa\eta_y}{M(1-\rho)}\frac{2\eta_y^2}{(1-\rho)^2}\frac{240 ML^4\rho^2\eta_y^2}{1-\rho}-\frac{16(q-1)\eta_y}{\mu M}\frac{L^2}{S_2}4\eta_y^2 \frac{4}{1-\rho} \frac{240 ML^4\rho^2\eta_y^2}{1-\rho}
         \\
         =& \frac{1}{8}\eta_y\mu\Big(1 - \frac{16\cdot352\cdot240L^4\kappa^2\rho^2}{(1-\rho)^4}\eta_y^4 - \frac{8\cdot16\cdot16\cdot240L^4\kappa^2\rho^2}{(1-\rho)^2}\frac{q-1}{S_2}\eta_y^4\Big)
         \\
         \leq
        & \frac{1}{8}\eta_y\mu\Big(1 - \frac{16\cdot352\cdot240L^4\kappa^2\rho^2}{(1-\rho)^4}\eta_y^4 - \frac{8\cdot16\cdot16\cdot240L^4\kappa^2\rho^2}{(1-\rho)^2}\eta_y^4\Big)
         \\
         \leq & \frac{1}{9}\eta_y\mu
    \end{aligned}
\end{equation}
where the first inequality is by \cref{lemma:bound-alpha,eq:hat-values,eq:bound-ac-all,eq:bound-hatbz-all}; the second inequality is by $S_2\geq q$; the last inequality is by \xzy{$\eta_y\leq \frac{1}{32\sqrt{5}}\cdot\frac{1-\rho}{ L\sqrt{\kappa}}$}.
Therefore, if we use \cref{eq:adelta-minus-long-tail} and the fact $\frac{1}{9}\eta_y\mu\leq 1$ which implied by \cref{param-condition:basic-2} within \cref{eq:bound-deltat-sum-final-1.5}, we obtain that
\begin{equation}\label{eq:bound-deltat-sum-final-2}
    \begin{aligned}
    \bE\left[\frac{1}{9}\eta_y\mu\sum_{j=0}^{T}\delta_{j} 
    \right]
        \leq &\bE\left[ 
        \delta_{0} 
         + N\frac{16\eta_y q\sigma^2}{\mu S_1} 
         +a_Z\|Z^0_\perp\|^2_F
         +\frac{16(q-1)\eta_y}{\mu M}\frac{L^2}{S_2}4\eta_y^2\|D_\perp^0\|^2_F \right.
         \\
         & + \frac{2a_Z}{1-\rho}C_{0,Z}
         +\frac{16(q-1)\eta_y}{\mu M}\frac{L^2}{S_2}4\eta_y^2\frac{C_{0,D}}{\alpha}
         \\
         &
                  \left.+ 
         \Big(
         a_\Phi + \frac{2a_Z}{1-\rho}\frac{\eta_y^2}{\alpha(1-\rho)}2\hat{b}_\Phi + 
         \frac{16(q-1)\eta_y}{\mu M}\frac{L^2}{S_2}4\eta_y^2\frac{C_1}{\alpha}2\hat{b}_\Phi
         \Big)
         \sum_{j=0}^{T-1}\|\grad \Phi(\bbx^j)\|^2
         \right].
    \end{aligned}
\end{equation}
If we sum up \cref{eq:bound-Phi-final} from $n=0$ to $N-1$, it follows that
\begin{equation}
    \begin{aligned}
        \bE[\Phi(\bbx^{T})  ]
         \leq  & \bE\Big[\Phi(\bbx^{0}) 
         + \Big(\eta_x + \frac{3(\kappa +1)L\eta_x^2}{2}\Big)\frac{4q\sigma^2}{S_1}N
         - c_\Phi\sum_{t=0}^{T-1}\|\grad \Phi(\bbx^t)\|^2
          + c_\delta\sum_{t=0}^{T-1}\delta_t
          \\
          &
          + c_Z \sum_{t=0}^{T-1}\|Z^t_\perp \|^2_F
         +\frac{4(q-1)}{M}\Big(\eta_x + \frac{3(\kappa +1)L\eta_x^2}{2}\Big)\frac{L^2}{S_2}4\eta_y^2\sum_{t=0}^{T-2}\|D^t_\perp\|_F^2\Big].
    \end{aligned}
\end{equation}
Then substituting \cref{eq:bound-Z-perp-final-2} and \cref{eq:bound-D-perp-final-2} into above inequality, we obtain that
\begin{equation}
    \begin{aligned}
        \bE[ \Phi(\bbx^T) ] \leq &\bE\Bigg[ \Phi(\bbx^0) + (\eta_x+\frac{3(\kappa+1)L\eta_x^2}{2})\frac{4q\sigma^2}{S_1}N - c_\Phi\sum_{j=0}^{T-1}\|\grad\Phi(\bbx^j)\|^2 + c_Z\|Z^0_\perp\|^2_F \\
        & + \frac{2c_Z}{1-\rho}\Big\{
        C_{0,Z} +\frac{\eta_y^2}{\alpha(1-\rho)}
        \sum_{j=0}^{T-2}\big(2\hat{b}_\delta\delta_j+
        2\hat{b}_\Phi\|\grad\Phi(\bbx^j)\|^2
        \big)
        \Big\} \\
        &
        +\frac{4(q-1)}{M}(\eta_x + \frac{3(\kappa+1)L\eta_x^2}{2})\frac{L^2}{S_2}4\eta_y^2\|D^0_\perp\|^2_F 
        \\
        &
        +\frac{4(q-1)}{M}(\eta_x + \frac{3(\kappa+1)L\eta_x^2}{2})\frac{L^2}{ S_2}4\eta_y^2
        \Big\{
        \frac{C_{0,D}}{\alpha}
         +\frac{C_1}{\alpha}
       \sum_{j=0}^{T-2}
        (2\hat{b}_\delta\delta_j +2 \hat{b}_{\Phi}\|\grad \Phi(\bbx^j)\|^2)
        \Big\} + c_\delta\sum_{j=0}^{T-1}\delta_j\Bigg]
    \end{aligned}
\end{equation}
By rearranging terms, we get
\begin{equation}
\begin{aligned}
    &c_\Phi\bE\left[\sum_{j=0}^{T-1}\|\grad \Phi(\bbx^j)\|^2 \right]
    \\
    \leq&\bE\Bigg[ \Phi(\bbx^0) - \Phi(\bbx^T) +  (\eta_x+\frac{3(\kappa+1)L\eta_x^2}{2})\frac{4q\sigma^2}{S_1}N
    + c_Z\|Z^0_\perp\|^2_F 
    + \frac{2c_Z}{1-\rho}C_{0,Z}\\
    &
    +\frac{4(q-1)}{M}(\eta_x + \frac{3(\kappa+1)L\eta_x^2}{2})\frac{L^2}{S_2}4\eta_y^2\|D^0_\perp\|^2_F
        +\frac{4(q-1)}{M}(\eta_x + \frac{3(\kappa+1)L\eta_x^2}{2})\frac{L^2}{ S_2}4\eta_y^2
        \frac{C_{0,D}}{\alpha}
    \\
    & + \Big(\frac{2c_Z}{1-\rho}\frac{\eta_y^2}{\alpha(1-\rho)}2\hat{b}_\Phi + \frac{4(q-1)}{M}(\eta_x + \frac{3(\kappa+1)L\eta_x^2}{2})\frac{L^2}{ S_2}4\eta_y^2\frac{C_1}{\alpha}\cdot 2\hat{b}_\Phi    
    \Big) \sum_{j=0}^{T-2}\|\grad \Phi(\bbx^j)\|^2 \\
    & +  \Big(\frac{2c_Z}{1-\rho}\frac{\eta_y^2}{\alpha(1-\rho)}2\hat{b}_\delta + \frac{4(q-1)}{M}(\eta_x + \frac{3(\kappa+1)L\eta_x^2}{2})\frac{L^2}{ S_2}4\eta_y^2\frac{C_1}{\alpha}\cdot 2\hat{b}_\delta    
    +c_\delta \Big) \sum_{j=0}^{T-1}\delta_j \Bigg].
\end{aligned}
\end{equation}
Moreover, if we use \cref{eq:bound-deltat-sum-final-2} and the fact that $\delta_T\geq 0$  within the above inequality, it follows that
\begin{equation}\label{eq:long-sum-grad_Phi}
    \begin{aligned}
           & c_\Phi\bE\left[\sum_{j=0}^{T-1}\|\grad \Phi(\bbx^j)\|^2  \right]
           \\
           \leq& \Bigg[\Phi(\bbx^0) - \Phi(\bbx^T) +  (\eta_x+\frac{3(\kappa+1)L\eta_x^2}{2})\frac{4q\sigma^2}{S_1}N
    + c_Z\|Z^0_\perp\|^2_F 
    + \frac{2c_Z}{1-\rho}C_{0,Z}
    \\
    &
    +\frac{4(q-1)}{M}(\eta_x + \frac{3(\kappa+1)L\eta_x^2}{2})\frac{L^2}{S_2}4\eta_y^2\|D^0_\perp\|^2_F
        +\frac{4(q-1)}{M}(\eta_x + \frac{3(\kappa+1)L\eta_x^2}{2})\frac{L^2}{ S_2}4\eta_y^2
        \frac{C_{0,D}}{\alpha}
    \\
    & + \Big(\frac{2c_Z}{1-\rho}\frac{\eta_y^2}{\alpha(1-\rho)}2\hat{b}_\Phi + \frac{4(q-1)}{M}(\eta_x + \frac{3(\kappa+1)L\eta_x^2}{2})\frac{L^2}{ S_2}4\eta_y^2\frac{C_1}{\alpha}\cdot 2\hat{b}_\Phi    
    \Big) \sum_{j=0}^{T-2}\|\grad \Phi(\bbx^j)\|^2 \\
    & + \frac{9}{\eta_y\mu }\Big(\frac{2c_Z}{1-\rho}\frac{\eta_y^2}{\alpha(1-\rho)}2\hat{b}_\delta + \frac{4(q-1)}{M}(\eta_x + \frac{3(\kappa+1)L\eta_x^2}{2})\frac{L^2}{ S_2}4\eta_y^2\frac{C_1}{\alpha}\cdot 2\hat{b}_\delta   
    +c_\delta
    \Big)\cdot\Big\{
    \\
    &
     \delta_{0} 
         + N\frac{16\eta_y q\sigma^2}{\mu S_1} 
         +a_Z\|Z^0_\perp\|^2_F
         +\frac{16(q-1)\eta_y}{\mu M}\frac{L^2}{S_2}4\eta_y^2\|D_\perp^0\|^2_F 
         \\
         & + \frac{2a_Z}{1-\rho}C_{0,Z}
         +\frac{16(q-1)\eta_y}{\mu M}\frac{L^2}{S_2}4\eta_y^2\frac{1}{\alpha}C_{0,D}
         \\
         &+ 
         \Big(
         a_\Phi + \frac{2a_Z}{1-\rho}\frac{\eta_y^2}{\alpha(1-\rho)}2\hat{b}_\Phi + 
         \frac{16(q-1)\eta_y}{\mu M}\frac{L^2}{S_2}4\eta_y^2\frac{C_1}{\alpha}2\hat{b}_\Phi
         \Big)
         \sum_{j=0}^{T-2}\|\grad \Phi(\bbx^j)\|^2
    \Big\}\Bigg]
    \end{aligned}
\end{equation}
Next, we will bound the coefficients in the last inequality specifically. First,
     it follows from \cref{eq:bound-etax-square}  and $T=Nq$ that,
     \begin{equation}\label{eq:coeff-noname-1}
         (\eta_x+\frac{3(\kappa+1)L\eta_x^2}{2})\frac{4q\sigma^2}{S_1} N\leq \frac{5q\sigma^2}{S_1}\eta_x N =  \frac{5\sigma^2}{S_1}\eta_x T,\quad  N\frac{16\eta_y q\sigma^2}{\mu S_1}  = T\frac{16\eta_y \sigma^2}{\mu S_1}
     \end{equation}
     Secondly, it follows from \cref{eq:bound-etax-square}, \cref{lemma:bound-alpha} and $S_2\geq q$
     \begin{equation}\label{eq:bound-rv-noname1}
         \frac{4(q-1)}{M}(\eta_x + \frac{3(\kappa+1)L\eta_x^2}{2})\frac{L^2}{S_2}4\eta_y^2 \leq \frac{20L^2\eta_y^2\eta_x}{M},
         \quad
                  \frac{4(q-1)}{M}(\eta_x + \frac{3(\kappa+1)L\eta_x^2}{2})\frac{L^2}{ S_2}4\eta_y^2\frac{1}{\alpha } \leq  \frac{40L^2\eta_y^2\eta_x}{M(1-\rho)}.
     \end{equation}
     Thirdly, it follows from \cref{lemma:bound-alpha} and $S_2\geq q$ that
     \begin{equation}\label{eq:bound-noname-1.1}
         \frac{16(q-1)\eta_y}{\mu M}\frac{L^2}{S_2}4\eta_y^2\leq\frac{64L\kappa\eta_y^3}{M},
         \quad
         \frac{16(q-1)\eta_y}{\mu M}\frac{L^2}{S_2}4\eta_y^2\frac{1}{\alpha}\leq\frac{128L\kappa\eta_y^3}{M(1-\rho)},
     \end{equation}
     Fourthly, it follows from \cref{lemma:bound-alpha,eq:bound-etax-square,eq:hat-values,eq:bound-ac-all,eq:bound-hatbz-all} that
     \begin{equation}
  \begin{aligned}
    &\frac{2c_Z}{1-\rho}\frac{\eta_y^2}{\alpha(1-\rho)}2\hat{b}_\Phi + \frac{4(q-1)}{M}(\eta_x + \frac{3(\kappa+1)L\eta_x^2}{2})\frac{L^2}{ S_2}4\eta_y^2\frac{C_1}{\alpha}\cdot 2\hat{b}_\Phi \\
    \leq &
    \frac{105L^2\eta_x}{M(1-\rho)} \frac{2\eta_y^2}{(1-\rho)^2}
 \frac{144ML^2\rho^2\eta_x^2}{1-\rho} 
 +
 \frac{4}{M} \frac{5\eta_x}{4}L^24\eta_y^2\frac{4}{1-\rho}\frac{144ML^2\rho^2\eta_x^2}{1-\rho} 
 \\
 = &\frac{210\cdot144L^4\rho^2}{(1-\rho)^4}\eta_x^3\eta_y^2+ \frac{80\cdot144L^4\rho^2}{(1-\rho)^2}\eta_x^3\eta_y^2
  \\
  \leq &  \frac{290\cdot144L^4\rho^2}{(1-\rho)^4}\eta_x^3\eta_y^2.
        \end{aligned} 
     \end{equation}
Moreover, we have that
  \begin{equation}
     \begin{aligned}
                  &\frac{2c_Z}{1-\rho}\frac{\eta_y^2}{\alpha(1-\rho)}2\hat{b}_\delta + \frac{4(q-1)}{M}(\eta_x + \frac{3(\kappa+1)L\eta_x^2}{2})\frac{L^2}{ S_2}4\eta_y^2\frac{C_1}{\alpha} 2\hat{b}_\delta   +c_\delta 
                  \\
                  \leq & 
                  \frac{105L^2\eta_x}{(1-\rho)M} \frac{2\eta_y^2}{(1-\rho)^2}
 \frac{240ML^4\rho^2\eta_y^2}{1-\rho}
 +
 \frac{4}{M}\frac{5\eta_x}{4}L^2 4\eta_y^2\frac{4}{1-\rho} \frac{240ML^4\rho^2\eta_y^2}{1-\rho}
  + 
  \frac{3L^2\eta_x}{2}
    \\
    = &  \frac{210\cdot240L^6\rho^2}{(1-\rho)^4}\eta_x\eta_y^4
    +
    \frac{80\cdot240L^6\rho^2}{(1-\rho)^2}\eta_x\eta_y^4 + \frac{3L^2\eta_x}{2}
    \\
    =&\frac{3L^2\eta_x}{2}\Big(
    \frac{140\cdot240L^4\rho^2}{(1-\rho)^4}\eta_y^4
    +
    \frac{12800L^4\rho^2}{(1-\rho)^2}\eta_y^4 + 1
    \Big)
    \\
    \leq & 2L^2\eta_x,
     \end{aligned}
     \end{equation}
    where the first inequality is by \cref{lemma:bound-alpha,eq:bound-etax-square,eq:hat-values,eq:bound-ac-all,eq:bound-hatbz-all} and $S_2\geq q$ and the last inequality is by \xzy{$\eta_y\leq \frac{1}{32\sqrt{5}}\cdot\frac{(1-\rho)^2}{ L}$}.
     Next, we have that
     \begin{equation}\label{eq:coeff-noname-2}
         \begin{aligned}
              & a_\Phi + \frac{2a_Z}{1-\rho}\frac{\eta_y^2}{\alpha(1-\rho)}2\hat{b}_\Phi + 
         \frac{16(q-1)\eta_y}{\mu M}\frac{L^2}{S_2}4\eta_y^2\frac{C_1}{\alpha}2\hat{b}_\Phi
         \\
        \leq & \frac{13\eta_x^2\kappa^2}{\eta_y\mu} + \frac{2\cdot176L\kappa\eta_y}{M(1-\rho)}\frac{2\eta_y^2}{(1-\rho)^2}\frac{144ML^2\rho^2\eta_x^2}{1-\rho} + \frac{16\eta_y}{\mu M}L^2 4\eta_y^2\frac{4}{1-\rho}\frac{144ML^2\rho^2\eta_x^2}{1-\rho}
         \\
          = &\frac{13\kappa^2\eta_x^2}{\mu\eta_y}
          + \frac{4\cdot 144\cdot 176L^3\kappa\rho^2}{(1-\rho)^4}\eta_y^3\eta_x^2
          +
          \frac{16^2\cdot144L^3\kappa\rho^2}{(1-\rho)^2}\eta_y^3\eta_x^2
         \\
         = &\frac{\kappa^2\eta_x^2}{\mu\eta_y}\Big(
          13 + \frac{4\cdot 144\cdot 176L^4\rho^2}{(1-\rho)^4\kappa^2}\eta_y^4
          +
          \frac{16^2\cdot144L^4\rho^2}{(1-\rho)^2\kappa^2}\eta_y^4
          \Big)
          \\
         \leq & \frac{14\kappa^2}{\mu}\frac{\eta_x^2}{\eta_y}
         \end{aligned}
     \end{equation}
     where the first inequality follows from \cref{lemma:bound-alpha,eq:hat-values,eq:bound-ac-all,eq:bound-hatbz-all} and $S_2\geq q$; the last inequality follows from the condition of \xzy{$\eta_y\leq\frac{1}{32\sqrt{5}}\cdot\frac{(1-\rho)^2}{ L}$}.
     If we use all the analysis from \cref{eq:coeff-noname-1} to \cref{eq:coeff-noname-2} within  \cref{eq:long-sum-grad_Phi} and  then use \cref{eq:bound-ac-all}, i.e., $c_\Phi\geq\frac{\eta_x}{12}$, it follows that
     \begin{equation}\label{eq:long-grad-Phi-1}
         \begin{aligned}
             &(\frac{\eta_x}{12}
             -\frac{290\cdot144L^4\rho^2}{(1-\rho)^4}\eta_x^3\eta_y^2 
             - \frac{9}{\eta_y\mu}\cdot2L^2\eta_x\cdot\frac{14\kappa^2}{\mu}\frac{\eta_x^2}{\eta_y}
             )
             \sum_{j=0}^{T-1}\mathbb{E}\left[\|\grad\Phi(\bbx^j)\|^2\right]\leq \mathbb{E}\left[\Phi(\bbx^0)-\Phi(\bbx^T)\right]
             + 
             \frac{9}{\eta_y\mu}\cdot2L^2\eta_x\delta_0
             \\
             &
             +\Big(
              \frac{40L^2\eta_y^2\eta_x}{M(1-\rho)}
             + \frac{9}{\eta_y\mu} \cdot2L^2\eta_x\cdot\frac{128L\kappa\eta_y^3}{M(1-\rho)}
             \Big)C_{0,D}
             +
             \Big(
             \frac{2c_Z}{1-\rho}+\frac{9}{\eta_y\mu} \cdot2L^2\eta_x\cdot\frac{2a_Z}{1-\rho}
             \Big)C_{0,Z}
             \\
             &
        +\Big(c_Z+\frac{9}{\eta_y\mu} \cdot2L^2\eta_x \cdot a_Z\Big)\|Z^0_\perp\|^2_F
             +\Big(
             \frac{20L^2\eta_y^2\eta_x}{M}
             +\frac{9}{\eta_y\mu} \cdot2L^2\eta_x\cdot 
             \frac{64L\kappa\eta_y^3}{M}
             \Big)\|D^0_\perp\|^2_F
             \\
             & +\frac{5\sigma^2}{S_1}\eta_x T + \frac{9}{\eta_y\mu} \cdot2L^2\eta_x\cdot T\frac{16\eta_y \sigma^2}{\mu S_1}.
         \end{aligned}
     \end{equation}
Moreover, it follows from \cref{eq:bound-ac-all} and $\eta_x\leq \frac{1}{64\kappa^2}\eta_y$ and \xzy{$\eta_y\leq \frac{1}{32\sqrt{5}L}\{(1-\rho)^2,\frac{1}{\kappa}\}$} that
\begin{equation}\label{eq:long-coeff-1}
    \begin{aligned}
    & \frac{\eta_x}{12}
             -\frac{290\cdot144L^4\rho^2}{(1-\rho)^4}\eta_x^3\eta_y^2 
             - \frac{9}{\eta_y\mu}\cdot2L^2\eta_x\cdot\frac{14\kappa^2}{\mu}\frac{\eta_x^2}{\eta_y}\geq \frac{\eta_x}{50}
             \\
        &\frac{40L^2\eta_y^2\eta_x}{M(1-\rho)}
             + \frac{9}{\eta_y\mu} \cdot2L^2\eta_x\cdot\frac{128L\kappa\eta_y^3}{M(1-\rho)}
             =
            \frac{L^2\eta_y^2\eta_x}{M(1-\rho)}(40 + 18\cdot128\kappa^2)
            \leq \frac{2344L^2\kappa^2\eta_y^2\eta_x}{M(1-\rho)}, \\
       & c_Z+ \frac{9}{\eta_y\mu}\cdot2L^2\eta_x \cdot a_Z\leq \frac{105L^2\eta_x}{2M} +  \frac{9}{\eta_y\mu}\cdot{2L^2\eta_x}\cdot\frac{176L\kappa\eta_y}{M}
       =
       \frac{L^2\eta_x}{M}(\frac{105}{2} + 18\cdot176\kappa^2)
       \leq 
       \frac{6441L^2\kappa^2\eta_x}{2M}, \\
        & \frac{5\sigma^2}{S_1}\eta_x T + \frac{9}{\eta_y\mu} \cdot2L^2\eta_x\cdot T\frac{16\eta_y \sigma^2}{\mu S_1} \leq \frac{293\kappa^2\sigma^2\eta_x}{S_1}\cdot T,
    \end{aligned}
\end{equation}
If we use \cref{eq:long-coeff-1} within \cref{eq:long-grad-Phi-1}, we obtain that
     \begin{equation}\label{eq:long-grad-Phi-2}
         \begin{aligned}
             \frac{\eta_x}{50}\sum_{j=0}^{T-1}\mathbb{E}\left[\|\grad\Phi(\bbx^j)\|^2\right]\leq & \mathbb{E}\left[\Phi(\bbx^0)-\Phi(\bbx^T)\right]+\frac{18L\kappa\eta_x}{\eta_y}\delta_0 +
          \frac{6441L^2\kappa^2\eta_x}{2M}\|Z^0_\perp\|^2_F
             +\frac{2344L^2\kappa^2\eta_x\eta_y^2}{2M}\|D^0_\perp\|^2_F
             +\frac{293\kappa^2\sigma^2\eta_x}{S_1}\cdot T
             \\
             & +\frac{2344L^2\kappa^2\eta_x\eta_y^2}{M(1-\rho)}C_{0,D} +
            \frac{6441L^2\kappa^2\eta_x}{M(1-\rho)}C_{0,Z}
         \end{aligned}
     \end{equation}
Furthermore, if we use \cref{eq:bound-CZ-final,eq:bound-CD-final} within above \cref{eq:long-grad-Phi-2}, we obtain that
     \begin{equation}\label{eq:long-grad-Phi-3}
         \begin{aligned}
             \frac{\eta_x}{50}&\sum_{j=0}^{T-1}\mathbb{E}\left[\|\grad\Phi(\bbx^j)\|^2\right]\leq \mathbb{E}\left[ \Phi(\bbx^0)-\Phi(\bbx^T) \right]+\frac{18L\kappa\eta_x}{\eta_y}\delta_0
             \\
             &+
           \frac{6441L^2\kappa^2\eta_x}{2M}\|Z^0_\perp\|^2_F
             +\frac{2344L^2\kappa^2\eta_x\eta_y^2}{2M}\|D^0_\perp\|^2_F
             +\frac{293\kappa^2\sigma^2\eta_x}{S_1}\cdot T
             \\
             &+
            \frac{2344L^2\kappa^2\eta_x\eta_y^2}{M(1-\rho)}
            \Big(
            \frac{240\rho^2L^2}{(1-\rho)^2}\|Z_\perp^0\|_F^2
          +\frac{201}{200}\|D^0_\perp\|^2_F
         +  N(3q/80+16) \cdot\frac{M\sigma^2}{S_1}\frac{\rho^2}{(1-\rho)}
            \Big)
            \\
                         & +\frac{6441L^2\kappa^2\eta_x}{M(1-\rho)}
             \Big(
        \|Z^0_\perp\|^2_F
         + 
         \|D^0_\perp\|^2_F
         +
         N(3q/80+16) \cdot\frac{\eta_y^2M\sigma^2}{S_1}\frac{\rho^2}{(1-\rho)^3}\Big)
         \\
            =&
            \mathbb{E}\left[\Phi(\bbx^0)-\Phi(\bbx^T) \right]
            +\frac{18L\kappa\eta_x}{\eta_y}\delta_0
            +\frac{\kappa^2\sigma^2\eta_x}{S_1}\cdot T \cdot\Big(
            293
            + 
            \frac{2344\cdot17\rho^2L^2\eta_y^2}{(1-\rho)^2}
            +
           \frac{6441\cdot17\rho^2L^2\eta_y^2}{(1-\rho)^4}
            \Big)
             \\
             & +
          \frac{L^2\kappa^2\eta_x}{M(1-\rho)}
          \Big(
            \frac{6441(1-\rho)}{2} + \frac{2344\cdot240\rho^2L^2\eta_y^2}{(1-\rho)^2}
            + 6441
          \Big)
          \|Z^0_\perp\|^2_F
          \\
          & + 
          \frac{L^2\kappa^2\eta_x}{M(1-\rho)}\Big(\frac{2344\eta_y^2(1-\rho)}{2}
            + \frac{2344\cdot201\eta_y^2}{100}
            +
            6441
        \Big)
        \|D^0_\perp\|^2_F
        \\
            \leq & \mathbb{E}\left[\Phi(\bbx^0)-\Phi(\bbx^T)\right]
            +\frac{18L\kappa\eta_x}{\eta_y}\delta_0
             +\frac{9772L\kappa^2\eta_x}{M}\cdot(\|D^0_\perp\|^2_F+\|Z^0_\perp\|^2_F)+\frac{323\kappa^2\sigma^2\eta_x}{S_1}\cdot T,
         \end{aligned}
     \end{equation}
where the last inequality is by \xzy{$\eta_y\leq \frac{1}{32\sqrt{5}}\cdot\frac{(1-\rho)^2}{\max\{L,1\}}$}.
Therefore, we obtain \cref{eq:convergence-explicit},
which completes the proof.
\end{proof}

\xzrr{Having simplified the parameters as discussed above, we are now ready to prove the main result, as stated in \cref{thm:main-result}. For the sake of completeness, we provide the detailed version of \cref{thm:main-result}.}
\xzrr{\begin{theorem}\label{thm:main-result-detailed}
   Suppose Assumptions~\ref{ASPT:SC},~\ref{ASPT:lower-bounded-Phi},~\ref{ASPT:general-sto-grad},~\ref{ASPT:smooth-F-mean-squared} and \ref{ASPT:mixture-matrix} hold. Moreover, $\{\eta_x,\eta_y\}$ and $\{S_1,S_2,q\}$  are chosen such that
  \begin{equation}\label{eq:final-step-size}
  \begin{aligned}
        &\eta_y = \frac{1}{32\sqrt{5}L}\min\{\frac{1}{\kappa},(1-\rho)^2\}, \;\eta_x = \frac{1}{64\kappa^2}\eta_y,\; 
        \\
        & S_1 = \lfloor 100\cdot323\kappa^2\frac{\sigma^2}{\epsilon^2} \rfloor,\quad
      S_2 \geq q, \quad
      q\geq 1.
  \end{aligned}
  \end{equation}
  Then 
 $\frac{1}{T}\sum_{t=1}^{T-1}\mathbb{E}[\|\grad \Phi (\bbx^t)\|^2]\leq\epsilon^2$
  holds for all $T\geq1$ such that 
  \begin{equation}
        T\geq \max\{ \frac{1}{\eta_x}\Delta_\Phi,\frac{18L\kappa}{\eta_y}\delta_0,
     \frac{9772L^2\kappa^2}{M}\Lambda_0
     \} \frac{300}{\epsilon^2}
  \end{equation}
  and $T = N q$ for some $N \in\mathbb{N}^+$, where $\Delta_\Phi$
  and $\Lambda_0$ are defined in \cref{thm:main-convergence} and $\delta_0$ is defined in \cref{eq:def-Delta-t}.
\end{theorem}}
\begin{proof}
Indeed, we can compute that the parameter choices in \cref{eq:final-step-size} satisfy \cref{param-condition:basic-2}. Then the inequality
\begin{equation}
    \label{eq:bd-final-grad-Phi}
    \frac{1}{T}\sum_{i=1}^{T-1}\mathbb{E}[\|\grad \Phi (\bbx^t)\|^2]\leq\epsilon^2
\end{equation}
directly follows by invoking the parameters choice in \cref{thm:main-result-detailed} within \cref{eq:convergence-explicit}. 
\end{proof}

\subsection*{Bound on Dual Optimality}

Applying \cref{eq:bound-ac-all,eq:coeff-noname-2}, and \cref{lemma:bound-alpha} within \cref{eq:bound-deltat-sum-final-2} yields
\begin{equation}
    \begin{aligned}
    \bE\left[\frac{1}{9}\eta_y\mu\sum_{j=0}^{T}\delta_{j} 
    \right]
        \leq &\bE\left[ 
        \delta_{0} 
         + N\frac{16\eta_y q\sigma^2}{\mu S_1} 
         +\frac{176L\kappa\eta_y}{M}\|Z^0_\perp\|^2_F
         + \frac{64L\kappa\eta_y^3}{M}\|D_\perp^0\|^2_F \right.
         \\
         & + \frac{352L\kappa\eta_y}{(1-\rho)M}C_{0,Z}
         +\frac{128L\kappa\eta_y^3}{(1-\rho)M}C_{0,D}
         \\
         &
                  \left.+ 
      \frac{14\kappa^2}{\mu}\frac{\eta_x^2}{\eta_y}
         \sum_{j=0}^{T-1}\|\grad \Phi(\bbx^j)\|^2
         \right].
    \end{aligned}
\end{equation}
Moreover, it follows from parameter choice in \cref{thm:main-result-detailed}, and the above analysis and \cref{eq:bound-CZ-final,eq:bound-CD-final} that
\begin{equation}\label{eq:bound-bigO-CZD}
    C_{0,Z}=\cO\Big(T\cdot\frac{M\eta_y^2}{(1-\rho)^3\kappa^2}\epsilon^2
    \Big),\quad
    C_{0,D} = \cO\Big(\frac{L^2}{(1-\rho)^2} + T\cdot\frac{M}{(1-\rho)\kappa^2}\epsilon^2
    \Big),\quad
     \sum_{j=0}^{T-1}\|\grad \Phi(\bbx^j)\|^2 = \cO\Big(T\cdot\epsilon^2
     \Big).
\end{equation}
Therefore, we obtain that
\begin{equation}\label{eq:noname-3}
    \begin{aligned}
    \bE\left[\frac{1}{T}\sum_{j=0}^{T}\delta_{j} 
    \right]
        = & \cO\Big(
        \frac{1}{T}\cdot\frac{1}{\eta_y \mu} + \frac{1}{\mu^2\kappa^2}\epsilon^2 + \frac{1}{T}\cdot\frac{\kappa^2}{M}+ \frac{1}{T}\cdot\frac{\kappa^2\eta_y^2}{M}
        \\
        & +\frac{\kappa^2}{(1-\rho)M}\cdot\frac{M\eta_y^2}{(1-\rho)^3\kappa^2}\epsilon^2
        + \frac{1}{T}\cdot\frac{\kappa^2\eta_y^2}{(1-\rho)M}\cdot
        \frac{L^2}{(1-\rho)^2}
        +
        \frac{\kappa^2\eta_y^2}{(1-\rho)M}\cdot
        \frac{M}{(1-\rho)\kappa^2}\epsilon^2
        \\&
        + \frac{\kappa^2\eta_x^2}{\mu^2\eta_y^2}\epsilon^2
        \Big)
        \\
        = & \cO\Big(
        \frac{1}{T}\cdot\frac{1}{\eta_y \mu}  +\frac{1}{T}\cdot\frac{\kappa^2}{M} +\frac{1}{T}\cdot\frac{L^2\kappa^2\eta_y^2}{(1-\rho)^3M} + \frac{1}{L^2}\epsilon^2 + \frac{\eta_y^2}{(1-\rho)^4}\epsilon^2 + \frac{\kappa^2\eta_x^2}{\mu^2\eta_y^2}\epsilon^2
        \Big)
    \end{aligned}
\end{equation}
Then, without loss of generality, assuming that $L\geq1$, it follows from the choice of $T,\eta_x,\eta_y$ that
\begin{equation*}
    \frac{1}{T} = \cO\Big(\min\{1/(L\kappa^2),1/(L^2\kappa)\}\cdot\min\{1/\kappa,(1-\rho)^2\}\epsilon^2\Big),
    \quad
    \frac{1}{\eta_y} = \cO\Big(L\cdot\max\{\kappa,\frac{1}{(1-\rho)^2}\}\Big),
    \quad
    \eta_x/\eta_y = \cO\Big(1/\kappa^2\Big);
    \end{equation*}
thus, applying the above bounds within \cref{eq:noname-3} yields that
\begin{equation}\label{eq:bd-final-delta}
    \frac{1}{T}\sum_{t=0}^{T-1}\bE[\delta_t] = \cO\Big(
    \frac{1}{T}\cdot\frac{1}{\eta_y\mu} + \frac{\kappa^2\eta_x^2}{\mu^2\eta_y^2}\epsilon^2
    \Big) = \cO\Big(
    \min\{\frac{1}{L\kappa},\frac{1}{L^2}\}\epsilon^2 + \frac{1}{L^2}\epsilon^2
    \Big) 
      = \cO\left(\frac{\epsilon^2}{L^2} \right)
\end{equation}
\subsection*{Bound on Consensus Error}
\begin{proof}
Recall \eqref{eq:bound-Z-perp-final-2}, namely,
\begin{equation}
        \begin{aligned}
    \bE\Big[\sum_{j=1}^T  \| Z^j_\perp \|^2_F\Big] \
\leq 
\frac{2}{1-\rho}\Big\{\yy{C_{0,Z}} 
+ \sa{\frac{\eta_y^2}{\alpha(1-\rho)}}
    \bE\bigg[
\sum_{j=0}^{T-2}
\Big(
2\hat{b}_\delta\delta_j
+
2\hat{b}_\Phi \|\grad \Phi(\bbx^j)\|^2
\Big)
\bigg]\Big\}
\end{aligned}
\end{equation}

Also, recall \cref{eq:bound-bigO-CZD}
\begin{equation}
     \begin{aligned}
        C_{0,Z}=\cO\Big(T\cdot\frac{M\eta_y^2}{(1-\rho)^3\kappa^2}\epsilon^2
    \Big),\quad
     \sum_{j=0}^{T-1}\|\grad \Phi(\bbx^j)\|^2 = \cO\Big(T\cdot\epsilon^2\Big),\quad \sum_{j=0}^{T-1}\delta_j = \cO\Big(T\cdot\frac{1}{L^2}\epsilon^2
     \Big).
     \end{aligned}
 \end{equation}
 Moreover, by \cref{lemma:bhat-values} and \cref{lemma:bound-alpha}, we have $\hat{b}_\delta\leq\frac{120ML^4\rho^2\eta_y^2}{1-\rho}$, $\hat{b}_\Phi\leq \frac{72ML^2\rho^2\eta_x^2}{1-\rho}$, $\alpha \ge \frac{1-\rho}{2}$; thus, without loss of generality, assuming $L\geq 1$,  we further obtain
\begin{equation}\label{eq:bd-consensus-all}
\begin{aligned}
        \bE\Bigg[ \frac{1}{T} \sum_{t=0}^{T-1}  \| Z^t_\perp \|^2_F\Bigg]
    & = \cO\left(
    \frac{M\eta_y^2}{(1-\rho)^4\kappa^2}\epsilon^2 + \frac{\eta_y^2}{(1-\rho)^3}\cdot\frac{ML^4\rho^2\eta_x^2}{(1-\rho)}\cdot\frac{1}{L^2}\epsilon^2 + \frac{\eta_y^2}{(1-\rho)^3}\cdot\frac{ML^2\rho^2\eta_x^2}{(1-\rho)}\epsilon^2
    \right) \\
    & =\cO\left(\frac{M\eta_y^2}{(1-\rho)^4}\epsilon^2 \cdot(\frac{1}{\kappa^2}  + L^2\rho^2\eta_x^2 )\right) 
    \\
    & =\cO\left(\frac{M\eta_y^2}{\kappa^2(1-\rho)^4}\epsilon^2 \right)
    \\
    & = \cO\left(\frac{M}{L^2\kappa^2}\min\{\frac{1}{\kappa^2(1-\rho)^4},1\}\epsilon^2  \right)
\end{aligned}
\end{equation}
\end{proof}
\end{document}